\documentclass[11pt]{amsart} 

\usepackage{amssymb,cite,hyperref}
\usepackage{youngtab}
\usepackage{ytableau}
\usepackage[margin=1in]{geometry}
\usepackage{tensor,enumitem,verbatim}
\usepackage[all,cmtip]{xy}
\usepackage{todonotes}

\newtheorem{theorem}{Theorem}[section]
\newtheorem{lemma}[theorem]{Lemma}
\newtheorem{cor}[theorem]{Corollary}
\newtheorem{rk}[theorem]{Remark}

\newtheorem{prop}[theorem]{Proposition}

\numberwithin{equation}{section}
\numberwithin{figure}{section}

\numberwithin{equation}{section}

\theoremstyle{definition}

\newtheorem{remark}[theorem]{Remark}

\newcommand{\q}{\mathfrak{q}}

\newcommand{\Hd}{\mathcal{H}_d}
\newcommand{\two}{M\otimes N\otimes V^{\otimes d}}

\newcommand{\x}{\tilde{x}}
\newcommand{\y}{\tilde{y}}
\newcommand{\z}{\tilde{z}}
\newcommand{\hh}{\hspace{.3 in}}
\newcommand{\vv}{v_M\otimes v_N\otimes v_1 \otimes \cdots \otimes v_d}
\newcommand{\kk}{\kappa}

\title[Two Boundary Centralizer Algebras for $\mathfrak{q}(n)$]{Two Boundary Centralizer Algebras for $\mathfrak{q}(n)$}

\author{Jieru Zhu}
\address{Department of Mathematics \\
		University at Buffalo \\
		Buffalo, NY, 14260, USA}
\email{jieruzhu@buffalo.edu}

\allowdisplaybreaks

\begin{document}

\maketitle

\begin{abstract}
We define the degenerate two boundary affine Hecke-Clifford algebra $\mathcal{H}_d$, and show it admits a well-defined $\mathfrak{q}(n)$-linear action on the tensor space $M\otimes N\otimes V^{\otimes d}$, where $V$ is the natural module for $\mathfrak{q}(n)$, and $M, N$ are arbitrary modules for $\mathfrak{q}(n)$, the Lie superalgebra of Type Q. When $M$ and $N$ are irreducible highest weight modules parameterized by a staircase partition and a single row, respectively, this action factors through a quotient of $\mathcal{H}_d$. We then construct explicit modules for this quotient, $\mathcal{H}_{p,d}$, using combinatorial tools such as shifted tableaux and the Bratteli graph. These modules belong to a family of modules which we call calibrated. Using the relations in $\mathcal{H}_{p,d}$, we also classify a specific class of calibrated modules. The irreducible summands of $M\otimes N\otimes V^{\otimes d}$ coincide with the combinatorial construction, and provide a weak version of the Schur-Weyl type duality.
\end{abstract}

\section{Introduction}

In the early twentieth century, Schur studied the actions of the general linear group, $GL(V)$, and the symmetric group on a tensor space $V^{\otimes d}$. The two actions fully centralize each other. This is now known as Schur-Weyl duality and it provides a powerful link between the representation theories of these two groups. For example, the finite-dimensional irreducible representations of $GL(V)$ which occur as summands of $V^{\otimes d}$ are in bijection with the finite dimensional irreducible representations of the symmetric group, whenever $\operatorname{dim} V \geq d$. Schur-Weyl duality continues to be studied and has been generalized to many other settings. 

For example, following Schur's work, given a finite dimensional module $M$ and the natural module $V$ for the general linear Lie algebra $\mathfrak{gl}_n(\mathbb{C})$, Arakawa-Suzuki \cite{Arakawa} studied an action of the affine Hecke algebra $\mathcal{H}^{\operatorname{aff}}_d$ on $M\otimes V^{\otimes d}$ which centralizes the action of $\mathfrak{gl}_n(\mathbb{C})$. As a result, there exists a family of functors from the category of finite-dimensional $\mathfrak{gl}_n(\mathbb{C})$-modules to the category of finite-dimensional $\mathcal{H}^{\operatorname{aff}}_d$-modules.

An analogue of $\mathcal{H}^{\operatorname{aff}}_d$, called the degenerate two-boundary braid algebra $\mathcal{G}_d$, was studied by Daugherty \cite{Daugherty}. The quantum version of this algebra has a diagrammatic presentation in \cite{Ram} . In particular, given any two finite dimensional $\mathfrak{gl}_n(\mathbb{C})$-modules $M,N$, there exists a well-defined action of $\mathcal{G}_d$ on the tensor module $M\otimes N\otimes V^{\otimes d}$. As a special case, one can choose $M,N$ to be simple modules parameterized by rectangular Young diagrams. Daugherty then defined a quotient $\mathcal{H}^{\operatorname{ext}}_d$ of $\mathcal{G}_d$, under extra relations dependent on the Young diagrams. It follows that the action of $\mathcal{G}_d$ satisfies the extra relations and factors through the quotient $\mathcal{H}^{\operatorname{ext}}_d$. Moreover, Daugherty constructed irreducible $\mathcal{H}^{\operatorname{ext}}_d$-modules using combinatorial tools, and showed that irreducible direct summands of $M\otimes N\otimes V^{\otimes d}$ are isomorphic to these combinatorially constructed modules.

A super (i.e. $\mathbb{Z}_2$-graded) analogue of the above result was developed by the author in \cite{Zhu}. This was done by studying the connection between the representation theory of $\mathcal{H}_d^{ \operatorname{ext}}$ and the representation theory of the general linear Lie superalgebras $\mathfrak{gl}_{n|m}(\mathbb{C})$. Similar to the representation theory of $\mathfrak{gl}_n(\mathbb{C})$, polynomial representations of $\mathfrak{gl}_{n|m}(\mathbb{C})$ are also parameterized by combinatorial objects, and characters of polynomial representations are given by hook Schur functions.  

In particular, given finite dimensional supermodules $M,N$ for the general linear Lie superalgebra $\mathfrak{gl}_{n|m}(\mathbb{C})$, $M\otimes N\otimes V^{\otimes d}$ is naturally a module for $\mathfrak{gl}_{n|m}(\mathbb{C})$. An action of $\mathcal{G}_d$ on $M\otimes N\otimes V^{\otimes d}$ can be defined using a certain Casimir element, and this action commutes with the action of $\mathfrak{gl}_{n|m}(\mathbb{C})$. Moreover, when $M$ and $N$ are irreducible representations whose highest weights are given by rectangular Young diagrams, the defining relations for $\mathcal{H}_d^{\operatorname{ext}}$, as a quotient of $\mathcal{G}_d$, are also satisfied, and this induces a further action of $\mathcal{H}_d^{\operatorname{ext}}$. We therefore recover the irreducible  $\mathcal{H}_d^{\operatorname{ext}}$-modules in \cite{Daugherty} by studying irreducible summands of $M\otimes N\otimes V^{\otimes d}$ as $\mathcal{H}_d^{\operatorname{ext}}$-modules, and they coincide with Daugherty's combinatorial construction in \cite{Daugherty}.

On the other hand, Hill-Kujawa-Sussan \cite{HKS} studied the Type Q version of the construction in Arakawa-Suzuki \cite{Arakawa}. In particular, the Type Q analogue of the affine Hecke algebra is the affine Hecke-Clifford algebra $H_d$, whose underlying vector space is the tensor product between $\mathcal{H}^{\operatorname{aff}}$ and the Clifford algebra. Given any finite dimensional module $M$ for the Type Q Lie superalgebra $\q(n)$, there is a well-defined action of  $H_d$ on $M\otimes V^{\otimes d}$, which commutes with the $\q(n)$-action. Similar to Arakawa-Suzuki\cite{Arakawa}, there is a family of functors from the category of finite-dimensional $\q(n)$-modules to the category of finite-dimensional $H_d$-modules. In addition, Hill-Kujawa-Sussan constructed $H_d$-modules using combinatorial tools such as shifted Young tableaux.

This article is a generalization of Hill-Kujawa-Sussan \cite{HKS} to the two boundary setting, or alternatively, a generalization of Daugherty \cite{Daugherty} and Zhu \cite{Zhu} to Type Q. In Section~\ref{2.1}, we first define the degenerate two boundary affine Hecke-Clifford algebra $\mathcal{H}_d$ in the spirit of the two boundary Hecke algebra $\mathcal{G}_d$. Similar to the case of $\mathfrak{gl}_n(\mathbb{C})$ and $\mathfrak{gl}_{n|m}(\mathbb{C})$, there is an odd Casimir element $\Omega\in \q(n)\otimes \q(n)$ which induces a $\q(n)$-linear action on $M\otimes N$ for any $\q(n)$-modules $M$ and $N$. We have our first main result in Theorem \ref{hacts}:

\begin{theorem}
Let $d\geq 0 $ and $M,N$ be arbitrary modules for $\q(n)$. There is a well-defined algebra homomorphism $\mathcal{H}_d \to \operatorname{End}_{\q(n)}(M\otimes N\otimes V^{\otimes d})$.
\end{theorem}

We then focus our study on the case when $M$ and $N$ are polynomial modules. By definition these are modules which occur as direct summands of $V^{\otimes e}$ for some $e\geq 0$. Polynomial modules are semisimple as explained in Section~\ref{3.1}, and they are closed under tensor product. According to \cite{Wang}, irreducible polynomial modules of $\q(n)$ are parameterized by strict partitions defined in Section~\ref{3.1}, and their tensor product decomposes into polynomial representations in a fashion controlled by combinatorics of the Schur P-function \cite{Stembridge}. We further choose $M$ and $N$ so that the irreducible summands occurring in the decomposition of $M\otimes N$ have the smallest possible multiplicity. Based on a combinatorial result developed by Bessenrodt \cite{Bessenrodt}, this is true when $M=L(\alpha)$ is parameterized by a staircase shape $\alpha$ and $N=L(\beta)$ is parameterized by a single row partition $\beta$. In this case, any irreducible summand occurs with multiplicity $2$ (see Section \ref{3.1}). A similar result is true for $W\otimes V$ when $W$ is any simple polynomial module and $N$ is the natural module. These multiplicity results allow us to describe summands of $M\otimes N\otimes V^{\otimes d}$ using the Bratteli graph introduced in Section \ref{4.1}.

A special phenomenon exists in the theory of $\mathfrak{q}(n)$-modules, called Super Schur's Lemma, in which the endomorphism ring of a simple module can sometimes be two-dimensional, spanned by the identity map and an odd endomorphism. Because of this, we take slightly larger algebras $\mathcal{H}^{{+}}_{d}$ and $\mathcal{H}^{{++}}_{d}$ that are tensor products between $\mathcal{H}_d$ and the Clifford algebras on one or two generators. If $\beta$ has $p$ boxes, we define $\mathcal{H}^{\operatorname{ev}}_{d,p}$ as a quotient of $\mathcal{H}^{+}_d$, and $\mathcal{H}^{\operatorname{od}}_{d,p}$ as a quotient of $\mathcal{H}^{++}_d$, under extra relations involving $n$ and $p$. A key formula in Proposition \ref{Quo_2_key} establishes a connection between $\Omega$ and certain even central elements in $\q(n)$, whose actions on a simple polynomial module can be easily calculated. Using this formula, we show the above action factors through the quotient, leading to the following result in Theorem~\ref{relations}.

\begin{theorem}\label{Int_hf}
There is a well-defined action
\begin{align*}
\rho:\mathcal{H}^{\operatorname{ev}}_{d,p} \to \operatorname{End}_{\q(n)}(L(\alpha)\otimes L(\beta) \otimes V^{\otimes d}).
\end{align*}
Further, if $n$ is odd, there is a well-defined action
\begin{align*}
\rho:\mathcal{H}^{\operatorname{od}}_{d,p} \to \operatorname{End}_{\q(n)}(L(\alpha)\otimes L(\beta) \otimes V^{\otimes d}).
\end{align*}
\end{theorem}

Similar to Daugherty and Hill-Kujawa-Sussan, we also construct $\mathcal{H}^{\operatorname{ev}}_{d,p}$- and $\mathcal{H}^{\operatorname{od}}_{d,p}$-modules  in Section~\ref{4.2} using combinatorial tools. In Lemma~\ref{Pre_5_com}, we reformulated the decomposition formula for Schur P-fuctions given by Stembridge \cite{Stembridge}, and define a Bratteli graph $\Gamma_{\alpha,\beta}$ whose vertices are strict partitions and directed edges are defined representation theoretically. Let $\lambda$ be a fixed vertex in this graph, and denote $\Gamma^{\lambda}$ to be the set of all paths from $\alpha$ to $\lambda$. Let $f: \Gamma^{\lambda}\to \mathbb{C}^{\times}$ be any function satisfying the condition $(\ref{Cal_2_f})$, such as the one in Lemma~\ref{preparation}. We define vector spaces $\mathcal{D}^{\lambda}_{f}$ and $\mathcal{E}^{\lambda}_{f}$, whose bases are given in terms of paths in the Bratteli graph and the generators of $\mathcal{H}^{\operatorname{ev}}_{d,p}$ and $\mathcal{H}^{\operatorname{od}}_{d,p}$ acts by explicit formulas. In particular, the graph and these formulas depend on $\alpha$ and $\beta$, which are determined by the parameters $n$ and $p$ in the defining relations of $\mathcal{H}^{\operatorname{ev}}_{d,p}$ and $\mathcal{H}^{\operatorname{od}}_{d,p}$. We have the following result in Theorem~\ref{relationshold} and Theorem~\ref{irreducibility}.

\begin{theorem}
The vector space $\mathcal{D}^{\lambda}_f$ admits a well-defined action of $\mathcal{H}^{\operatorname{ev}}_{p,d}$, and is irreducible. Similarly, the vector space $\mathcal{E}^{\lambda}_f$ admits a well-defined action of $\mathcal{H}^{\operatorname{od}}_{p,d}$, and is irreducible.
\end{theorem}

Following the spirit of Hill-Kujawa-Sussan \cite{HKS} and Wan \cite{Wan}, we study $\mathcal{H}^{\operatorname{ev}}_{p,d}$- and $\mathcal{H}^{\operatorname{od}}_{p,d}$-modules on which certain polynomial generators act by eigenvalues. We then classify a specific family of these modules, whose eigenvalues are assumed to be given by the combinatorial data in the Bratteli graph. In particular, for each path $T$ and integer $i$, the eigenvalue $\kappa_T(i)$ is defined in Section~\ref{4.2}. In Theorem~\ref{classify} and Theorem~\ref{oddclassify}, we show that these eigenvalues determine the module structure of $\mathcal{D}^{\lambda}_f$ and $\mathcal{E}^{\lambda}_f$ defined above. By definition the algebra $\mathcal{H}^{\operatorname{ev}}_{p,d}$ contains the Clifford algebra on $d+1$ generators as a subalgebra. 

\begin{theorem}Given an $\mathcal{H}^{\operatorname{ev}}_{p,d}$-module $\mathcal{W}^{\lambda}$, if it is free over $\operatorname{Cl}_{d+1}$ with basis $\{v_T\}_{T\in \operatorname{\Gamma^{\lambda}}}$, where each $v_T$ is homogeneous  and $z_i.v_T=\kappa_T(i)v_T$, $0\leq i\leq d$, then $\mathcal{W}^{\lambda} \simeq \mathcal{D}^{\lambda}_f $ for some function $f$. A similar result holds for $\mathcal{H}^{\operatorname{od}}_{p,d}$ and the module $\mathcal{E}^{\lambda}_f$, if additionally, the module $\mathcal{W}^{\lambda}$ admits an odd $\mathcal{H}^{\operatorname{od}}_{p,d}$-endomorphism.
\end{theorem}

The module $L(\alpha)\otimes L(\beta) \otimes V^{\otimes d}$ can be regarded as a bimodule for $\q(n)$ and the centralizer $\mathcal{Z}_d=\operatorname{End}_{\q(n)}(L(\alpha)\otimes L(\beta)\otimes V^{\otimes d})$. Similar to Schur-Weyl duality, a $\mathfrak{q}(n)$-version of the double centralizer theorem describes properties of its irreducible summands as a bimodule. Since the image $\rho(\mathcal{H}^{\operatorname{ev}}_{p,d})$ or $\rho(\mathcal{H}^{\operatorname{od}}_{p,d})$ in Theorem~\ref{Int_hf} is a subalgebra of $\mathcal{Z}_d$, if $\mathcal{L}^{\lambda}$ is an irreducible $\mathcal{Z}_d$-summand of $L(\alpha)\otimes L(\beta) \otimes V^{\otimes d}$, its restriction $\operatorname{Res}^{\mathcal{Z}_d}_{\rho(\mathcal{H}^{\operatorname{ev}}_{p,d})}\mathcal{L}^{\lambda}$  is a module for $\rho(\mathcal{H}^{\operatorname{ev}}_{p,d})$ and further for $\mathcal{H}^{\operatorname{ev}}_{p,d}$. Similarly, $\operatorname{Res}^{\mathcal{Z}_d}_{\rho(\mathcal{H}^{\operatorname{od}}_{p,d})}\mathcal{L}^{\lambda}$ can be regarded as a module for $\mathcal{H}^{\operatorname{od}}_{p,d}$.  In Theorem~\ref{punchline}, we show that we can recover the combinatorially constructed irreducible modules $\mathcal{D}^{\lambda}_f$ via this restriction:

\begin{theorem}
When $n$, the number of nonzero rows in $\alpha$ is even, there exists an isomorphism $\operatorname{Res}^{\mathcal{Z}_d}_{\rho(\mathcal{H}^{\operatorname{ev}}_{p,d})}\mathcal{L}^{\lambda}\simeq \mathcal{D}^{\lambda}_f$ for some choice of $f$. When $n$ is odd, then there exists an isomorphism $\operatorname{Res}^{\mathcal{Z}_d}_{\rho(\mathcal{H}^{\operatorname{od}}_{p,d})}\mathcal{L}^{\lambda}\simeq \mathcal{E}^{\lambda}_f$ for some choice of $f$.
\end{theorem}

\subsection*{Acknowledgement} The author thanks her PhD advisor, Jonathan Kujawa, for providing the main idea of this project and for overseeing its process. The author also thanks the referee for their previous time reading this article and for the valuable improvements.

\section{Centralizing Actions}
\subsection{The Lie superalgebra $\mathfrak{q}(n)$}\label{1.1}
We first introduce some basics. The Lie superalgebra $\q=\mathfrak{q}(n)$ is the $\mathbb{C}$-vector space 
\begin{align*}
\q(n)=\left\{ \begin{bmatrix}
A & B \\ B& A
\end{bmatrix}  \hspace{.1 in}| \hspace{.1 in} A,B \in \operatorname{Mat}_{n,n}(\mathbb{C})  \right\},
\end{align*}
with a $\mathbb{Z}_2$-grading $\q(n)=\q_{\overline{0}}\oplus \q_{\overline{1}}$ where
\begin{align*}
\q_{\overline{0}}=&\left\{ \begin{bmatrix}
A & 0 \\ 0& A
\end{bmatrix}  \hspace{.1 in}| \hspace{.1 in} A\in \operatorname{Mat}_{n,n}(\mathbb{C})  \right\},\\
\q_{\overline{1}}=&\left\{ \begin{bmatrix}
0 & B \\ B& 0
\end{bmatrix}  \hspace{.1 in}| \hspace{.1 in} B\in \operatorname{Mat}_{n,n}(\mathbb{C})  \right\}.
\end{align*}
An element $x\in \q$ is said to be homogeneous of degree $\overline{i}$ if $x\in \q_{\overline{i}}$, $i=0,1$, and we denote by $\overline{x}$ the degree of $x$. For homogeneous $x,y\in \q$, the Lie superbracket is defined as
\begin{align}
[x,y]=xy-(-1)^{\overline{x}\cdot \overline{y}} yx \label{bracket},
\end{align}
using matrix multiplication and extended linearly to all elements in $\q$. This sign convention will also apply in similar contexts, where an extra negative sign is introduced whenever odd elements commute past each other. It is straightforward to check that this superbracket satisfies the axioms for Lie superalgebras:
\begin{align*}
&[x,y]+(-1)^{\overline{x}\cdot \overline{y}}[y,x]=0,\\
&(-1)^{\overline{y}\cdot \overline{z}}[x,[y,z]]+(-1)^{\overline{z}\cdot \overline{x}}[y,[z,x]]+(-1)^{\overline{x}\cdot \overline{y}}[z,[x,y]]=0.
\end{align*}

For two Lie superalgebras $\mathfrak{g}$ and $\mathfrak{h}$, an \emph{even Lie superalgebra homomorphism} $\phi:\mathfrak{g} \to \mathfrak{h}$ is a map such that $[\phi(x),\phi(y)]=\phi([x,y])$, $\phi(\mathfrak{g}_{\overline{0}})\subset \phi(\mathfrak{h}_{\overline{0}})$ and $\phi(\mathfrak{g}_{\overline{1}})\subset \phi(\mathfrak{h}_{\overline{1}})$. A \emph{supermodule} $W$ for $\q(n)$ is a $\mathbb{Z}_2$-graded $\mathbb{C}$-vector space with an even homomorphism of Lie superalgebras $\q(n) \to \operatorname{End}(W)$. Here, $\operatorname{End}(W)$ is equipped with the superbracket defined similar to $(\ref{bracket})$ using composition of maps, and an element $f\in \operatorname{End}(W)$ is even if and only if $f$ preserves the grading in $W$, and is odd if and only if $f$ reverses the grading. For this article, we will always be discussing supermodules and refer to them as just modules.

Let $V=\mathbb{C}^{2n}$ be the set of column vectors of height $2n$. For $1\leq i\leq n$, Let $e_i$ be the column vector with $1$ on the $i$-th entry and $0$ everywhere else, and $f_i$ be the column vector with $1$ on the $n+i$-th entry and $0$ everywhere else.  Impose a $\mathbb{Z}_2$-grading on $V$ with $V_{\overline{0}}=\langle e_1,\dots, e_n\rangle$ and $V_{\overline{1}}= \langle f_1,\dots, f_n\rangle$. Let $\q$ acts on $V$ by matrix multiplication on the left. It is straightforward to check that this defines a $\q(n)$-module structure. For $\q(n)$-modules $M$ and $N$, define the module structure on $M\otimes N$ as follows: on $ v\otimes w\in M\otimes N $, a homogeneous element $x\in \q(n)$ acts as
\begin{align}\label{tensor}
x.(v\otimes w) = (x.v) \otimes w +(-1)^{\overline{x}\cdot\overline{v}} v\otimes (x.w).
\end{align}
We extend this definition to a tensor product of multiple modules, such as $M\otimes N\otimes W$, by taking successive tensor products, and it is straightforward to check $M\otimes (N\otimes W)\simeq (M\otimes N)\otimes W$ under the obvious map. Hence for arbitrary $\q(n)$-modules $M$ and $N$, the tensor product $M\otimes N\otimes V^{\otimes d}$ is again a $\q(n)$-module.

\subsection{The two boundary affine Hecke-Clifford algebra}\label{2.1} The goal of this section is to define another superalgebra, which acts on $M\otimes N\otimes V^{\otimes d}$, and the action supercommutes with that of $\q(n)$. This algebra is a natural generalization of the Sergeev algebra. First, a \emph{superalgebra} is an associative algebra $A$ with a $\mathbb{Z}_2$-grading $A=A_{\overline{0}}\oplus A_{\overline{1}}$, such that $A_{\overline{i}}A_{\overline{j}}\subset A_{\overline{i+j}}$ for $i,j\in\{0,1\}$. Let $d\in \mathbb{Z}_{\geq 0}$, the Clifford algebra $\operatorname{Cl}_d$ is a superalgebra generated by odd $c_1,\dots,c_d$, subject to the relations
\begin{align}\label{clifford}
c_i^2=-1, \hspace{.5 in} c_ic_j=-c_jc_i \hspace{.1 in}(i\neq j).
\end{align}
The Sergeev algebra $\operatorname{Ser}_d$ is generated by odd elements $c_1,\dots,c_d$, even elements $s_1,\dots,s_{d-1}$, under the $\operatorname{Cl}_d$-relations and 
\begin{align}
s_i^2=&1, \hspace{.3 in} s_is_j=s_js_i \hspace{.1 in}(|i-j|>1),\hspace{.2 in} s_is_{i+1}s_i=s_{i+1}s_is_{i+1} \hspace{.1 in} (1\leq i\leq d-2) \label{symmetricgrouprelations}\\
s_ic_j=&c_js_i \hspace{.1 in}(j\neq i,i+1), \hspace{.2 in} s_ic_i=c_{i+1}s_i, \hspace{.2 in} s_ic_{i+1}=c_is_i. \label{cliffordrelations}
\end{align}
Notice $s_1,\dots,s_{d-1}$ generate the group algebra for symmetric group $S_d$. Recall that $V$ is the natural representation for $\q(n)$. In \cite{Sergeev}, Sergeev defined an action of $\operatorname{Ser}_d$ on $V^{\otimes d}$, where $s_i$ acts by signed permutation, using the sign convention introduced in the previous section:
\begin{align}
s_i.(v_1\otimes \cdots\otimes v_i\otimes v_{i+1}\otimes \cdots \otimes v_d)=(-1)^{\overline{v_i}\cdot\overline{v_{i+1}}}v_1\otimes \cdots \otimes v_{i+1}\otimes v_i\otimes \cdots \otimes v_d,\label{sergeev1}
\end{align}
and $c_i$ acts by 
\begin{align}
c_i (v_1\otimes \cdots\otimes v_i \otimes  \cdots \otimes v_d)=(-1)^{\overline{v_1}+\cdots+\overline{v_{i-1}}} v_1\otimes \cdots\otimes C. v_i \otimes \cdots \otimes v_d.\label{sergeev2}
\end{align}
Here, $C\in \operatorname{End}(V)$ is left multipllication by the matrix
\begin{align}
C= \begin{bmatrix}
0 & -I_n \\ I_n & 0
\end{bmatrix}, \label{Cdefinition}
\end{align}
and $I_n$ is the $n\times n$ identity matrix. This action is known to \emph{supercommute} with $\q(n)$, in the following sense: for any $y\in \operatorname{Ser}_d$, $x\in \mathfrak{q}(n)$ and $v\in V^{\otimes d}$, $y.(x.v)=(-1)^{\overline{x}\cdot \overline{y}}x.(y.v)$. In general, for $\q(n)$-modules $W$ and $U$, a homogeneous linear map $f:W\to U$ is said to supercommute with $\q(n)$, if $f(x.w)=(-1)^{\overline{x}\cdot \overline{f}}x.f(w)$ for all homogeneous $x\in \q(n)$ and $w\in W$. Denote by $\operatorname{End}_{\q(n)}(W)$ the space spanned by all homogeneous endomorphisms of $W$ that supercommute with $\q(n)$.

Generalizing this result, Nazarov \cite{Naz} defined the following affine Hecke-Clifford algebra $H_d$: it is the superalgebra generated by $\operatorname{Ser}_d$ and even $x_1,\dots,x_d$, under further relations.
\begin{align}
x_i x_j =& x_j x_i,  \hspace{.1 in}(1\leq i,j\leq d) \hspace{.2 in} s_ix_i = x_{i+1}s_i-1+c_ic_{i+1}, \label{affinecliffordtwist0}\\
c_iz_j=&  x_jc_i \hspace{.1 in} (i\neq j) \hspace{.2 in}
c_iz_i= -x_ic_i. \label{affinecliffordtwist}
\end{align}
To define the action of $H_d$ on $V^{\otimes d}$, Hill-Kujawa-Sussan introduced an even Casimir tensor element. Let $E_{ij}$ be the elementary $n\times n$ matrix with $1$ in the $(i,j)$-position and zero elsewhere, and let
\begin{align}\label{ef}
e_{ij}=\begin{bmatrix}
E_{ij} & 0 \\ 0 & E_{ij}
\end{bmatrix},\hspace{.4 in} f_{ij}=\begin{bmatrix}
0 & E_{ij} \\  E_{ij} & 0
\end{bmatrix}.
\end{align}
The even Casimir element is 
\begin{align}\label{evenomega}
\overline{\Omega}=\sum_{1\leq i,j\leq n} e_{ij} \otimes f_{ji}C -\sum_{1\leq i,j\leq n} f_{ij} \otimes e_{ji}C \in \q(n) \otimes \operatorname{End}(V).
\end{align}
In \cite[Theorem 7.4.1]{HKS}, Hill-Kujawa-Sussan showed that for an arbitrary $\q(n)$-module $M$, there is a well defined action of $H_d$ on $M\otimes V^{\otimes d}$, on which the Sergeev algebra $\operatorname{Ser}_d$ acts as (\ref{sergeev1}) and (\ref{sergeev2}), and $z_k$ acts on $M\otimes  V^{\otimes d}$ as follows, using the element $\overline{\Omega}$:
\begin{align}
&x_k. (v_M\otimes v_1\otimes \cdots\otimes v_k \otimes \cdots\otimes v_d) \label{omega}\\
=&\sum_{1\leq i,j\leq n} (-1)^{\overline{v_M+v_1+\cdots+v_{k-1}}}e_{ij}.v_M\otimes v_1\otimes \cdots\otimes f_{ji}.C(v_k) \otimes \cdots\otimes v_d\\
-&\sum_{1\leq i,j\leq n} f_{ij}.v_M\otimes v_1\otimes \cdots\otimes e_{ji}.C(v_k) \otimes \cdots\otimes v_d.
\end{align}
For short, we will also denote this action as $\overline{\Omega}_{M,k}$, because the two tensor factors in $\overline{\Omega}$ act on the module $M$ and the $k$-th copy of $V$, respectively. We extend this notation to actions denoted as $X_{W,k}$ or $X_{\ell,k}$, for any $X \in \q\otimes \q$, any $\q(n)$-module $W$ (for example, $W=M\otimes N$ or $W=M\otimes N\otimes V^{\otimes i}$) and any integer $1\leq \ell,k\leq d$. To be precise, when $X=x\otimes y$, and $\Delta(x)=\sum_x x^{(1)}\otimes x^{(2)}$ under Sweedler's notation \cite{Swee},
\begin{align}
&X_{M\otimes N, k}. (v_M\otimes v_1\otimes \cdots\otimes v_k \otimes \cdots\otimes v_d) \label{omegadef1}\\
=& \sum_x (-1)^{\overline{v_M}\cdot \overline{x^{(2)}}+\overline{y}\cdot (\overline{v_M}+\cdots+\overline{v_{k-1}})} x^{(1)}.v_M \otimes x^{(2)} v_N \otimes v_1\otimes \cdots\otimes y.v_k \otimes \cdots\otimes v_d. \notag
\end{align}
And extend this linearly when $X$ is a linear combination of pure tensors. One can also infer the formula when $W=M\otimes N\otimes V^{\otimes i}$, by acting $x$ on $M\otimes N\otimes V^{\otimes i}$ using the $(i+1)$-fold coproduct to obtain $i+2$ tensor factors, then applying the result component-wise, and $y$ on the $k$-th copy of $V$. Similarly, $X_{\ell, k}$ acts on the $\ell$-th and $k$-th copies of $V$ simultaneously:
\begin{align}
&X_{\ell, k}. (v_M\otimes v_1\otimes \cdots\otimes v_{\ell}\otimes \cdots\otimes v_k \otimes \cdots\otimes v_d)\label{omegadef2}\\
=&(-1)^{\overline{x}\cdot(\overline{v_M}+\cdots\overline{v_{\ell-1}}+\overline{y}\cdot(\overline{v_M}+\cdots+\cdots \overline{v_{k-1}})} v_M\otimes v_1\otimes \cdots\otimes  x.v_{\ell}\otimes \cdots\otimes y.v_k \otimes \cdots\otimes v_d. \notag
\end{align}

The following superalgebra is a generalization of $H_d$ to the setting of $M\otimes N\otimes V^{\otimes d}$. This was motivated by an analogous construction in the $\mathfrak{gl}(n)$ setting: the usual Schur-Weyl duality on the tensor representation was extended to the one boundary setting  $M\otimes V^{\otimes d}$ by Arakawa-Suzuki \cite{Arakawa}, and then further extended to the two boundary setting $M\otimes N\otimes V^{\otimes d}$ by Daugherty \cite{Daugherty}. The following superalgebra is motivated by the extended degenerate two boundary affine Hecke algebra in \cite{Daugherty}.  We define the two boundary degenerate affine Hecke-Clifford algebra $\mathcal{H}_d$, to be the superalgebra generated by the even generators $s_1,\dots,s_{d-1}$, odd generators $c_1,\dots,c_d$, odd generators $\tilde{x}_1,\tilde{z}_0,\dots,\tilde{z}_d$, under the Sergeev relations among $s_i$ and $c_i$, together with further relations (\ref{Hecke1}) through (\ref{lastindefinition}):

\noindent(Hecke relations)
\begin{align}
s_i \tilde{z}_i &= \tilde{z}_{i+1}s_i+c_i-c_{i+1}  &(1\leq i\leq d-1),\label{Hecke1}\\
\tilde{x}_1 s_i &= s_i \tilde{x}_1 &(2\leq i\leq n),\\
\tilde{z}_j s_i &= s_i \tilde{z}_j &(j \neq i,i+1, \hspace{.1 in} 0\leq j\leq d).
\end{align}
(Clifford twist relations)
\begin{align}
c_i\tilde{x}_1 &= - \tilde{x}_1 c_i &(1\leq i\leq d), \label{Ctwist1} \\
c_i\tilde{z}_j &= - \tilde{z}_j c_i &(1\leq i \leq d, \hspace{.1 in}0\leq j\leq d). \label{Ctwist2} 
\end{align}
(Polynomial relations)
\begin{align}
\tilde{x}_1(s_1\tilde{x}_1s_1-(c_1-c_2)s_1 ) &=-(s_1\tilde{x}_1s_1-(c_1-c_2)s_1 )\tilde{x}_1,  \label{x1x2}\\
\tilde{z}_1 \tilde{z}_2 &=- \tilde{z}_2\tilde{z}_1 , \hspace{.2 in} \z_0\z_1=-\z_1\z_0.
\end{align}
(Relations between polynomial rings)
\begin{align}
\tilde{z}_2 \tilde{x}_1 &= - \tilde{x}_1 \tilde{z}_2, \\
(\tilde{z}_0-\tilde{z}_1+\tilde{x}_1)\tilde{x}_1 &= -\tilde{x}_1 (\tilde{z}_0-\tilde{z}_1+\tilde{x}_1). \label{lastindefinition}
\end{align}

It is worth pointing out that, using $\tilde{x}_1$ and the simple transpositions $s_i$, one can generate a polynomial ring in $d$ variables: for $1\leq i \leq d$, define element $\tilde{x}_i$ recursively by 
\begin{align}
\tilde{x}_{i+1}=s_i\tilde{x}_is_i-(c_i-c_{i+1})s_i.\label{Hecke2}
\end{align}
Also, define additional variables $\tilde{y}_i$, such that $\tilde{y}_1=\tilde{z}_1-\tilde{x}_1$, and 
\begin{align}
\tilde{y}_{i+1}=s_i\tilde{y}_is_i-(c_i-c_{i+1})s_i, \label{Hecke3}
\end{align}
then the following proposition explores further relations among these newly defined elements, and draws an analogy between $\mathcal{H}_d$ and the extended degenerate two boundary affine Hecke algebra in Daugherty's construction \cite{Daugherty}.

\begin{prop}\label{extragenerators}
The following relations hold in $\mathcal{H}_d$. Assume $1\leq i,j\leq d$ for all indices if not specified otherwise.

\noindent (Hecke relations)
\begin{align*}
&\tilde{x}_i s_j = s_j \tilde{x}_i , \hspace{.2 in}
\tilde{y}_i s_j = s_j \tilde{y}_i  & (i\neq j,j+1).
\end{align*}
(Clifford twist relations)
\begin{align*}
c_j\tilde{x}_i = - \tilde{x}_i c_j,  \hspace{.2 in} 
c_j\tilde{y}_i = - \tilde{y}_i c_j  \hspace{.5 in}(1\leq i,j\leq d).
\end{align*}
(Polynomial relations) 
\begin{align*}
\tilde{x}_i\tilde{x}_j =-\tilde{x}_j \tilde{x}_i,  \hspace{.2 in}
\tilde{y}_i\tilde{y}_j =-\tilde{y}_j \tilde{y}_i \hspace{.1 in}  (i\neq j,  \hspace{.05 in} 1\leq i,j\leq d)  \hspace{.3 in}  \tilde{z}_i\tilde{z}_j =-\tilde{z}_j \tilde{z}_i  \hspace{.1 in} (i\neq j,  \hspace{.05 in} 0\leq i,j\leq d). 
\end{align*}
(Relations between polynomial rings)
\begin{align}
\tilde{x}_i\tilde{z}_j&=-\tilde{z}_j\tilde{x}_i, \hspace{.1 in}  \tilde{y}_i\tilde{z}_j=-\tilde{z}_j\tilde{y}_i& ( i<j ), \label{xz}\\
(\tilde{z}_0+\tilde{z}_1-\tilde{y}_1)\tilde{y}_1 &= -\tilde{y}_1 (\tilde{z}_0+\tilde{z}_1-\tilde{y}_1), & \label{yz}\\
\tilde{x}_i+\tilde{y}_i&=\tilde{z}_i - \displaystyle\sum_{1\leq j\leq i-1}(c_j-c_i)t_{j,i}. &\label{jm}
\end{align}
Here, $t_{j,i}=s_j s_{j+1} \cdots s_{i-2} s_{i-1} s_{i-2}\cdots s_{j+1}s_j$ corresponds to the symmetric group element $(ij)$ in cycle notation.
\end{prop}

\begin{proof}
We check each set of relations in order.

(Hecke relations) Let us induct on $i$. The base case when $i=1$ is a defining relation in $\mathcal{H}_d$. Notice by (\ref{Hecke2}) we have 
\begin{align*}
\tilde{x}_{k+1}=s_k\tilde{x}_ks_k-(c_k-c_{k+1})s_k.
\end{align*}
Suppose the statement is true for all $i\leq k$, $j\neq i, i-1$.
When $i=k+1$, $i\neq j,j+1,j+2$ (or equivalently, $k \neq j-1,j,j+1$),
\begin{align*}
\hh \x_{k+1}s_j - s_j \x_{k+1}
&=(s_k\x_ks_k-(c_k-c_{k+1})s_k)s_j -s_j(s_k\x_ks_k-(c_k-c_{k+1})s_k)\\
&= s_k\x_ks_js_k-s_ks_j\x_ks_k-(c_k-c_{k+1})s_ks_j-s_j(c_k-c_{k+1})s_k\\
&=-(c_k-c_{k+1})s_ks_j+(c_k-c_{k+1})s_js_k=0.
\end{align*}
When $i=k+1=j+2$, 
\begin{align*}
\hh \x_{k+1}s_j - s_j \x_{k+1}
&=(s_k\x_ks_k-(c_k-c_{k+1})s_k)s_{k-1} -s_{k-1}(s_k\x_ks_k-(c_k-c_{k+1})s_k)\\
&= (s_k(s_{k-1}\x_{k-1}s_{k-1}-(c_{k-1}-c_k)s_{k-1})s_k-(c_k-c_{k+1})s_k)s_{k-1}\\
&\hh  -s_{k-1}(s_k(s_{k-1}\x_{k-1}s_{k-1}-(c_{k-1}-c_k)s_{k-1})s_k-(c_k-c_{k+1})s_k)\\
&=s_ks_{k-1}\x_{k-1}s_ks_{k-1}s_k-s_k(c_{k-1}-c_k)s_{k-1}s_ks_{k-1}-(c_k-c_{k+1})s_ks_{k-1}\\
&\hh -s_{k}s_{k-1}s_k\x_{k-1}s_{k-1}s_k+s_{k-1}s_k(c_{k-1}-c_k)s_{k-1}s_k+s_{k-1}(c_k-c_{k+1})s_k\\
&=-s_k(c_{k-1}-c_k)s_ks_{k-1}s_k-(c_k-c_{k+1})s_ks_{k-1}\\
&\hh +s_{k-1}(c_{k-1}-c_{k+1})s_ks_{k-1}s_k+(c_{k-1}-c_{k+1})s_{k-1}s_k\\
&=-(c_{k-1}-c_{k+1})s_{k-1}s_k-(c_k-c_{k+1})s_ks_{k-1}\\
&\hh +s_{k-1}(c_{k-1}-c_{k+1})s_{k-1}s_ks_{k-1}+(c_{k-1}-c_{k+1})s_{k-1}s_k\\
&=-(c_k-c_{k+1})s_ks_{k-1}+(c_k-c_{k+1})s_ks_{k-1}=0.
\end{align*}
For the $\tilde{y}_j$ version, $(\tilde{z}_1-\tilde{x}_1)s_i=s_i(\tilde{z}_1-\tilde{x}_1)$ for $2\leq i\leq d$, which provides the base case $\tilde{y}_1s_i=s_i\tilde{y}_1$. The induction argument is exactly same as above.

(Clifford twist relations) Again induct on the index for $\tilde{x}$. The base case $\tilde{x}_1c_1=-c_1\tilde{x}_1$ is included as one of the original relations in $\mathcal{H}_d$. Assume that the statement is true for $i\leq k$.
\begin{align*}
 \x_{k+1}c_j+c_j\x_{k+1}&=(s_k\x_ks_k-(c_k-c_{k+1})s_k)c_j+c_j(s_k\x_ks_k-(c_k-c_{k+1})s_k)\\
&=-(c_k-c_{k+1})s_kc_j-c_j(c_k-c_{k+1})s_k.
\end{align*} 
The above quantity is $0$ if $j\neq k,k+1$. When $j=k$,
\begin{align*}
 \x_{k+1}c_j+c_j\x_{k+1}&=-(c_k-c_{k+1})s_kc_k-c_k(c_k-c_{k+1})s_k\\
&=-(c_k-c_{k+1})c_{k+1}s_k-c_k(c_k-c_{k+1})s_k=0.
\end{align*}
Similarly when $j=k+1$,
\begin{align*}
 \x_{k+1}c_j+c_j\x_{k+1}&=-(c_k-c_{k+1})s_kc_{k+1}-c_{k+1}(c_k-c_{k+1})s_k\\
&=-(c_k-c_{k+1})c_{k}s_k-c_{k+1}(c_k-c_{k+1})s_k=0.
\end{align*}
The base case for $\tilde{y}_i$ follows from $(\tilde{z}_1-\tilde{x}_1)c_j=-c_j(\tilde{z}_1-\tilde{x}_1)$ and a similar induction argument proves the relations for $\tilde{y}_i$.

(Polynomial relations) The polynomial relations for $\tilde{x}_i$ are shown as follows. Relation (\ref{x1x2}) is equivalent to $\tilde{x}_1\tilde{x}_2=-\tilde{x}_2\tilde{x}_1$. We now show $\tilde{x}_1\tilde{x}_j=-\tilde{x}_j\tilde{x}_1$ for all $2\leq j\leq d$. Assume it is true for $j\leq k$, then 
\begin{align*}
\hh \x_1\x_{k+1}+\x_{k+1}\x_1 
&= \x_1(s_k\x_ks_k-(c_k-c_{k+1})s_k)+(s_k\x_ks_k-(c_k-c_{k+1})s_k)\x_1\\
&=s_k(\x_1\x_k+\x_k\x_1)s_k-\x_1(c_k-c_{k+1})s_k-(c_k-c_{k+1})s_k\x_1=0.
\end{align*}
Now we show $\tilde{x}_i\tilde{x}_j=-\tilde{x}_j\tilde{x}_i$ by fixing $j$ and induction on $i$. Suppose it is true for $i\leq k-1$, then
\begin{align*}
\x_k\x_{j}+\x_{j}\x_k 
&=(s_{k-1}\x_{k-1}s_{k-1}-(c_{k-1}-c_k)s_{k-1})\x_{j}+\x_{j}(s_{k-1}\x_{k-1}s_{k-1}-(c_{k-1}-c_k)s_{k-1})\\
&=s_{k-1}(\x_{k-1}\x_{j}+\x_{j}\x_{k-1})s_{k-1}-(c_{k-1}-c_k)s_{k-1}\x_{j}-\x_{j}(c_{k-1}-c_k)s_{k-1}=0.
\end{align*}
The argument for $\tilde{z}_i\tilde{z}_j=-\z_j\z_i$ is exactly the same. For the $\tilde{y}$-version, the induction argument is also the same and the base case is shown via
\begin{align*}
&\hh (\z_1-\x_1)(s_1(\z_1-\x_1)s_1-(c_1-c_2)s_1) +(s_1(\z_1-\x_1)s_1-(c_1-c_2)s_1)(\z_1-\x_1)\\
&=\z_1\z_2+\z_2\z_1 - \z_1s_1\x_1s_1-s_1\x_1s_1\z_1 -\x_1\z_2-\z_2\x_1 +\x_1s_1\x_1s_1+s_1\x_1s_1\x_1\\
&=-(s_1\z_2-(c_1-c_2))\x_1s_1-s_1\x_1(\z_2s_1+(c_1-c_2))\\
&\hh +(s_1\x_2 -(c_1-c_2))\x_1s_1+s_1\x_1(\x_2s_1+(c_1-c_2))=0.
\end{align*}
The last equality requires relations $\tilde{z}_2\x_1=-\x_1\z_2$ and $\x_1\x_2=-\x_2\x_1$ in $\mathcal{H}_d$.

(Relations between polynomial rings) To show (\ref{xz}), we first claim $\tilde{z}_j\x_1=-\x_1\z_j$ $(2\leq j\leq d)$ by induction on $j$. The base case $j=2$ is a defining relation in $\mathcal{H}_d$. Suppose this statement is true for  $i\leq k$, then 
\begin{align*}
\z_{k+1} \x_1 +\x_1\z_{k+1}&= (s_k\z_ks_k-(c_k-c_{k+1})s_k)\x_1+\x_1(s_k\z_ks_k-(c_k-c_{k+1})s_k)\\
&= s_k (\z_k\x_1+\x_1\z_k)s_k =0.
\end{align*}
Then we claim $\z_i\x_j=\x_j\z_i$ by fixing $i$ and induction on $j$. Suppose this is true for $j\leq j_0$, then
\begin{align*}
\tilde{z}_i \tilde{x}_{j_0+1} +  \tilde{x}_{j_0+1} \tilde{z}_i&=\tilde{z}_i(s_{j_0}\tilde{x}_{j_0}s_{j_0}-(c_{j_0}-c_{{j_0}+1})s_{j_0})+(s_{j_0}\tilde{x}_{j_0}s_{j_0}-(c_{j_0}-c_{{j_0}+1})s_{j_0})\tilde{z}_i\\
&=s_{j_0}(\tilde{z}_i\tilde{x}_{j_0}+\tilde{x}_{j_0}\tilde{z}_i)s_{j_0}-(\tilde{z}_i(c_{j_0}-c_{{j_0}+1})s_{j_0})+(c_{j_0}-c_{{j_0}+1})s_{j_0})\tilde{z}_i)=0.
\end{align*}
Notice the argument fails when $s_{j_0}$ and $\z_i$ no longer commute, especially when $j_0=i-1$, which explains the condition $i<j$ on the indices.

The induction argument for $\tilde{y}$-version is the same, with the base case following from 
\begin{align*}
&\z_i\y_1=\z_i(\z_1-\x_1)=-(\z_1-\x_1)\z_i=-\y_1\z_i &(2\leq i\leq d).
\end{align*}

To show (\ref{yz}), notice
\begin{align*}
&\hh (\tilde{z}_0+\tilde{z}_1-(\z_1-\x_1))(\z_1-\x_1)+(\z_1-\x_1) (\tilde{z}_0+\tilde{z}_1-(\z_1-\x_1))\\
&= (\z_0+\x_1)(\z_1-\x_1)+(\z_1-\x_1) (\z_0+\x_1)\\
&= \x_1(\z_1-\x_1)-\z_0\x_1+\z_0\z_1 + (\z_1-\x_1)\x_1 +\z_1\z_0-\x_1\z_0\\
&= \x_1(\z_1-\x_1-\z_0)+(\z_1-\x_1-\z_0)\x_1=0.
\end{align*}
Lastly, for (\ref{jm}), let us induct on $i$. The base case $i=1$ is from the definition of $\y_1$. Assume the claim is true for $i\leq k$,
\begin{align*}
 \tilde{x}_{k+1}+\tilde{y}_{k+1}&= s_k\x_ks_k+s_k\y_ks_k-2(c_k-c_{k+1})s_k= s_k(\z_k - \displaystyle\sum_{1\leq j\leq k-1}(c_j-c_k)t_{j,k} )s_k-2(c_k-c_{k+1})s_k\\
&=\z_{k+1}-s_k(\displaystyle\sum_{1\leq j\leq k-1}(c_j-c_k)t_{j,k} )s_k-(c_k-c_{k+1})s_k\\
&=\z_{k+1}-(\displaystyle\sum_{1\leq j\leq k-1}(c_j-c_{k+1})s_kt_{j,k}s_k)-(c_k-c_{k+1})s_k\\
&=\z_{k+1}-(\displaystyle\sum_{1\leq j\leq k-1}(c_j-c_{k+1})t_{j,k+1})-(c_k-c_{k+1})s_k=\z_{k+1}-(\displaystyle\sum_{1\leq j\leq k}(c_j-c_{k+1})t_{j,k+1}).
\end{align*}
\end{proof}

\subsection{An action of $\mathcal{H}_d$}
For arbitrary $\q(n)$-module $M$ and $N$, we aim to define an action of $\mathcal{B}_d$ on $M\otimes N\otimes V^{\otimes d} $ which supercommutes with $\q(n)$. Recall the elements $e_{ij}$ and $f_{ij}$ defined in (\ref{ef}). We now define the odd Casimir tensor element
\begin{align}\label{oddomega}
\Omega=\sum_{1\leq i,j\leq n}e_{ij}\otimes f_{ji} -\sum_{1\leq i,j\leq n} f_{ij}\otimes e_{ji} \in \q(n)\otimes \q(n),
\end{align}
this is related to the even Casimir tensor $\overline{\Omega}$ in (\ref{evenomega}) by $\overline{\Omega}=\Omega(1\otimes C)$ for the element $C\in \operatorname{End}(V)$ in (\ref{Cdefinition}). Similar to $\overline{\Omega}$, $\Omega$ is known to induce $\q(n)$-linear actions: on a tensor product $M\otimes N$ of arbitrary $\q(n)$-modules $M$ and $N$, the action of $\Omega$ supercommutes with that of $\q(n)$.    Also recall the notation $\Omega_{M\otimes N,k}$ introduced after (\ref{omega}). In particular, $\Omega_{M\otimes N,k} \in \operatorname{End}(M\otimes N\otimes V^{\otimes d})$ is defined to be the action of the following element, using the module structure on $M\otimes N$ introduced in  (\ref{tensor}):
\begin{align*}
&\Omega_{M\otimes N,k}=\Omega_{M,k}+\Omega_{N,k}\\
=&\sum_{1\leq i,j\leq d} e_{ij}\otimes 1_N \otimes 1^{\otimes k-1} \otimes f_{ji} \otimes 1^{\otimes d-k}+\sum_{1\leq i,j\leq d} 1_M\otimes e_{ij} \otimes 1^{\otimes k-1} \otimes f_{ji} \otimes 1^{\otimes d-k}\\
&-\sum_{1\leq i,j\leq d} f_{ij}\otimes 1_N \otimes 1^{\otimes k-1} \otimes e_{ji} \otimes 1^{\otimes d-k}-\sum_{1\leq i,j\leq d} 1_M\otimes f_{ij} \otimes 1^{\otimes k-1} \otimes e_{ji} \otimes 1^{\otimes d-k}.
\end{align*}

Recall the definition of $\Omega_{M\otimes N\otimes V^{\otimes i-1},i}$ and $\Omega_{\ell,k}$ in Eq (\ref{omegadef1}) and Eq (\ref{omegadef2}). In particular, $\Omega_{\ell,k}$ applies the first and second tensor factors of $\Omega$ to the $\ell$-th and $k$-th copy of $V$, respectively. Notice an element $x\in \mathfrak{q}(n)$ acts on $M\otimes N\otimes V ^{\otimes i-1}$ via the following element, after applying coproduct $i$ times:
\begin{align*}
x\otimes 1^{\otimes i}+1\otimes x\otimes1^{\otimes i-1}+\cdots+1^{\otimes i} \otimes x,
\end{align*}
it follows that the action of $\Omega_{M\otimes N\otimes V^{\otimes i-1}, i}$ is equivalent to applying the first factor in $\Omega$ to each of $M$, $N$ and the first $(i-1)$-th copies of $V$, while applying the second factor to the $i$-th copy of $V$, in other words,
\begin{align}
\Omega_{M\otimes N\otimes V^{\otimes i-1}, i}=\Omega_{M,i}+\Omega_{N,i}+\Omega_{1,i}+\cdots+\Omega_{i-1,i}. \label{omegadef3}
\end{align}

\begin{theorem}\label{hacts}
There is a well-defined action of the algebra $\mathcal{H}_d$ on $V^{\otimes d}$, where $s_1,\dots,s_{d-1}$, $c_1,\dots,c_d$ act by the Sergeev action in (\ref{sergeev1}) and (\ref{sergeev2}), $\tilde{x}_1$ acts as $\Omega_{M,1}$, $\tilde{z}_0$ acts as $\Omega_{M,N}$, and $\tilde{z}_i$ acts as $\Omega_{M\otimes N\otimes V^{\otimes i-1}, i}$  in Eq (\ref{omegadef3}) for $1\leq i\leq d$ (under the convention $V^{\otimes 0}=\mathbb{C}$). Moreover, this action supercommutes with $\mathfrak{q}(n)$, and induces an algebra homomorphism $\mathcal{H}_d \to \operatorname{End}_{\q(n)}(V^{\otimes d})$.
\end{theorem}

\begin{proof}
It is enough to show that all the defining relations in $\mathcal{H}_d$ are satisfied. The ``Sergeev relations" are automatically satisfied because of the action of the Sergeev algebra. Some of these calculations have an analogue using the even Casimir tensor in (\ref{evenomega}), in the proof of \cite[Theorem 7.4.1]{HKS}.

1) Hecke relations.

First we argue that $s_i\Omega_{j,i}=\Omega_{j,i+1}s_i$. For any $j\leq i-1$, $v_M\otimes v_N\otimes v_1\otimes\cdots\otimes v_d\in \two$,
\begin{align*}
&\hh s_i\Omega_{j,i}( \vv) \\
&=(-1)^{\overline{v_M}+\cdots +\overline{v_{i-1}}}s_i  \sum_{1\leq p,q\leq n} v_M\otimes\cdots\otimes e_{pq} v_j \otimes \cdots \otimes  f_{qp}v_i \otimes v_{i+1}\otimes  \cdots \otimes v_d \\
&\hh- (-1)^{\overline{v_M}+\cdots +\overline{v_{j-1}}}s_i \sum_{1\leq p,q\leq n} v_M\otimes\cdots\otimes f_{pq} v_j \otimes \cdots \otimes  e_{qp}v_i \otimes v_{i+1}\otimes \cdots \otimes v_d\\
&= (-1)^{(\overline{v_i}+\overline{1})\overline{v_{i+1}}+\overline{v_M}+\cdots +\overline{v_{i-1}} } \sum_{1\leq p,q\leq n}v_M\otimes\cdots\otimes e_{pq} v_j \otimes \cdots  \otimes v_{i+1}\otimes  f_{qp}v_i  \otimes\cdots \otimes v_d\\
&\hh - (-1)^{\overline{v_i}\cdot \overline{v_{i+1}}+\overline{v_M}+\cdots +\overline{v_{j-1}}} \sum_{1\leq p,q\leq n} v_M\otimes\cdots\otimes f_{pq} v_j \otimes \cdots  \otimes v_{i+1} \otimes e_{qp}v_i \otimes\cdots \otimes v_d.
\end{align*}
On the other hand, 
\begin{align*}
&\hh \Omega_{j,i+1}s_i(\vv)\\
&= (-1)^{\overline{v_{i+1}}\cdot \overline{v_i}}\Omega_{j,i+1} (v_M\otimes v_N\otimes v_1\otimes \cdots \otimes v_{i+1}\otimes v_{i}\otimes \cdots\otimes v_d)\\
&= (-1)^{\overline{v_{i+1}}\cdot \overline{v_i}+\overline{v_M}+\cdots+\overline{v_{i-1}}+\overline{v_{i+1}}} \sum_{1\leq p,q\leq n}v_M\otimes \cdots \otimes e_{pq}v_j \otimes \cdots \otimes v_{i+1}\otimes f_{qp} v_i \otimes \cdots\otimes v_d\\
&\hh +(-1)^{\overline{v_{i+1}}\cdot \overline{v_i}+\overline{v_M}+\cdots+\overline{v_{j-1}}} \sum_{1\leq p,q\leq n}v_M\otimes \cdots \otimes f_{pq}v_j \otimes \cdots \otimes v_{i+1}\otimes e_{qp} v_i\otimes \cdots\otimes v_d.
\end{align*}
And the two expressions are equal as expected. Similarly, $s_i\Omega_{M,i}=\Omega_{M,i+1}s_i $ and $s_i\Omega_{N,i}=\Omega_{N,i+1}s_i $. We also check that $-\Omega_{i,i+1}s_i-c_i+c_{i+1}=0$:

Recall that $\{e_1,\dots,e_n,f_1,\dots,f_1\}$ defined in Section~\ref{2.1} form an ordered basis of $V$. If $w$ is either symbol $e$ or $f$, then the action of a basis of $\q(n)$ in (\ref{ef}) act as
\begin{align*}
e_{pq}w_i=\delta_{iq}w_p, \hspace{.2 in} f_{pq}w_i=(-1)^{\overline{w_i}}\delta_{iq}(c.w_p).
\end{align*}
Assume that $v_i$ and $v_{i+1}$ are basis vectors of $V$, and $v_i=w_a$, $v_{i+1}=u_b$, where $w,u$ are symbols $e$ or $f$. The following sums are over all $1\leq p,q\leq n$.
\begin{align*}
&\hh (-\Omega_{i,i+1}s_i-c_i+c_{i+1}).(v_M\otimes \cdots \otimes v_{i-1}\otimes w_a \otimes u_b \otimes \cdots \otimes v_d)\\
&=-(-1)^{\overline{w_a}\cdot\overline{u_b}} \Omega_{i,i+1}(v_M\otimes \cdots \otimes v_{i-1}\otimes w_b \otimes u_a \otimes \cdots \otimes v_d)\\
&\hh -(-1)^{\overline{v_M}+\cdots+\overline{v_{i-1}}} v_M\otimes \cdots \otimes v_{i-1}\otimes c.w_a \otimes u_b \otimes \cdots \otimes v_d\\
&\hh +(-1)^{\overline{v_M}+\cdots+\overline{v_{i-1}}+\overline{w_a}}v_M\otimes \cdots \otimes v_{i-1}\otimes w_a \otimes c.u_b \otimes \cdots \otimes v_d\\
&=-(-1)^{\overline{w_a}\cdot\overline{u_b}+\overline{v_M}+\cdots+\overline{v_{i-1}}+\overline{u_b}}\sum v_M\otimes \cdots \otimes v_{i-1}\otimes e_{pq}u_b \otimes f_{qp}w_a \otimes \cdots \otimes v_d\\
&\hh +(-1)^{\overline{w_a}\cdot\overline{u_b}+\overline{v_M}+\cdots+\overline{v_{i-1}}} \sum v_M\otimes \cdots \otimes v_{i-1}\otimes f_{pq}u_b \otimes e_{qp}w_a \otimes \cdots \otimes v_d\\
&\hh -(-1)^{\overline{v_M}+\cdots+\overline{v_{i-1}}} v_M\otimes \cdots \otimes v_{i-1}\otimes c.w_a \otimes u_b \otimes \cdots \otimes v_d\\
&\hh +(-1)^{\overline{v_M}+\cdots+\overline{v_{i-1}}+\overline{w_a}}v_M\otimes \cdots \otimes v_{i-1}\otimes w_a \otimes c.u_b \otimes \cdots \otimes v_d\\
&=-(-1)^{\overline{w_a}\cdot\overline{u_b}+\overline{v_M}+\cdots+\overline{v_{i-1}}+\overline{u_b}+\overline{w_a}} v_M\otimes \cdots \otimes v_{i-1}\otimes u_a \otimes c.w_b \otimes \cdots \otimes v_d\\
&\hh +(-1)^{\overline{w_a}\cdot\overline{u_b}+\overline{v_M}+\cdots+\overline{v_{i-1}}+\overline{u_b}} v_M\otimes \cdots \otimes v_{i-1}\otimes c.u_a \otimes w_b \otimes \cdots \otimes v_d\\
&\hh -(-1)^{\overline{v_M}+\cdots+\overline{v_{i-1}}} v_M\otimes \cdots \otimes v_{i-1}\otimes c.w_a \otimes u_b \otimes \cdots \otimes v_d\\
&\hh +(-1)^{\overline{v_M}+\cdots+\overline{v_{i-1}}+\overline{w_a}}v_M\otimes \cdots \otimes v_{i-1}\otimes w_a \otimes c.u_b \otimes \cdots \otimes v_d.
\end{align*}
The answer is zero if $u$ and $w$ are the same symbols. If $u$ is $e$ and $w$ is $f$, $\overline{u_b}=\overline{0}$, $\overline{w_a}=\overline{1}$, $c.w_i=-u_i$ and $c.u_i=w_i$.
\begin{align*}
&\hh (-\Omega_{i,i+1}s_i-c_i+c_{i+1}).(v_M\otimes \cdots \otimes v_{i-1}\otimes w_a \otimes u_b \otimes \cdots \otimes v_d)\\
&=-(-1)^{\overline{v_M}+\cdots+\overline{v_{i-1}}} v_M\otimes \cdots \otimes v_{i-1}\otimes u_a \otimes u_b \otimes \cdots \otimes v_d\\
&\hh +(-1)^{\overline{v_M}+\cdots+\overline{v_{i-1}}} v_M\otimes \cdots \otimes v_{i-1}\otimes w_a \otimes w_b \otimes \cdots \otimes v_d\\
&\hh -(-1)^{\overline{v_M}+\cdots+\overline{v_{i-1}}+1} v_M\otimes \cdots \otimes v_{i-1}\otimes u_a \otimes u_b \otimes \cdots \otimes v_d\\
&\hh +(-1)^{\overline{v_M}+\cdots+\overline{v_{i-1}}+\overline{1}}v_M\otimes \cdots \otimes v_{i-1}\otimes w_a \otimes w_b \otimes \cdots \otimes v_d=0.
\end{align*}
On the other hand, if If $u$ is $f$ and $w$ is $e$, $\overline{u_b}=\overline{1}$, $\overline{w_a}=\overline{0}$, $c.w_i=u_i$ and $c.u_i=-w_i$.
\begin{align*}
&\hh (-\Omega_{i,i+1}s_i-c_i+c_{i+1}).(v_M\otimes \cdots \otimes v_{i-1}\otimes w_a \otimes u_b \otimes \cdots \otimes v_d)\\
&=-(-1)^{\overline{v_M}+\cdots+\overline{v_{i-1}}+\overline{1}} v_M\otimes \cdots \otimes v_{i-1}\otimes u_a \otimes u_b \otimes \cdots \otimes v_d\\
&\hh +(-1)^{\overline{v_M}+\cdots+\overline{v_{i-1}}} v_M\otimes \cdots \otimes v_{i-1}\otimes w_a \otimes w_b \otimes \cdots \otimes v_d\\
&\hh -(-1)^{\overline{v_M}+\cdots+\overline{v_{i-1}}} v_M\otimes \cdots \otimes v_{i-1}\otimes u_a \otimes u_b \otimes \cdots \otimes v_d\\
&\hh +(-1)^{\overline{v_M}+\cdots+\overline{v_{i-1}}+\overline{1}}v_M\otimes \cdots \otimes v_{i-1}\otimes w_a \otimes w_b \otimes \cdots \otimes v_d=0.
\end{align*}
Therefore,
\begin{align*}
&\hh s_i \tilde{z}_i -\tilde{z}_{i+1}s_i-c_i+c_{i+1}\\
&=s_i(\Omega_{M,i}+\Omega_{N,i}+\Omega_{1,i}+\cdots+ \Omega_{i-1,i})-(\Omega_{M,i+1}+\Omega_{N,i_1}+\Omega_{1,i+1}+\cdots+ \Omega_{i,i+1})s_i-c_i+c_{i+1}\\
&=-\Omega_{i,i+1}s_i-c_i+c_{i+1}=0.
\end{align*}

The remaining commuting Hecke relations are straightforward to check.

2) Clifford twist relations.

It is enough to check that $\Omega_{i,j}$ commutes with $c_k$, for any $1\leq i,j\leq d$. The cases when $k\neq i,j$ are straightforward to check. When $i=k$, it is enough to check that $\Omega (c\otimes 1) = -(c \otimes 1)\otimes \Omega$. The following sums are over all $1\leq p,q\leq d$.
\begin{align*}
\Omega (c\otimes 1) &=\sum (e_{pq}\otimes f_{qp})(1\otimes c)-\sum(f_{pq}\otimes e_{qp})(1\otimes c)=\sum (e_{pq}\otimes  f_{qp}c)-\sum(f_{pq}\otimes e_{qp} c)\\
&=-\sum (e_{pq}\otimes  cf_{qp})-\sum(f_{pq}\otimes c e_{qp} )=- \sum(1\otimes c)(e_{pq}\otimes f_{qp})+\sum(1\otimes c)(f_{pq}\otimes e_{qp})\\
&= -(1\otimes c)\Omega.
\end{align*}
Similarly, $\Omega (c\otimes 1)= -(c\otimes 1)\otimes \Omega$ and therefore $\Omega_{i,j}c_j=-c_j \Omega_{i,j}$. 

3) Polynomial relations. 

For the relations $\x_1\x_2=-\x_2\x_1$ and $\z_0\z_1=-\z_1\z_0$, we will first show
\begin{align}\label{omegamixed}
\Omega_{M,N}(\Omega_{M,1}+\Omega_{N,1})=-(\Omega_{M,1}+\Omega_{N,1})\Omega_{M,N}.
\end{align}
First,
\begin{align*}
&\hh\Omega_{M,N}\Omega_{M,1}+\Omega_{M,1}\Omega_{M,N}\\
&=\sum_{i,j}(e_{ij}\otimes f_{ji}\otimes 1-f_{ij}\otimes e_{ji}\otimes 1) \sum_{p,q}(e_{pq}\otimes 1\otimes f_{qp}-f_{pq}\otimes 1\otimes e_{qp})\\
&\hh +\sum_{p,q}(e_{pq}\otimes 1\otimes f_{qp}-f_{pq}\otimes 1\otimes e_{qp})\sum_{i,j}(e_{ij}\otimes f_{ji}\otimes 1-f_{ij}\otimes e_{ji}\otimes 1)\\
&=\sum_{i,j,p,q}((e_{ij}e_{pq}-e_{pq}e_{ij})\otimes f_{ji}\otimes f_{qp} + (e_{ij}f_{pq}-f_{pq}e_{ij})\otimes f_{ji}\otimes e_{qp}  \\
&\hh -(f_{ij}e_{pq}-e_{pq}f_{ij})\otimes e_{ji}\otimes f_{qp}+ (f_{ij}f_{pq}+f_{pq}f_{ij})\otimes e_{ji}\otimes e_{qp} \\
&=\sum_{i,j,p,q}((\delta_{jp}e_{iq}-\delta_{iq}e_{pj})\otimes f_{ji}\otimes f_{qp}+(\delta_{jp}f_{iq}-\delta_{iq}f_{pj})\otimes  f_{ji}\otimes e_{qp}\\
&\hh -(\delta_{jp}f_{iq}-\delta_{iq}f_{pj})\otimes e_{ji}\otimes f_{qp}+(\delta_{ip}e_{iq}+\delta_{iq}e_{pj})\otimes e_{ji}\otimes e_{qp}\\
&=\sum_{i,p,q}e_{iq}\otimes f_{pi}\otimes f_{qp}-\sum_{j,p,q}e_{pj}\otimes f_{jq}\otimes f_{qp}+\sum_{i,p,q}f_{iq}\otimes  f_{pi}\otimes e_{qp} -\sum_{j,p,q}f_{pj}\otimes  f_{jq}\otimes e_{qp}\\
 &\hh -\sum_{i,p,q}f_{iq}\otimes e_{pi}\otimes f_{qp}+\sum_{j,p,q}f_{pj}\otimes e_{jq}\otimes f_{qp}+\sum_{j,p,q}e_{iq}\otimes e_{jp}\otimes e_{qp}+\sum_{j,p,q}e_{pj}\otimes e_{jq}\otimes e_{qp},
\end{align*}
Also
\begin{align*}
&\hh \Omega_{M,N}\Omega_{N,1}+\Omega_{N,1}\Omega_{M,N}\\
&= \sum_{p,q}(e_{pq}\otimes f_{qp}\otimes 1-f_{pq}\otimes e_{qp}\otimes 1)\sum_{i,j}(  1\otimes e_{ij}\otimes f_{ji}-1\otimes f_{ij}\otimes e_{ji})\\
&+\sum_{i,j}( 1\otimes e_{ij}\otimes f_{ji}-1\otimes f_{ij}\otimes e_{ji}) \sum_{p,q}(e_{pq}\otimes f_{qp}\otimes 1-f_{pq}\otimes e_{qp}\otimes 1)\\
&= \sum_{i,j,p,q}e_{pq}\otimes (f_{qp}e_{ij}-e_{ij}f_{qp})\otimes f_{ji}-\sum_{i,j,p,q}e_{pq}\otimes (f_{qp}f_{ij}+f_{ij}f_{qp})\otimes e_{ji}\\
&\hh -\sum_{i,j,p,q}f_{pq}\otimes (e_{qp} e_{ij}-e_{ij}e_{qp})\otimes f_{ji}+\sum_{i,j,p,q} f_{pq}\otimes (e_{qp}f_{ij}-f_{ij}e_{qp})\otimes e_{ji}\\
&=\sum_{i,j,p,q}e_{pq}\otimes (\delta_{pi}f_{qj}-\delta_{jq}f_{ip})\otimes f_{ji}-\sum_{i,j,p,q}e_{pq}\otimes (\delta_{pi}e_{qj}+\delta_{jq}e_{ip})\otimes e_{ji}\\
&\hh -\sum_{i,j,p,q}f_{pq}\otimes (\delta_{pi}e_{qj}-\delta_{jq}e_{ip})\otimes f_{ji}+\sum_{i,j,p,q} f_{pq}\otimes (\delta_{pi}f_{qj}-\delta_{jq}f_{ip})\otimes e_{ji}\\
&=\sum_{j,p,q}e_{pq}\otimes f_{qj}\otimes f_{jp}-\sum_{i,p,q}e_{pq}\otimes f_{ip}\otimes f_{qi} -\sum_{j,p,q}e_{pq}\otimes e_{qj}\otimes e_{jp} -\sum_{i,p,q}e_{pq}\otimes e_{ip}\otimes e_{qi}\\
&\hh -\sum_{j,p,q}f_{pq}\otimes e_{qj}\otimes f_{jp}+\sum_{i,p,q}f_{pq}\otimes e_{ip}\otimes f_{qi}+\sum_{j,p,q} f_{pq}\otimes f_{qj}\otimes e_{jp}-\sum_{i,p,q} f_{pq}\otimes f_{ip}\otimes e_{qi}.
\end{align*}
By comparison, these two expressions are equal, (\ref{omegamixed}) holds,
therefore $\z_0\z_1=-\z_1\z_0$. Similarly,
\begin{align*}
\Omega_{M,1}(\Omega_{M,2}+\Omega_{1,2})=-(\Omega_{M,2}+\Omega_{1,2})\Omega_{M,1}.
\end{align*}
By an argument similar to the ones proving``Hecke relations'', $\tilde{x}_2$ acts as $\Omega_{M,2}+\Omega_{1,2}$, therefore $\x_1\x_2=-\x_2\x_1$.

Lastly, $\tilde{z}_1\tilde{z}_2=-\tilde{z}_2\tilde{z}_1$ is equivalent to \begin{align*}
(\Omega_{M,1}+\Omega_{N,1})(\Omega_{M,2}+\Omega_{N,2}+\Omega_{1,2})=-(\Omega_{M,2}+\Omega_{N,2}+\Omega_{1,2})(\Omega_{M,1}+\Omega_{N,1}).
\end{align*}
To prove this, one simply uses the fact that
\begin{align}
\Omega_{M,1}(\Omega_{M,2}+\Omega_{1,2})=&-(\Omega_{M,2}+\Omega_{1,2})\Omega_{M,1},\label{var1}\\
\Omega_{N,1}(\Omega_{N,2}+\Omega_{1,2})=&-(\Omega_{N,2}+\Omega_{1,2})\Omega_{N,1},\label{var2}\\
\Omega_{M,1}\Omega_{N,2}+\Omega_{N,2}\Omega_{M,1}=&\Omega_{N,1}\Omega_{M,2}+\Omega_{M,2}\Omega_{N,1}=0. \label{var3}
\end{align}
where (\ref{var1}) and (\ref{var2}) are variations of (\ref{omegamixed}), and (\ref{var3}) is straightforward to check.

(4) Relations between polynomial rings.

Since
\begin{align*}
&\hh\z_2\x_1+\x_1\z_2= (\Omega_{M,2}+\Omega_{N,2}+\Omega_{1,2})\Omega_{M,1}+\Omega_{M,1}(\Omega_{M,2}+\Omega_{N,2}+\Omega_{1,2}),
\end{align*}
one simply uses (\ref{var1}) and the first part of (\ref{var3}).

Lastly,
\begin{align}
(\z_0-\z_1+\x_1)\x_1+\x_1(\z_0-\z_1+\x_1)=(\Omega_{M,N}-\Omega_{N,1})\Omega_{M,1}+\Omega_{M,1}(\Omega_{M,N}-\Omega_{N,1}),
\end{align}\label{notlastrelation}
where

and
\begin{align*}
&\hh\Omega_{N,1}\Omega_{M,1}+\Omega_{M,1}\Omega_{N,1}\\
&=\sum_{i,j}(1\otimes e_{ij}\otimes f_{ji}-1\otimes f_{ij}\otimes e_{ji})\sum_{pq}(e_{pq}\otimes 1\otimes f_{qp}-f_{pq}\otimes 1\otimes e_{qp})\\
&\hh +\sum_{pq}(e_{pq}\otimes 1\otimes f_{qp}-f_{pq}\otimes 1\otimes e_{qp})\sum_{i,j}(1\otimes e_{ij}\otimes f_{ji}-1\otimes f_{ij}\otimes e_{ji})\\
&= \sum_{i,j,p,q} (e_{pq}\otimes e_{ij}\otimes (f_{ji}f_{qp}+f_{qp}f_{ji})+f_{pq}\otimes e_{ij}\otimes (f_{ji}e_{qp}-e_{qp}f_{ji})  \\
&\hh - e_{pq}\otimes f_{ij}\otimes (e_{ji}f_{qp}-f_{qp}e_{ji})-f_{pq}\otimes f_{ij}\otimes (e_{ji}e_{qp}-e_{qp}e_{ji}))\\
&= \sum_{i,j,p,q} (e_{pq}\otimes e_{ij}\otimes (\delta_{iq}e_{jp}+\delta_{jp}e_{qi})+f_{pq}\otimes e_{ij}\otimes (\delta_{iq}f_{jp}-\delta_{jp}f_{qi})  \\
&\hh - e_{pq}\otimes f_{ij}\otimes (\delta_{iq}f_{jp}-\delta_{jp}f_{qi})-f_{pq}\otimes f_{ij}\otimes (\delta_{iq}e_{jp}-\delta_{jp}e_{qi}))\\
&= \sum_{j,p,q} e_{pq}\otimes e_{qj}\otimes e_{jp}+\sum_{i,p,q}e_{pq}\otimes e_{ip}\otimes e_{qi}+\sum_{j,p,q}f_{pq}\otimes e_{qj}\otimes f_{jp} -\sum_{i,p,q}f_{pq}\otimes e_{ip}\otimes f_{qi}  \\
&\hh - \sum_{j,p,q}e_{pq}\otimes f_{qj}\otimes f_{jp}+\sum_{i,p,q}e_{pq}\otimes f_{ip}\otimes f_{qi} -\sum_{j,p,q}f_{pq}\otimes f_{qj}\otimes e_{jp}+\sum_{i,p,q}f_{pq}\otimes f_{ip}\otimes e_{qi}.
\end{align*}
By comparing this expression and (\ref{omegamixed}) one obtains that (\ref{lastrelation}) is zero. Therefore we have checked all the relations.
\end{proof}

The proof for ``Hecke relations'' lead to an explicit description of the action of $\tilde{x}_i$ and $\tilde{y}_i$ defined in (\ref{Hecke2}) and (\ref{Hecke3}).

\begin{cor}\label{xiyiact}
Based on the action in Theorem~\ref{hacts}, for $1\leq i\leq d$, $\tilde{x}_i$ acts as 
\begin{align*}
\Omega_{M\otimes V^{\otimes i-1},V}=\Omega_{M,i}+\Omega_{1,i}+\cdots+\Omega_{i-1,i},
\end{align*}
and $\tilde{y}_i$ acts as 
\begin{align*}
\Omega_{N\otimes V^{\otimes i-1},V}=\Omega_{N,i}+\Omega_{1,i}+\cdots+\Omega_{i-1,i}.
\end{align*}
\end{cor}

We mentioned earlier that this action generalizes the construction by Sergeev \cite{Sergeev} and Hill-Kujawa-Sussan \cite{HKS}. We now describe this more precisely. Recall that $H_d$ is the degenerate affine Hecke-Clifford algebra introduced shortly after (\ref{Cdefinition}). Similar to (\ref{Hecke2}) and (\ref{Hecke3}), define elements $w_i \in \operatorname{Ser}_d$ recursively via $w_1=0$, $w_{i+1}=s_iw_is_i-(c_i-c_{i+1})s_i$ for $1\leq i\leq d-1$.

\begin{prop}\label{specialize}
1) There exists a surjective homomorphism $\phi_1: \mathcal{H}_d \to \operatorname{Ser}_d$, sending $\tilde{x}_1\mapsto 0$, $z_0\mapsto 0$, $\tilde{z}_i\mapsto w_i(1\leq i\leq d)$, and $s_i$, $c_i$ to generators under the same name. When $M=N=\mathbb{C}$ is the trivial $\q(n)$-module, the action of $\Hd$ on $M\otimes N\otimes V^{\otimes d}$  factors through the quotient $\operatorname{Ser}_d$ and induces the action introduced in (\ref{sergeev1}) and (\ref{sergeev2}). 

2) There exists a surjective homomorphism $\phi_2: \mathcal{H}_d\to {H}_{d}$, sending $\tilde{x}_1\mapsto 0$, $\tilde{z}_0\mapsto 0$, and $\tilde{z}_i\mapsto -x_ic_i (1\leq i\leq d)$, $s_i$, $c_i$ to generators under the same name.  When $M=\mathbb{C}$ is the trivial $\q(n)$-module, the action of $\mathcal{H}_d$ on $M\otimes N\otimes V^{\otimes d}$  factors through the quotient ${H}_d$ and induces the action introduced in (\ref{omega}).
\end{prop}

\begin{proof}
Under the specified maps, all the relations in $\mathcal{H}_d$ are satisfied. The second half of both claims follows directly from the fact that $\Omega=-\overline{\Omega}(1\otimes c)$, $\Omega_{M,i}=0$ whenever $M= \mathbb{C}$, and $\Omega_{N,i}=0$ whenever $N= \mathbb{C}$. 
\end{proof}

\subsection{When $M$ or $N$ is of Type Q}\label{2.4}
For the remaining sections of this article, we shall impose further conditions on the modules $M$ and $N$. In particular, $\q(n)$-modules admit a special phenomenon in which Schur's lemma no longer holds. Intead, the following super Schur's lemma holds, as mentioned in \cite[Lemma 3.4]{Wang}. A $\q(n)$-module is simple if it has no proper submodules.

\begin{lemma}[Super Schur's Lemma]\cite[Lemma 3.4]{Wang}\label{Pre_3_superschur}
Let $W,U$ be two simple modules for $\q(n)$, then
\begin{align*}
\operatorname{dim}\operatorname{Hom}_{\q(n)}(W,U)=\begin{cases}
0 \hspace{.5 in} \text{ if } W \not\simeq  U,\\
1 \hspace{.5 in} \text{ if } W\simeq U \text{ is of Type M, by definition}, \\
2 \hspace{.5 in} \text{ if } W\simeq U \text{ is of Type Q, by definition}.
\end{cases}
\end{align*}
In the last case, $\operatorname{End}_{\q(n)}(W)$ is spanned by the identity map and an odd endomorphism $c$.
\end{lemma}

We shall define two slightly larger algebras $\mathcal{H}^{+}_d$ and $\mathcal{H}^{++}_d$, such that $\mathcal{H}_d\subset \mathcal{H}^{+}_d \subset  \mathcal{H}^{++}_d$. In particular, for two superalgebras $A$ and $B$, $A\otimes B$ is again a superalgebra with a $\mathbb{Z}_2$-grading: 
\begin{align*}
(A\otimes B)_{\overline{0}}=A_{\overline{0}}\otimes B_{\overline{0}} \oplus A_{\overline{1}}\otimes B_{\overline{1}}, \hspace{.2 in}
(A\otimes B)_{\overline{1}}=A_{\overline{0}}\otimes B_{\overline{1}} \oplus A_{\overline{1}}\otimes B_{\overline{0}},
\end{align*}
and the multiplication is defined as 
\begin{align*}
(a_1\otimes b_1)(a_2\otimes b_2)=(-1)^{\overline{b_1}\cdot\overline{a_2}} a_1a_2 \otimes b_1b_2.
\end{align*}
Recall that $\operatorname{Cl}_k$ is the Clifford algebra on $k$ letters, defined in (\ref{clifford}). Let $\mathcal{H}^{+}_d = \mathcal{H}_d \otimes \operatorname{Cl}_1$, and $c_0$ be the extra Clifford generator. Further, let $\mathcal{H}^{++}_d=\mathcal{H}^{+}_d \otimes \operatorname{Cl}_1$ and $c_M$ be the extra Clifford generator. Then these two algebras act on $M\otimes N\otimes V^{\otimes d}$ under the specific conditions on $M$ or $N$.

\begin{prop}
When $N$ is of Type Q, let $c^N\in \operatorname{End}_{\q(n)}(N)$ be an odd endomorphism such that $(c^N)^2=-\operatorname{id}_N$, then there is a well defined action of $\mathcal{H}^{+}_d$ on $M\otimes N\otimes V^{\otimes d}$, where $\mathcal{H}_d$ acts as in Theorem~\ref{hacts}, and $c_0$ acts as $1\otimes c^N \otimes 1^{\otimes d}$. Moreover, this action supercommutes with $\q(n)$ and induces a homomorphism of superalgebras $\mathcal{H}^{+}_d \to \operatorname{End}_{\q(n)}(M\otimes N\otimes V^{\otimes d})$. 

Furthermore, if $M$ is also of Type Q, let $c^M\in \operatorname{End}_{\q(n)}(M)$ be an odd endomorphism such that $(c^M)^2=-\operatorname{id}_M$, then there is a well defined action of $\mathcal{H}^{++}_d$ on $M\otimes N\otimes V^{\otimes d}$, where $\mathcal{H}^{+}_d$ acts as above, and $c_M$ acts as $c^M\otimes 1 \otimes 1^{\otimes d}$. Moreover, this action supercommutes with $\q(n)$ and induces a homomorphism of superalgebras $\mathcal{H}^{++}_d \to \operatorname{End}_{\q(n)}(M\otimes N\otimes V^{\otimes d})$.

\end{prop}

\begin{proof}
When $N$ if of Type Q, it is enough to show that the extra relations 
\begin{align*}
c_0^2=&-1, \hspace{.2 in} c_0c_i=-c_ic_0 \hspace{.1 in} (1\leq i\leq d) &c_0 s_i = s_i c_0 \hspace{.1 in}(1\leq i\leq d-1), \\
c_0 \tilde{x}_i =& -\tilde{x}_i c_0 , \hspace{.2 in} c_0 \tilde{y}_i = -\tilde{y}_i c_0 \hspace{.1 in} (1\leq i\leq d)  &c_0 \tilde{z}_i = -\tilde{z}_i c_0  \hspace{.1 in}(0\leq i\leq d).
\end{align*}
are satisfied. The first line is easy to check. The second line follows from the fact that
\begin{align*}
\Omega_{M,1}(1\otimes c\otimes 1)=-(1\otimes c\otimes 1)\Omega_{M,1}, \hspace{.2 in} \Omega_{N,1}(1\otimes c\otimes 1)=-(1\otimes c\otimes 1)\Omega_{N,1}.
\end{align*}
The case when $M$ is also of Type Q can be proved similarly.
\end{proof}

\section{A quotient of the Hecke-Clifford algebra}

\subsection{Combinatorics of polynomial modules}\label{3.1}
We will now take $M$ and $N$ to be an even more specific type of $\q(n)$-modules called polynomial modules. A \emph{polynomial} $\q(n)$-module is a direct summand of $V^{\otimes k}$ for some $k\in \mathbb{Z}_{\geq 0}$. It is immediate that a tensor product of polynomial modules are still polynomial. By Sergeev \cite{Sergeev}, the category of polynomial $\q(n)$-modules is semisimple, therefore, the tensor product of polynomial modules decomposes into a direct sum of polynomial modules. Furthermore, simple polynomial modules are paramatrized by the following set
\begin{align*}
\Lambda=\{\lambda=(\lambda_1,\dots,\lambda_n)\in \mathbb{Z}_{\geq 0}^n \hspace{.1 in}| \hspace{.1 in} \lambda_1\geq \lambda_2\geq \cdots \geq \lambda_n, \hspace{.1 in} \lambda_i>\lambda_{i+1} \text{ if }\lambda_i\neq 0 \}.
\end{align*}
Let $L(\lambda)$ be the simple module parameterized by $\lambda$. Recall that $e_{ij}\in \q(n)$ are elements defined in (\ref{ef}). The module $L(\lambda)$ is known to have the following property. Let $\mathfrak{h}_{\overline{0}}\subset \mathfrak{q}(n)$ be the subalgebra spanned by $e_{11},\dots,e_{nn}$, and let $\epsilon_i=e_{ii}^*\in \mathfrak{h}^*$ be their duals. For $\mu\in \Lambda$, identify $\mu$ with the element $\mu_1\epsilon_1+\cdots+\mu_n\epsilon_n \in \mathfrak{h}^*$. The $\mu$-weight space for a given $\q(n)$-module $W$ is the vector space
\begin{align}\label{weightspace}
W_{\mu}=\{w\in W \hspace{.1 in} | \hspace{.1 in}  e_{ii}.v=\mu(e_{ii})v, 1\leq i\leq n\},
\end{align}
then the module $L(\lambda)$ is known to have a weight space decomposition $L(\lambda)=\bigoplus_{\mu} L(\lambda)_{\mu}$, where the direct sum is over all weights $\mu\in \mathfrak{h}^*$ for which $L(\lambda)_{\mu}$ is nonzero. Moreover, $L(\lambda)_{\lambda}\neq 0$, and $\lambda$ is highest according to a certain partial order among all the weights occurring in the direct sum. We will refer to any $v_{\lambda}\in L(\lambda)_{\lambda}$ a highest weight vector.

The weights in $\Lambda$ can also be represented using shifted Young diagrams (sometimes also called strict partitions). A \emph{shifted Young diagram} $\lambda=(\lambda_1,\dots,\lambda_n)$ is a stacking of boxes, with $\lambda_1$ boxes in the first row, $\lambda_2$ boxes in the second row, etc, and the beginning of the $i$-th row is the $i$-th column. For example, the diagram associated to $(4,3,2,1)$ is the following 
\ytableausetup{smalltableaux}
\begin{align}\label{staircase}
\ydiagram{4,1+3,2+2,3+1}.
\end{align}
Our choice of $M$ and $N$ depends on the following considerations:

1) $M$ and $N$ should be polynomial, because a tensor product between polynomial modules decomposes into a direct sum of polynomial modules.

2) The multiplicity of each irreducible summand in $M\otimes N$ should be as small as possible.

Our next goal is to choose suitable $M$ and $N$, then give a complete description of all irreducible summands in $M\otimes N\otimes V^{\otimes d}$, and their multiplicities, for any nonnegative integer $d$. To do so we need some combinatorics developed by Stembridge \cite{Stembridge}, in addition to the character theory of $\q(n)$-modules.

For a finite dimensional $\q(n)$-modules with weight space decomposition $W=\oplus_{\mu\in I} W_{\mu}$, where $I$ is the set of weights for which the weight space is nonzero. The character of $W$ is defined as 
\begin{align*}
\operatorname{Ch}W=\sum_{\mu \in I} \operatorname{dim}W_{\mu}x_1^{\mu_1}x_2^{\mu_2}\cdots x_n^{\mu_n}.
\end{align*}

Let $\ell(\lambda)$ be the number of nonzero entries in $\lambda$. As mentioned in \cite[Theorem 4.11]{Brundan} and the fact that the Euler modules coincide with $L(\lambda)$ when $L(\lambda)$ is polynomial,  its character is given by $
 \operatorname{Ch} L(\lambda)= 2^{\lfloor \frac{\ell(\lambda)+1}{2} \rfloor} P_{\lambda}$,  where $P_{\lambda}$ is the Schur's P-function. This is a combinatorially defined function by counting certain types of tableaux of shape $\lambda$. The \emph{multiplicity} $m^{\gamma}_{\lambda,\mu}$ of each irreducible module $L(\gamma)$ in $L(\lambda)\otimes L(\mu)$ is related to the structural constants of $P_{\lambda}$ in the following sense, using properties of characters:
\begin{align*}
2^{\lfloor \frac{\ell(\lambda+1)}{2} \rfloor} P_{\lambda}\cdot 2^{\lfloor \frac{\ell(\mu+1)}{2} \rfloor} P_{\mu}=\operatorname{Ch}L(\lambda)\cdot \operatorname{Ch}L(\mu)=\operatorname{Ch}(L(\lambda)\otimes L(\mu))\\
=\operatorname{Ch}\left(\bigoplus_{\gamma}L(\gamma)^{\oplus m_{\lambda,\mu}^{\gamma}}\right)=\sum_{\gamma} m^{\gamma}_{\lambda,\mu}\operatorname{Ch}L(\gamma)= \sum_{\gamma} m^{\gamma}_{\lambda,\mu} 2^{\lfloor \frac{\ell(\gamma+1)}{2} \rfloor} P_{\gamma}.
\end{align*}

Therefore, if $f^{\gamma}_{\lambda,\mu}$ is the structural constant appearing in  $P_{\lambda}P_{\mu}=\sum_{\gamma} f^{\gamma}_{\lambda,\mu} P_{\gamma}$, then 
\begin{align}
m^{\gamma}_{\lambda,\mu}=f^{\gamma}_{\lambda,\mu}2^{\lfloor \frac{\ell(\lambda+1)}{2} \rfloor} 2^{\lfloor \frac{\ell(\mu+1)}{2} \rfloor}2^{-\lfloor \frac{\ell(\gamma+1)}{2} \rfloor}. \label{structuralconstant}
\end{align} 
In \cite[Theorem 8.3]{Stembridge}, Stembridge computed the structural constants $f^{\gamma}_{\lambda,\mu}$ for the functions $P_{\lambda}$ by counting a certain type of tableaux satisfying the lattice condition. In particular, a \emph{semistandard shifted tableaux} of skew shape $\gamma/\lambda$, is a filling of boxes in $\gamma$ that are not in $\lambda$, with integers $1,2,\dots$ and $1',2',\dots$ under the order $1'<1<2'<2<\cdots$, subject to the \emph{semistandard} condition: the numbers weakly increase along each row and each column, and each $i'$ (the primed integers) occur at most once along each row and each $i$ (the regular integers) occur at most once along each column. Given a semistandard tableau $T$ of shape $\lambda/\gamma$ with entries from $\{1,1',\dots,s,s'\}$, let the \emph{multiplicity} $\mu(T)\in \mathbb{Z}^{\mathbb{N}}$ be $(\mu_1,\mu_2,\dots,\mu_s,0,0,\dots)$, where $\mu_1$ is the total number of $1$ and $1'$, $\mu_2$ is the total number of $2$ and $2'$, etc. 

To compute $f_{\lambda,\mu}^{\gamma}$ we also need to define a \emph{lattice} condition: in particular, let $w=w_1w_2\cdots w_{t}$ be a word in $1,1',\dots,s,s'$. Denote by $\# i$	the number of times entry $i$ appears. For any integer $i\geq 0$, define the occurrence $m_i(k)$ in word $w$ at step $k$ as follows, with the convention $m_i(0)=0$.\\
1) $1\leq k\leq t$, 
\begin{align*}
m_i(k)= \#  i \text{ in }w_{t-k+1}, \dots, w_t,
\end{align*}
2) $t+1\leq k\leq 2t$,
\begin{align*}
m_i(k)=m_i(k-t)+\#  i' \text{ in }w_{1}, \dots, w_{k-t}.
\end{align*}

In other words, the function $m_i(k)$ is counting regular integers  from the end of $w$ backward to the beginning, and on top of that, counting primed integers from the beginning of $w$ forward to the end. Call $w$ a lattice word if the following is true

1) When $1\leq k\leq t$, if $m_i(k)=m_{i+1}(k)$, then $w_{t-k}\neq {i+1},{(i+1)'}$

2) When $t+1\leq k\leq 2t$, if $m_i(k)=m_{i+1}(k)$, then $w_{k-t+1}\neq {i},(i+1)'$

Define the absolute value $|i|=|i'|=i$ for $1\leq i\leq n$. In  \cite{Stembridge}, the Stembridge states the following rule for computing $f_{\lambda,\mu}^{\gamma}$ in (\ref{structuralconstant}), analogous to the Littlewood-Richardson rule for regular Schur functions.

\begin{theorem}\cite[Theorem 8.3]{Stembridge}\label{Pre_3_Stembridge}
$f_{\lambda,\mu}^{\gamma}$ is equal to the number of semistandard tableaux $T$ of shape $\gamma/\lambda$, satisfying the following property: let $w$ be the word obtained by reading the entries in $T$ from bottom row to the top row, from left to right, then $w$ is a lattice word of content $\mu$, and for any integer $i$ with nonzero multiplicity, the first letter in $w$ with absolute value $i$ is a regular integer. 
\end{theorem}

\textit{Example.} Let us compute $f_{\lambda,\mu}^{\gamma}$, where 
\begin{align*}
\ytableausetup{smalltableaux}
\lambda=\begin{ytableau}
 *(white) & \\
\none & \\
\end{ytableau} \hspace{.2 in}
\ytableausetup{smalltableaux}
\mu=\begin{ytableau}
 *(white) &  & \\
 \none &\\
\end{ytableau} \hspace{.2 in}
\ytableausetup{smalltableaux}
\gamma=\begin{ytableau}
*(white) &  & & \\
\none &  & &\\
\end{ytableau} \hspace{.2 in}.
\end{align*}
All semistandard shifted tableaux must be of the following form, if they are of shape $\gamma/\lambda$ with fillings of multiplicity $\mu$, where $\star$ is representing one of $1,1',2,2'$.
\begin{align*}
\begin{ytableau}
*(white) &  & \star & \star \\
\none &  & \star &\star \\
\end{ytableau}
\end{align*}
The square can be filled with the following posibilities:
\begin{align*}
\begin{ytableau}
 1&1 \\
 1'&2\\
\end{ytableau} \hspace{.2 in}
\begin{ytableau}
1&1 \\
 1'&2'\\
\end{ytableau} \hspace{.2 in}
\begin{ytableau}
1&1' \\
1'&2\\
\end{ytableau} \hspace{.2 in}
\begin{ytableau}
1&1' \\
1'&2'\\
\end{ytableau} \hspace{.2 in}
\begin{ytableau}
1&2 \\
1'&2'\\
\end{ytableau} \hspace{.2 in}
\begin{ytableau}
1&2' \\
1'&2'\\
\end{ytableau} \hspace{.2 in}
\begin{ytableau}
1&1\\
2&2\\
\end{ytableau} \hspace{.2 in}
\begin{ytableau}
1&1\\
2&2'\\
\end{ytableau} \hspace{.2 in}
\begin{ytableau}
1&1'\\
2&2\\
\end{ytableau} \hspace{.2 in}
\begin{ytableau}
1&1'\\
2&2'\\
\end{ytableau} \hspace{.2 in}
\begin{ytableau}
1&2\\
2&2'\\
\end{ytableau} \hspace{.2 in}
\begin{ytableau}
1&2'\\
2&2'\\
\end{ytableau} \hspace{.2 in}
\end{align*}
However, only the following satisfies the Stembridge rule, and therefore $f^{\gamma}_{\lambda,\mu}=1$:
\begin{align*}
\begin{ytableau}
1&1\\
2&2'\\
\end{ytableau} \hspace{.2 in}
\end{align*}
The others fail for various reasons: 

$\ytableausetup{smalltableaux}
\begin{ytableau}
 1&1 \\
 1'&2\\
\end{ytableau} \hspace{.2 in}$: the associated word $1'211$ has $1'$ before all other $1$.

$\ytableausetup{smalltableaux}
\begin{ytableau}
 1&1' \\
 *(yellow) 2&2\\
\end{ytableau} \hspace{.2 in}$: yellow box violates the lattice condition (1).

$\ytableausetup{smalltableaux}
\begin{ytableau}
 *(yellow)1& 1 \\
2&2'\\
\end{ytableau} \hspace{.2 in}$: yellow box violates the lattice condition (2).\\

Using the above result, Bessenrodt \cite[Theorem~2.2]{Bessenrodt} classified all pairs of partitions $(\lambda,\mu)$ such that $f_{\lambda,\mu}^{\gamma}$ is either $0$ or $1$. Among them, two cases are important for us and they are as follows:

1) $\lambda=(t,t-1,\dots,1)$, a staircase shape, and $\mu=(s,0,\dots,0)$.

2) $\lambda$ is arbitrary and $\mu$ is a single box.

For Case 2, it is immediate that $f_{\lambda,\mu}^{\gamma}=1$ if and only if $\gamma$ can be obtained from adding a single box to $\lambda$. We will refer to this as the \emph{Pieri Rule} for tensoring with $V$. We reformulate Theorem \ref{Pre_3_Stembridge} for Case 1. Denote by $|\lambda|$ the total number of boxes in $\lambda$.

\begin{lemma}\label{Pre_5_com}
When $\lambda=(t,t-1,\dots,1,0,0,\dots)$, $\mu=(s)$, $f_{\lambda,\mu}^{\gamma}=1$ if and only if $\gamma$ is of the form
\begin{align*}
\gamma
&=\lambda+(s_0,1,1,\dots,1)\\
&= (t+s_0,t,t-1,\dots,j+1,j,j-2,j-3,\dots,1).
\end{align*}
and $|\gamma|=|\lambda|+|\mu|$. In other words, $\gamma$ is obtained by pasting an upside down L-shaped diagram to the right of $\lambda$. Otherwise, $f^{\gamma}_{\lambda,\mu}=0$.
\end{lemma}

\begin{proof}
By \cite[Theorem 8.3]{Stembridge}, $f^{\gamma}_{\lambda,\mu}\neq 0$ if and only if there exists a semistandard tableau of shape $\gamma/\lambda$ with content $\mu$, satisfying the conditions in this theorem. Since $\mu$ has a single row, these entries must be either $1$ or $1'$. On the other hand, the semistandard condition eliminates a $2$-by-$2$ square to appear in $\gamma/\lambda$, because the first row must be
\begin{ytableau}
 1' & 1\\
  & 
\end{ytableau}
and the lower right box cannot be filled with either $1$ or $1'$.

Therefore, boxes in $\gamma/\lambda$ must form an upside-down L-shape. On the other hand, given any such partition $\lambda+(s_0,1,1,\dots,1)$, the following tableau satisfies the condition:
\begin{center}
\ytableausetup{smalltableaux}
\begin{ytableau}
1' & 1 & \cdots & 1 & 1\\
1' \\
1' \\
\vdots\\
1' \\
1
\end{ytableau}
\end{center}
the word $w=11'1'\cdots1'11\cdots1$ can be checked to satisfy the condition in \cite[Theorem 8.3]{Stembridge}.
\end{proof}

Based on this result, we take $M=L(\alpha)$ with $\alpha=(n,n-1,n-2,\dots,1)$. We will refer to $\alpha$ as the staircase of height $n$ because it has shape similar to (\ref{staircase}). On the other hand, take $N=L(\beta)$ with $\beta=(p,0,0,\dots,0)$, for a fixed positive integer $p$. Pictorially $\beta$ is a single row of $p$ boxes. 

Polynomial representations for $\q(n)$ are known to be parameterized by strict partitions with length at most $n$, and if $\gamma$ contains $\alpha$, $\ell(\gamma)$ is at least $n$.  Therefore any summand $L(\gamma)$ in $M\otimes N\otimes V^{\otimes d}$ has $\ell(\gamma)=n$. Furthermore, $\ell(\lambda)=\ell(\mu)=n$ in equation (\ref{structuralconstant}), and if $\ell(\mu)=1$, $f^{\gamma}_{\lambda,\mu}=1$ then $m^{\gamma}_{\lambda,\mu}=2f^{\gamma}_{\lambda,\mu}=2$ for all $\gamma$. Hence the multiplicity of any irreducible summand in $L(\alpha)\otimes L(\beta)$, or any irreducible summand in $L(\lambda)\otimes V$ in successive tensor products, is always two.

\subsection{Exta relations in the kernel}
When $M$ and $N$ are taken to be specifc modules $L(\alpha)$ and $L(\beta)$, further relations may hold for the action of $\mathcal{H}_d$. As mentioned in \cite[Section 1.4]{Wang}, a $\q(n)$-module $L(\lambda)$ is of Type Q (meaning it admits an odd module endomorphism) if and only if $\ell(\lambda)$ is odd, and is of Type M if and only if $\ell(\lambda)$ is even. Since $\ell(\alpha)=n$ and $\ell(\beta)=1$, $L(\beta)$ is of Type Q. Recall the two algebras $\mathcal{H}_d^{+}$ and $\mathcal{H}_d^{++}$ defined in Section~\ref{2.4}, and $\mathcal{H}_d^{+}$ acts on $L(\alpha)\otimes L(\beta)\otimes V^{\otimes d}$ when $L(\beta)$ is of Type Q. Furthermore, if $n$ is odd, then $L(\alpha)$ is also of Type Q, and the larger algebra $\mathcal{H}_d^{++}$ also acts on $L(\alpha)\otimes L(\beta)\otimes V^{\otimes d}$. Recall that $p$ is the number of boxes in $\beta$, and $\tilde{y}_1\in \mathcal{H}_d$ is defined shortly before Proposition~\ref{extragenerators}.

\begin{theorem}\label{relations}
The following relations are satisfied by the action of $\mathcal{H}_d^{+}$ on $L(\alpha)\otimes L(\beta)\otimes V^{\otimes d}$, and are also satisfied by the action of $\mathcal{H}_d^{++}$ on $L(\alpha)\otimes L(\beta)\otimes V^{\otimes d}$ when $n$ is odd.
\begin{align*}
\tilde{x}_1^2-n(n+1)=&0,\\
\tilde{y}_1^2\big(\tilde{y}_1^2-p(p+1)\big)=&0.
\end{align*}
\end{theorem}

To prove this, we need a few lemmas about the odd Casimir tensor $\Omega$ defined in (\ref{oddomega}), using central elements in $U(\q(n))$ introduced by Sergeev in \cite{Sergeev2}. These are elements $x_{ij}(m)$ defined inductively on all $m\in \mathbb{Z}$, $m\geq 1$:
\begin{align*}
x_{ij}(1)&=e_{ij},\hspace{.2 in} x'_{ij}(1)=f_{ij}\hspace{.5 in}(1\leq i,j\leq n),\\
x_{ij}(m)&=\displaystyle\sum_{s=1}^{n}(e_{is}x_{sj}(m-1)+(-1)^{m-1}f_{is}x'_{sj}(m-1)),\\
x'_{ij}(m)&=\displaystyle\sum_{s=1}^{n}(e_{is}x'_{sj}(m-1)+(-1)^{m-1}f_{is}x_{sj}(m-1)).
\end{align*}
Sergeev's central elements are
\begin{align}
z_r=\displaystyle\sum_{i=1}^{n}x_{ii}(2r-1) \label{Quo_2_zdef}
\end{align}
for $r\geq 1$, $r\in \mathbb{Z}$. Notice that $z_r$ is a homogeneous polynomial in $e_{ij}$, $f_{ij}$ of degree $2r-1$, and $\{z_r\}$ provides a list of central elements of odd degrees.

\begin{rk}
For future convenience, here are the explicit formulas for elements of low ranks.
\begin{align*}
x_{ij}(2)&=\sum_{1\leq s\leq n}(e_{is}e_{sj}-f_{is}f_{sj}), \hspace{.4 in}x'_{ij}(2)=\sum_{1\leq s\leq n}(e_{is}f_{sj}-f_{is}e_{sj}),\\
x_{ij}(3)&=\sum_{1\leq s\leq n}(e_{is}x_{sj}(2)+f_{is}x'_{sj}(2))=\sum_{1\leq s,k\leq n}(e_{is}e_{sk}e_{kj}-e_{is}f_{sk}f_{kj}+f_{is}e_{sk}f_{kj}-f_{is}f_{sk}e_{kj}),\\
z_1&=\sum_{1\leq i\leq n}e_{ii}, \hspace{.4 in}z_2=\sum_{1\leq i\leq n}x_{ii}(3).
\end{align*}
\end{rk}

The following proposition will allow us to compute the action of $z_0,\dots,z_d$ on $L(\alpha)\otimes L(\beta)\otimes V^{\otimes d}$:

\begin{prop}\label{Quo_2_key}
Let $z_i$ be the elements defined in  $(\ref{Quo_2_zdef})$. The following is true
\begin{align*}
\Omega^2=\frac{1}{3}\big(\Delta(z_2)-z_2\otimes 1-1\otimes z_2+2 z_1\otimes z_1\big).
\end{align*}
\end{prop}

\begin{proof}
First,
\begin{align*}
z_2=\sum_{1\leq i,s,k\leq n}(e_{is}e_{sk}e_{ki}-e_{is}f_{sk}f_{ki}+f_{is}e_{sk}f_{ki}-f_{is}f_{sk}e_{ki}),\\
\end{align*}
where
\begin{align*}
 \Delta(\sum_{1\leq i,s,k\leq n}e_{is}e_{sk}e_{ki})&=\sum_{1\leq i,s,k\leq n}(e_{is}\otimes 1+1\otimes e_{is})(e_{sk}\otimes 1+1\otimes e_{sk})(e_{ki}\otimes 1+1\otimes e_{ki})\\
&=\sum_{1\leq i,s,k\leq n}(1\otimes e_{is}e_{sk}e_{ki}+ e_{is}\otimes e_{sk}e_{ki}+e_{sk}\otimes e_{is}e_{ki}+e_{ki}\otimes e_{is}e_{sk}+\\
&= e_{is}e_{sk}\otimes e_{ki}+e_{is}e_{ki}\otimes e_{sk} +e_{sk}e_{ki}\otimes e_{is} +  e_{is}e_{sk}e_{ki}\otimes 1,\\
-\Delta(\sum_{1\leq i,s,k\leq n} e_{is}f_{sk}f_{ki})&=-\sum_{1\leq i,s,k\leq n} (e_{is}\otimes 1 +1\otimes e_{is})(f_{sk}\otimes 1+1\otimes f_{sk})(f_{ki}\otimes 1+1\otimes f_{ki})\\
&=-\sum_{1\leq i,s,k\leq n} (e_{is}f_{sk}f_{ki}\otimes 1+ e_{is}f_{sk}\otimes f_{ki}-e_{is}f_{ki}\otimes f_{sk}+f_{sk}f_{ki}\otimes e_{is}\\
&\hh +e_{is}\otimes f_{sk}f_{ki}+f_{sk}\otimes e_{is}f_{ki}-f_{ki}\otimes e_{is}f_{sk}+1\otimes e_{is}f_{sk}f_{ki}  ),\\
\Delta(\sum_{1\leq i,s,k\leq n} f_{is}e_{sk}f_{ki})&= \sum_{1\leq i,s,k\leq n}(f_{is}\otimes 1+1\otimes f_{is})(e_{sk}\otimes 1+1\otimes e_{sk})(f_{ki}\otimes 1+1\otimes f_{ki})\\
&= \sum_{1\leq i,s,k\leq n}(f_{is}e_{sk}f_{ki}\otimes 1 +f_{is}e_{sk}\otimes f_{ki}+f_{is}f_{ki}\otimes e_{sk}-e_{sk}f_{ki}\otimes f_{is}\\
&\hh +f_{is}\otimes e_{sk}f_{ki}+e_{sk}\otimes f_{is}f_{ki}-f_{ki}\otimes f_{is}e_{sk} +1\otimes f_{is}e_{sk}f_{ki} ),\\
 -\Delta(\sum_{1\leq i,s,k\leq n} f_{is}f_{sk}e_{ki}))&= -\sum_{1\leq i,s,k\leq n}(f_{is}f_{sk}e_{ki}\otimes 1 +f_{is}f_{sk}\otimes e_{ki}+f_{is}e_{ki}\otimes f_{sk}-f_{sk}e_{ki}\otimes f_{is}\\
&\hh +f_{is}\otimes f_{sk}e_{ki}-f_{sk}\otimes f_{is}e_{ki}+e_{ki}\otimes f_{is}f_{sk}+1\otimes f_{is}f_{sk}e_{ki}).
\end{align*}
Therefore,
\begin{align*}
&\hh \Delta(z_2)-z_2\otimes 1 -1\otimes z_2\\
&=\sum_{1\leq i,s,k\leq n}( e_{is}e_{sk}\otimes e_{ki}+e_{is}e_{ki}\otimes e_{sk} +e_{sk}e_{ki}\otimes e_{is} -e_{is}f_{sk}\otimes f_{ki}+e_{is}f_{ki}\otimes f_{sk}-e_{sk}f_{ki}\otimes f_{is}\\
&\hh +f_{is}e_{sk}\otimes f_{ki}-f_{is}e_{ki}\otimes f_{sk}+f_{sk}e_{ki}\otimes f_{is} -f_{is}f_{sk}\otimes e_{ki}+f_{is}f_{ki}\otimes e_{sk}-f_{sk}f_{ki}\otimes e_{is}\\
&\hh -f_{is}\otimes f_{sk}e_{ki}+f_{sk}\otimes f_{is}e_{ki}-f_{ki}\otimes f_{is}e_{sk} +e_{is}\otimes e_{sk}e_{ki}+e_{sk}\otimes e_{is}e_{ki}+e_{ki}\otimes e_{is}e_{sk}\\
&\hh +f_{is}\otimes e_{sk}f_{ki}-f_{sk}\otimes e_{is}f_{ki}+f_{ki}\otimes e_{is}f_{sk} -e_{is}\otimes f_{sk}f_{ki}+e_{sk}\otimes f_{is}f_{ki}-e_{ki}\otimes f_{is}f_{sk} ).
\end{align*}
On the other hand, let us compute $\Omega^2$. All sums without indexing set are understood to be taken over $1\leq i,j,p,q\leq n$.
\begin{align*}
3\Omega^2
&=\sum_{1\leq i,j\leq n} 3(e_{ij}\otimes f_{ji}-f_{ij}\otimes e_{ji})\sum_{1\leq p,q\leq n} (e_{pq}\otimes f_{qp}-f_{pq}\otimes e_{qp})\\
&=\sum 3(e_{ij}e_{pq}\otimes f_{ji}f_{qp}+e_{ij}f_{pq}\otimes f_{ji}e_{qp}-f_{ij}e_{pq}\otimes e_{ji}f_{qp}+f_{ij}f_{pq}\otimes e_{ji}e_{qp})\\
&=\sum \frac{3}{2}(e_{ij}e_{pq}\otimes f_{ji}f_{qp}-e_{pq}e_{ij}\otimes f_{ji}f_{qp}+e_{pq}e_{ij}\otimes f_{ji}f_{qp}+e_{ij}e_{pq}\otimes f_{ji}f_{qp})\\
&\hh  + \sum 3(e_{ij}f_{pq}\otimes f_{ji}e_{qp}-f_{pq}e_{ij}\otimes f_{ji}e_{qp} +f_{pq}e_{ij}\otimes f_{ji}e_{qp} - f_{ij}e_{pq}\otimes e_{ji}f_{qp})\\
&\hh + \sum\frac{3}{2}(f_{ij}f_{pq}\otimes e_{ji}e_{qp}+f_{pq}f_{ij}\otimes e_{ji}e_{qp}-f_{pq}f_{ij}\otimes e_{ji}e_{qp} +f_{ij}f_{pq}\otimes e_{ji}e_{qp})\\
&=\sum \frac{3}{2}((\delta_{jp}e_{iq}-\delta_{iq}e_{pj}))\otimes f_{ji}f_{qp}+e_{pq}e_{ij}\otimes f_{ji}f_{qp}+e_{pq}e_{ij}\otimes f_{qp}f_{ji})\\
&\hh  + \sum 3((\delta_{jp}f_{iq}-\delta_{iq}f_{pj})\otimes f_{ji}e_{qp} +f_{pq}e_{ij}\otimes f_{ji}e_{qp} - f_{pq}e_{ij}\otimes e_{qp}f_{ji})\\
&\hh + \sum\frac{3}{2}((\delta_{jp}e_{iq}+\delta_{iq}e_{pj})\otimes e_{ji}e_{qp}-f_{pq}f_{ij}\otimes e_{ji}e_{qp} +f_{pq}f_{ij}\otimes e_{qp}e_{ji})\\
&=\sum_{i,p,q} \frac{3}{2}e_{iq}\otimes f_{pi}f_{qp}-\sum_{j,p,q}\frac{3}{2}e_{pj}\otimes f_{jq}f_{qp}+\sum\frac{3}{2}e_{pq}e_{ij}\otimes (\delta_{iq}e_{jp}+\delta_{jp}e_{qi})\\
&\hh  + \sum_{i,p,q} 3f_{iq}\otimes f_{pi}e_{qp}-\sum_{j,p,q} 3f_{pj}\otimes f_{jq}e_{qp}+\sum 3f_{pq}e_{ij}\otimes (\delta_{iq}f_{jp}-\delta_{jp}f_{qi}) \\
&\hh + \sum_{i,p,q}\frac{3}{2}e_{iq}\otimes e_{pi}e_{qp}+\sum_{j,p,q}\frac{3}{2}e_{pj}\otimes e_{jq}e_{qp} -\sum\frac{3}{2}f_{pq}f_{ij}\otimes (\delta_{iq}e_{jp}-\delta_{jp}e_{qi})\\
&=\sum_{i,p,q} \frac{3}{2}e_{iq}\otimes f_{pi}f_{qp}-\sum_{j,p,q}\frac{3}{2}e_{pj}\otimes f_{jq}f_{qp}+\sum_{j,p,q}\frac{3}{2}e_{pq}e_{qj}\otimes e_{jp} +\sum_{i,p,q}\frac{3}{2}e_{pq}e_{ip}\otimes e_{qi}\\
&\hh  + \sum_{i,p,q} 3f_{iq}\otimes f_{pi}e_{qp}-\sum_{j,p,q} 3f_{pj}\otimes f_{jq}e_{qp}+\sum_{j,p,q} 3 f_{pq}e_{qj}\otimes f_{jp}-\sum_{i,p,q} 3f_{pq}e_{ip}\otimes f_{qi} \\
&\hh + \sum_{i,p,q}\frac{3}{2}e_{iq}\otimes e_{pi}e_{qp}+\sum_{j,p,q}\frac{3}{2}e_{pj}\otimes e_{jq}e_{qp} -\sum_{j,p,q}\frac{3}{2}f_{pq}f_{qj}\otimes e_{jp}+\sum_{i,p,q}\frac{3}{2}f_{pq}f_{ip}\otimes e_{qi}.
\end{align*}
Comparing the two results,
\begin{align*}
&\hh 3\Omega^2-(\Delta(z_2)-z_2\otimes 1 -1\otimes z_2)\\
&= \sum_{i,p,q} \frac{1}{2}(e_{iq}\otimes f_{pi}f_{qp}+ e_{iq}\otimes f_{qp}f_{pi})+\sum_{j,p,q}\frac{1}{2}(e_{pq}e_{jp}\otimes e_{qj}-e_{jp}e_{pq}\otimes e_{qj})\\
&\hh +\sum_{i,p,q} 2(f_{iq}\otimes f_{pi}e_{qp}-f_{iq}\otimes e_{qp}f_{pi})-\sum_{j,p,q} (f_{pj}\otimes f_{jq}e_{qp}-f_{pj}\otimes  e_{qp}f_{jq})\\
&\hh +\sum_{j,p,q} ( f_{pq}e_{qj}\otimes f_{jp} -e_{qj}f_{pq}\otimes f_{jp})-\sum_{i,p,q} 2(f_{pq}e_{ip}\otimes f_{qi}-e_{ip}f_{pq}\otimes f_{qi})\\
&\hh + \sum_{i,p,q}\frac{1}{2}(e_{iq}\otimes e_{pi}e_{qp}-e_{iq}\otimes e_{qp}e_{pi})+\sum_{j,p,q}\frac{1}{2}(f_{pq}f_{qj}\otimes e_{jp}+f_{qj}f_{pq}\otimes e_{jp})\\
&= \sum_{i,p,q} \frac{1}{2} e_{iq}\otimes (\delta_{iq}e_{pp}+e_{qi})+\sum_{j,p,q}\frac{1}{2}(\delta_{qj}e_{pp}-e_{jq})\otimes e_{qj}\\
&\hh +\sum_{i,p,q} 2f_{iq}\otimes (\delta_{iq}f_{pp}-f_{qi})-\sum_{j,p,q} f_{pj}\otimes (f_{jp}-\delta_{pj}f_{qq})\\
&\hh +\sum_{j,p,q}  (f_{pj}-\delta_{pj}f_{qq})\otimes f_{jp} -\sum_{i,p,q} 2(\delta_{qi}f_{pp}-f_{iq})\otimes f_{qi}\\
&\hh + \sum_{i,p,q}\frac{1}{2}e_{iq}\otimes (\delta_{iq}e_{pp}-e_{ip})+\sum_{j,p,q}\frac{1}{2}(e_{pj}+\delta_{pj}e_{qq})\otimes e_{jp}\\
&= 2\sum_{i,p}e_{ii}\otimes e_{pp} =2(\sum_{i}e_{ii})\otimes(\sum_p e_{pp})=2z_1\otimes z_1.
\end{align*}
\end{proof}

Brundan-Kleshchev calculated the action of $z_i$ on a highest weight vector of weight $\lambda=(\lambda_1,\dots,\lambda_n)$:

\begin{theorem}\cite[Lemma 8.4]{Brundan2}\label{Quo_2_z}
Let $M$ be a $\q(n)$-modules and $v_{\lambda}\in M_{\lambda}$ be a vector annihilated by $e_{ij}$ and $f_{ij}$, $\forall 1\leq i<j \leq n$, then $z_r.v_{\lambda}=z_r(\lambda)v$, where
\begin{align*}
z_r(\lambda)=\displaystyle\sum (-2)^{s+1}\lambda_{i_1}\cdots\lambda_{i_s}(\lambda_{i_1}^2-\lambda_{i_1})^{a_1}\cdots (\lambda_{i_s}^2-\lambda_{i_s})^{a_s},
\end{align*}
and the sum is taken over all $1\leq s\leq r$, $1\leq i_1<\cdots<i_s\leq n$, $a_i\in \mathbb{Z}_{\geq 0}$, $a_1+\cdots+a_s=r-s$.
\end{theorem}
\begin{rk}
Left multiplication by the element $z_r$ induces an even $\q(n)$-endomorphism on $L(\lambda)$, and hence acts by a scalar on $L(\lambda)$ according to super Schur's Lemma in Lemma~\ref{Pre_3_superschur}. This is the same scalar by which $z_r$ acts on a highest weight vector $v_{\lambda}\in L(\lambda)_{\lambda}$.
\end{rk}
For future use we write down the explicit formula for $z_1(\lambda)$ and $z_2(\lambda)$:
\begin{align*}
z_1(\lambda)&=\lambda_1+\cdots\lambda_n, \hspace{.4 in}z_2(\lambda)=(\lambda_1^3+\cdots \lambda_n^3)-(\lambda_1+\cdots+\lambda_n)^2.
\end{align*}

For a box $b$ in a shifted Young diagram, denote by $c(b)$ the content of $b$:
\begin{align}\label{content}
c(b)=\text{column of } b - \text{row of } b.
\end{align}

Recall the Pieri Rule introduced shortly before Lemma~\ref{Pre_5_com}, which describes the irreducible summands in $L(\lambda)\otimes V$. The following result shows that we can use the action of $\Omega^2$ to distinguish various isotypic components from each other in the decomposition.

\begin{cor}\label{Piericontent}
Let $L(\gamma)$ be an irreducible summand of $L(\lambda)\otimes V$, and let $b$ be the unique box in $\gamma$ that does not belong to $\lambda$. Then $\Omega^2$ acts on $L(\gamma)$ by the scalar $c(b)(c(b)+1)$.
\end{cor}
\begin{proof}
By Proposition~\ref{Quo_2_key}, $\Omega^2$ acts as $\frac{1}{3}(\Delta(z_2)-z_2\otimes 1-1\otimes z_2+2 z_1\otimes z_1)$. Also recall that $z_i$ acts on $L(\lambda)$ by the scalar $z_i(\lambda)$ in Theorem~\ref{Quo_2_z}, and $V=L(\epsilon_1)$. Assume $\gamma=\lambda+\epsilon_i$ for some $1\leq i\leq n$. 
In particular, let $|\lambda|=\lambda_1+\cdots \lambda_n$,
\begin{align*}
z_1(\lambda)&=|\lambda|, \hspace{.2 in} z_2(\lambda)=\lambda_1^3+\cdots \lambda_n^3 -|\lambda|^2, \hspace{.2 in} z_1(\gamma)=|\lambda|+1,\\
z_2(\gamma)&=\lambda_1^3+\cdots+(\lambda_i+1)^2+\cdots+ \lambda_n^3 -(|\lambda|+1)^2\hspace{.2 in} z_1(\epsilon_1)=1, \hspace{.2 in} z_2(\epsilon_1)=0.
\end{align*}
Therefore, $\Omega^2$ acts as the scalar
\begin{align*}
&\hh \frac{1}{3}(z_2(\gamma)-z_2(\lambda)-z_2(\epsilon_1)+2 z_1(\lambda) z_1(\epsilon_1))\\
&=\frac{1}{3}((\lambda_1^3+\cdots+(\lambda_i+1)^2+\cdots+ \lambda_n^3 -(|\lambda|+1)^2)-(\lambda_1^3+\cdots \lambda_n^3 -|\lambda|^2))+2|\lambda|=\lambda_i (\lambda_i+1). 
\end{align*}
and the last equality holds because the content of the first box in any row is zero, and the content of the added box is equal to the number of boxes in the $i$-th row of $\lambda$.
\end{proof}

Recall that $\alpha$ is the staircase Young diagram with length $n$ and $\beta$ is a single row of $p$ boxes. In Lemma~\ref{Pre_5_com} we gave a description of all partitions $\lambda$ such that $L(\lambda)$ is an irreducible summand of $L(\alpha)\otimes L(\beta)$. We can also use the action of $\Omega^2$ to distinguish these summands from each other. This following result will motivate some of the definitions in Section~\ref{4.2}.
\begin{cor}\label{x1squareacts}
Let $L(\lambda)\subset L(\alpha)\otimes L(\beta)$ be an irreducible summand. Then $\Omega^2$ acts on $L(\lambda)$ via the scalar $mp(m-p)$, where $m$ is the number of boxes in the first row of $\lambda$.
\end{cor}
\begin{proof}
Let $s=m-n$, Observe that
\begin{align*}
z_1(\alpha)&=\frac{n(n+1)}{2}, \hspace{.2 in} z_2(\alpha)=(1^3+2^3+\cdots+n^3)-(1+2+\dots+n)^2=0,\\
z_1(\beta)&=p, \hspace{.2 in} z_2(\beta)=p^3-p^2,\\
\gamma&=\alpha+(s,1,1,\dots,1)=(n+s,n,n-1,\dots, n-p+s+1,n-p+s-1,\dots,1),\\
z_1(\gamma)&=\frac{n(n+1)}{2}+p, \hspace{.2 in} z_2(\gamma)=(n+s)^3+(1^3+2^3+\cdots+n^3) -(n-p+s)^3-\left(\frac{n(n+1)}{2}+p\right)^2.
\end{align*}
Therefore, by Proposition~\ref{Quo_2_key}, for any $v\in L(\gamma)$, $\Omega^2$ acts by the scalar
\begin{align*}
&\hh \frac{1}{3}(z_2(\gamma)-z_2(\alpha)-z_2(\beta)+2z_1(\alpha)z_1(\beta))\\
&=\frac{1}{3}((n+s)^3+(1^3+2^3+\cdots+n^3) -(n-p+s)^3 \\
&-\left(\frac{n(n+1)}{2}+p)^2 -(p^3-p^2)+2\cdot \frac{n(n+1)}{2}\cdot p\right)\\
&= \frac{1}{3}((n+s)^3-(n-p+s)^3-p^3)= \frac{1}{3}(m^3-(m-p)^3+p^3)\\
&= \frac{1}{3}(3m^2p-3mp^2)= mp(m-p).
\end{align*}
\end{proof}

\begin{proof}[Proof of Theorem~\ref{relations}]
Using the isomorphism $L(\alpha)\otimes L(\beta)\otimes V^{\otimes d}\simeq L(\alpha)\otimes V^{\otimes d} \otimes L(\beta)$, the action of $\tilde{x}_1^2$ is equivalent to the action of $\Omega^2$ on the first two tensor factors $L(\alpha)\otimes V$. By the Pieri Rule, an irreducible summand $L(\gamma)$ is parameterized by $\gamma$ which is obtained by adding a box to $\alpha$. There is only one way to add such a box, namely at the end of the first row with content $n$. By Corollary~\ref{Piericontent}, $\tilde{x}_1^2$ acts as  $n(n+1)$.

Based on the additional formula provided in Corollary~\ref{xiyiact}, $\tilde{y}_1$ acts by $\Omega_{N,1}$, which is applying the two factors in $\Omega$ to $N=L(\beta)$ and the first copy of $V$, respectively. 
Since $\Omega^2$ acts by a scalar on each irreducible summand $W$ of $L(\beta)\otimes V$, the action of $\tilde{y}_1^2$ is a scalar on each $M\otimes W \otimes V^{\otimes d-1}$. By the Pieri Rule, an irreducible summand $L(\gamma)$ is parameterized by $\gamma$ which is obtained by adding a box to $\beta$. There are two ways to add such a box: to the end of the first row or the beginning of the second row, whose added box has content $p$ or $0$. By Corollary~\ref{Piericontent}, $\tilde{y}_1^2$ acts as either $p(p+1)$ or $0$.

\end{proof}

\subsection{A quotient of $\mathcal{H}^{+}_d$ or $\mathcal{H}^{++}_d$}

Recall the algebras $\mathcal{H}^{+}_d$ and $\mathcal{H}^{++}_d$ defined in Section~\ref{2.4}. Let $\mathcal{H}^{\operatorname{ev}}_{p,d}$ be a quotient of $\mathcal{H}^{+}_d$ under the extra relations in Theorem~\ref{relations}. Also, let $\mathcal{H}^{\operatorname{od}}_{p,d}$ be a quotient of $\mathcal{H}^{++}_d$ under the extra relations in Theorem~\ref{relations}. By Theorem~\ref{relations}, $\mathcal{H}^{\operatorname{ev}}_{p,d}$ acts on $L(\alpha)\otimes L(\beta)\otimes V^{\otimes d}$. Furthermore, when $n$ is odd, both $\mathcal{H}^{\operatorname{ev}}_{p,d}$ and $\mathcal{H}^{\operatorname{od}}_{p,d}$ act on $L(\alpha)\otimes L(\beta)\otimes V^{\otimes d}$. As we will see in Section~\ref{sec45}, the subtle difference in notation signals the fact that when $n$ is even, $\mathcal{H}^{\operatorname{ev}}_{p,d}$ provides $\mathfrak{q}(n)$-module endomorphisms on $L(\alpha)\otimes L(\beta)\otimes V^{\otimes d}$; where as when $n$ is odd,  both $\mathcal{H}^{\operatorname{ev}}_{p,d}\subset \mathcal{H}^{\operatorname{od}}_{p,d}$ provide $\mathfrak{q}(n)$-module endomorphisms, but we will focus on the larger algebra in light of future study of double-centralizer properties. In other words, $\mathcal{H}^{\operatorname{od}}_{p,d}$ is more likely to be the full centralizer of $\q(n)$ on $L(\alpha)\otimes L(\beta)\otimes V^{\otimes d}$ in the case when $n$ is odd.
 
For later use, it is more convenient to introduce presentations of these two quotients with even polynomial generators. We define even elements in $\mathcal{H}^{\operatorname{ev}}_{p,d}$ and $\mathcal{H}^{\operatorname{od}}_{p,d}$ via $z_0=\tilde{z}_0c_0$, $z_1=\tilde{z}_1c_1$, \dots, $z_d=\tilde{z}_dc_d$ and $x_1=\tilde{x}_1c_1$, and rewrite the relations as follows.

\begin{lemma}\label{evenpresentation}
When $n$ is even, the algebra $\mathcal{H}^{\operatorname{ev}}_{p,d}$ is isomorphic to the algebra generated by $x_1,z_0,\dots,z_d$, $c_0,c_1,\dots,c_d$ and $s_1,\dots,s_{d-1}$, subject to the Sergeev relations and the following relations

\noindent(Hecke relations)
\begin{align}
s_i {z}_i &= {z}_{i+1}s_i-1+c_ic_{i+1}  &(1\leq i\leq d-1), \label{heckerelation1}\\
{x}_1 s_i &= s_i {x}_1 &(2\leq i\leq n),\\
{z}_j s_i &= s_i {z}_j &(j \neq i,i+1). \label{heckerelation3}
\end{align}
(Clifford twist relations)
\begin{align}
x_1 c_1&= -c_1x_1, \hspace{.4 in} c_i{x}_1 = {x}_1 c_i \hspace{.1 in} (i=0,2,3,\dots,d), \\
z_i c_i&= -c_iz_i \hspace{.1 in} (0\leq i\leq d),  \hspace{.4 in} c_i{z}_j = {z}_j c_i  \hspace{.1 in} (j\neq i).
\end{align}
(Polynomial relations)
\begin{align}
{x}_1(s_1{x}_1s_1+(1-c_1c_2)s_1 ) &=(s_1{x}_1s_1+(1-c_1c_2)s_1 ){x}_1,  &\\
{z}_1 {z}_2 &={z}_2{z}_1. \label{zscommute}&
\end{align}
(Relations between polynomial rings)
\begin{align}
{z}_2 {x}_1 &= {x}_1 {z}_2, \label{between1}\\
({z}_0c_0c_1+{z}_1-{x}_1){x}_1 &= -{x}_1({z}_0c_0c_1+{z}_1-{x}_1).\label{importantrelation}
\end{align}
(Additional Sergeev relations for $c_0$)
\begin{align*}
c_0^2&=-1, \hspace{.4 in} c_0c_i=-c_ic_0 \hspace{.1 in}(1\leq i\leq d) \hspace{.4 in} c_0s_i=s_ic_0  \hspace{.1 in}(1\leq i\leq d-1).
\end{align*}
(Extra relations)
\begin{align}
x_1^2-n(n+1)=&0, \label{xeigenvalues}\\
(z_1-x_1)^2\big((z_1-x_1)^2-p(p+1)\big)=&0. \label{yeigenvalues}
\end{align}
\end{lemma}

\begin{lemma}
The algebra $\mathcal{H}^{\operatorname{od}}_{p,d}$ is isomorphic to the algebra generated by $x_1,z_0,\dots,z_d$, $c_0,c_1,\dots,c_d$, $s_1,\dots,s_{d-1}$ and the extra generator $c_M$, subject to the relations in Lemma~\ref{evenpresentation} and extra relations 
\begin{align*}
c_M^2&=-1, \hspace{.4 in} c_Mc_i=-c_ic_M \hspace{.1 in}(0\leq i\leq d), \hspace{.4 in} c_Ms_i=s_ic_M  \hspace{.1 in} (1\leq i\leq d-1),\\
c_Mx_1&=x_1c_M,  \hspace{.4 in} c_Mz_i=z_ic_M \hspace{.1 in} (0\leq i\leq d).
\end{align*}
\end{lemma}

\section{Calibrated modules for $\mathcal{H}^p_d$}
\subsection{The Bratteli graph} \label{4.1}
The goal of this section is to construct certain types of modules for $\mathcal{H}^p_d$ using combinatorial data. Recall in Lemma~\ref{Pre_5_com} we described all partitions $\lambda$ for which $L(\lambda)$ is a direct summand of $L(\alpha)\otimes L(\beta)$, and the Pieri rule describing all irreducible summands of $L(\mu)\otimes V$. Fix positive integers $n$ and $p$, which determines the partitions $\alpha$ and $\beta$. Define the Bratteli graph $\Gamma$ associated to $n$ and $p$, to be the directed graph with vertices in row $i\in \{-1,0,1,2,3,\dots\}$. Row $-1$ contains an only partition $\alpha$, Row $0$ contains the following partitions
\begin{align*}
\mathcal{P}_0=\mathcal{P}_0(\alpha,\beta)=\{\lambda \hspace{.1 in}|\hspace{.1 in} L(\lambda)\text{ is a summand of }L(\alpha)\otimes L(\beta)\}.
\end{align*}
For $i\geq 1$, define the vertices in Row $i$ recursively as follows:
\begin{align*}
\mathcal{P}_i=\mathcal{P}_i(\alpha,\beta)=\{\lambda \hspace{.1 in}|\hspace{.1 in} L(\lambda)\text{ is a summand of }L(\mu)\otimes V, \text{ for some }\mu \text{ in Row }i-1\}.
\end{align*}
Edges in $\Gamma$ are defined as follows. There is an edge from $\alpha$ to every partition $\lambda \in \mathcal{P}_0$. For $i\geq 0$, there is an edge from $\mu\in \mathcal{P}_{i}(\alpha,\beta)$ to $\lambda\in \mathcal{P}_{i+1}(\alpha,\beta)$, if and only if $L(\lambda)$ is a summand of $L(\mu)\otimes V$.

For example, when $n=5$, $p=3$,
\ytableausetup{smalltableaux}
\begin{align*}
\alpha=\begin{ytableau}
 *(white)& & & &\\
 \none & &&&\\
 \none &\none &&&\\
 \none &\none &\none &&\\
 \none &\none &\none &\none &\\
\end{ytableau}\hspace{.5 in} \beta=\begin{ytableau}
*(white) & &\\
\end{ytableau}.
\end{align*}
Partitions in $\mathcal{P}_0(\alpha,\beta)$:
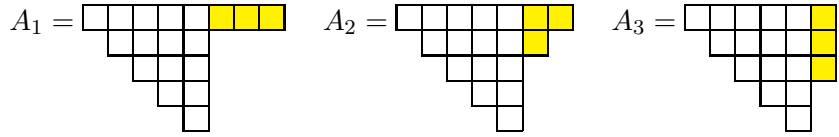
\begin{align*}
A_1=\begin{ytableau}
 *(white)& & & & & *(yellow)& *(yellow) & *(yellow)\\
 \none &&&&\\
  \none &\none &&&\\
 \none &\none &\none &&\\
  \none &\none &\none &\none &\\
\end{ytableau} \hspace{.2 in}
A_2=\begin{ytableau}
 *(white)& & & & & *(yellow)&*(yellow) \\
 \none &&&&&*(yellow)\\
  \none &\none &&&\\
\none &\none &\none & &\\
\none &\none &\none &\none & \\
\end{ytableau}\hspace{.2 in}
A_3=\begin{ytableau}
 *(white)& & & & &*(yellow)  \\
 \none &&&&&*(yellow)\\
 \none &\none &&&&*(yellow)\\
\none &\none &\none & &\\
 \none &\none &\none &\none &\\
\end{ytableau}
\end{align*}
For the remaining partitions we omit the staircase portion $\alpha$ and only display the yellow portion. Partitions in $\mathcal{P}_1(\alpha,\beta)$:
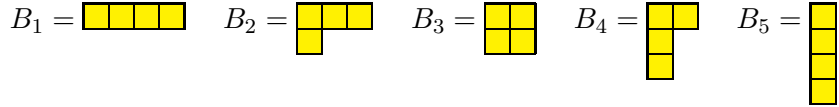
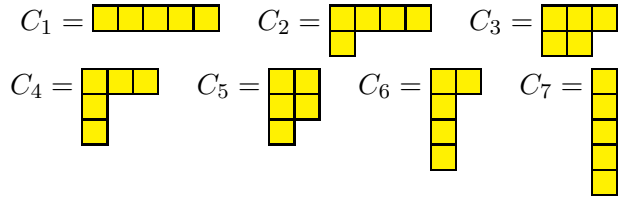
\begin{align*}
B_1=\begin{ytableau}
*(yellow) & *(yellow) &*(yellow) &*(yellow) \end{ytableau}\hspace{.2 in}
B_2=\begin{ytableau}
*(yellow) & *(yellow) &*(yellow) \\
*(yellow) \end{ytableau} \hspace{.2 in}
B_3=\begin{ytableau}
*(yellow) & *(yellow)  \\
*(yellow) &*(yellow) \end{ytableau}\hspace{.2 in}
B_4=\begin{ytableau}
*(yellow) & *(yellow)  \\
*(yellow) \\
*(yellow) \end{ytableau}\hspace{.2 in}
B_5=\begin{ytableau}
*(yellow)  \\
*(yellow) \\
*(yellow) \\
*(yellow) \end{ytableau}
\end{align*}
Partitions in $\mathcal{P}_2(\alpha,\beta)$:
\begin{align*}
C_1=\begin{ytableau}
*(yellow) & *(yellow) &*(yellow) &*(yellow) & *(yellow) \end{ytableau}\hspace{.2 in}
C_2=\begin{ytableau}
*(yellow) & *(yellow) &*(yellow) &*(yellow) \\
 *(yellow) \end{ytableau}\hspace{.2 in}
 C_3=\begin{ytableau}
*(yellow) & *(yellow) &*(yellow) \\
 *(yellow) &*(yellow) \end{ytableau}\hspace{.2 in}\\
C_4=\begin{ytableau}
*(yellow) & *(yellow) &*(yellow) \\
 *(yellow) \\
 *(yellow) \end{ytableau}\hspace{.2 in} 
 C_5=\begin{ytableau}
*(yellow) & *(yellow)  \\
 *(yellow) &*(yellow)\\
 *(yellow) \end{ytableau}\hspace{.2 in}
 C_6=\begin{ytableau}
*(yellow) & *(yellow)  \\
 *(yellow) \\
 *(yellow) \\
 *(yellow)\end{ytableau}\hspace{.2 in}
 C_7=\begin{ytableau}
*(yellow)  \\
 *(yellow) \\
 *(yellow) \\
 *(yellow)\\
 *(yellow) \end{ytableau}\hspace{.2 in}
\end{align*}

The associated Bratteli graph starts with the following as the first few rows: 
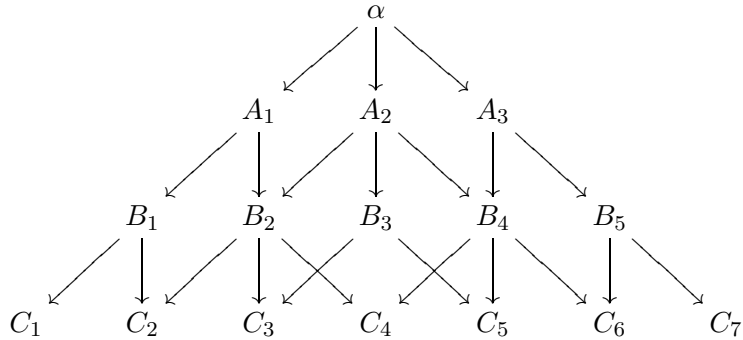
\begin{align*}
\xymatrix{
&&& \alpha    \ar[dl]\ar[d]\ar[dr] &&&\\
 && A_1 \ar[dl] \ar[d] & A_2 \ar[dl]\ar[d]\ar[dr] & A_3 \ar[d]\ar[dr]&&\\
& B_1 \ar[dl]\ar[d] & B_2 \ar[dl]\ar[d]\ar[dr] & B_3 \ar[dl]\ar[dr] & B_4 \ar[dl] \ar[d]\ar[dr] & B_5\ar[d] \ar[dr] &\\
 C_1 & C_2 & C_3 & C_4 & C_5 & C_6 & C_7 
}
\end{align*}

Fix a partition $\lambda \in \mathcal{P}_d(\alpha,\beta)$. Let $\Gamma^{\lambda}$ be the set of all directed paths from $\alpha$ to $\lambda$. Suppose a path $T\in \Gamma^{\lambda}$ travels through the vertices $\alpha, T^{(0)},T^{(1)},\dots,T^{(d)}$. It is helpful to record this information using a tableaux of skew shape $\lambda/T^{(0)}$. The position of integer $i$ indicates of the distinct box in $T^{(i)}$ that is not in $T^{(i-1)}$. For example, the path $T:\alpha\to A_2\to B_3 \to C_5$ above is represented by the following tableaux
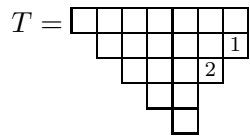
\begin{align*}
T=\begin{ytableau}
 *(white)& & & & & & \\
 \none &&&&&&1\\
  \none &\none &&&&2\\
\none &\none &\none & &\\
\none &\none &\none &\none & \\
\end{ytableau}.
\end{align*}

A Young tableau of skew shape with each of the entries $1,2,\dots,d$ exactly is said to satisfy the \emph{standard} condition, if the numbers increase along each row and column. A tableau obtained from a path $T$ will automatically satisfy the standard condition, based on the allowed positions of adding a box. Conversely, any standard Young tableau of skew shape $\lambda/\mu$ for some $\mu\in \mathcal{P}_0(\alpha,\beta)$, corresponds to a path $T\in \Gamma^{\lambda}$.

For future use we also need to define actions of $s_0,\dots,s_{d-1}$ on $\Gamma^{\lambda}\cup\{\star\}$, the set of paths enlarged by a symbol $\star$. We first have a corollary as a result of Lemma~\ref{Pre_5_com}. For two partitions $\mu=(\mu_1,\dots,\mu_n)$ and $\gamma=(\gamma_1,\dots,\gamma_n)$, we say $\mu$ is contained in $\gamma$ if and only if $\mu_i\leq \gamma_i$ for $1\leq i\leq n$.
\begin{cor}\label{twopaths}
For any partition $\mu \in \mathcal{P}_1(\alpha,\beta)$, there are at most two partitions in $\mathcal{P}_0(\alpha,\beta)$ contained in $\mu$. Furthermore, if there exists two such partitions, their first rows differ by one box.
\end{cor}
\begin{proof}
This is a direct consequence of Lemma~\ref{Pre_5_com}. If the skew shape $\mu/\alpha$ contains a $2$-by-$2$ square, then there is a unique $\lambda\subset \mu$ which satisfies the condition in Lemma~\ref{Pre_5_com}, by removing the lower right box in the $2$-by-$2$ square. On the other hand, if $\mu/\alpha$ does not contain such a square, then this skew shape forms an upside-down L, and a box can be removed from either far end of the vertical or horizontal edge, resulting in two partitions $\lambda_1,\lambda_2\in \mathcal{P}_0(\alpha,\beta)$ which are contained in $\mu$.
\end{proof}
As a result, for any path $T\in \Gamma^{\lambda}$, there is at most one other path which differs from $T$ at a vertex in Row $0$. If such a path $S$ exists, define $s_0.T=S$. If there is no such a path, let $s_0.T=\star$ by definition.

\textit{Example.} Let $T^{(0)}\in \mathcal{P}_0(\alpha,\beta)$ be the target of the first edge in $T$. For the following tableaux, we use green to highlight the boxes that are in $T^{(0)}$ but are not in $\alpha$. Then  $s_0.T=L$:
\begin{align*}
T=\begin{ytableau}
*(white)&&&&*(green)&*(green)&1&3 \\
\none & &&&*(green)&2\\
\none &\none &&&4\\
\none&\none& \none &
\end{ytableau}  \hspace{.5 in}
L=\begin{ytableau}
*(white)&&&&*(green)&*(green)&*(green)&3 \\
\none & &&&1&2\\
\none &\none &&&4\\
\none&\none& \none &
\end{ytableau} 
\end{align*}
For the following tableau, $s_0.T=\star$:
\begin{align*}
T=\begin{ytableau}
*(white)&&&&*(green)&*(green)&2&3 \\
\none & &&&*(green)&1\\
\none &\none &&&4\\
\none&\none& \none &
\end{ytableau}
\end{align*}

Similarly, for $1\leq i\leq d-1$, define $s_i.T$ to be the tableau by swapping the entries $i$ and $i+1$, if the resulting tableau is still standard. In terms of paths, this corresponds to reversing the order of adding the two boxes containing $i$ and $i+1$, and the two paths $T$ and $s_i.T$ share all other vertices except the one at Row $i$. If the resulting tableau is no longer standard, let $s_i.T=\star$ by definition.

\subsection{Constructing calibrated modules}\label{4.2} We define a module for $\mathcal{H}^{\operatorname{ev}}_{p,d}$ or $\mathcal{H}^{\operatorname{od}}_{p,d}$ to be \emph{calibrated} if it admits a basis on which the generators $z_0,\dots,z_d$ act by eigenvalues. We now construct an explicit class of calibrated modules for $\mathcal{H}^{\operatorname{ev}}_{p,d}$ and $\mathcal{H}^{\operatorname{od}}_{p,d}$.  Fix a partition $\lambda$ in Row $d$ of the Bratteli graph $\Gamma$, we define scalars based on a given path $T\in \Gamma^{\lambda}$. Recall the content of a box defined in (\ref{content}). Denote by $c_T(i)$ the content of the box containing integer $i$ in the tableau $T$. Recall $T^{(0)}$ is the partition in Row $0$ of path $T$, and let $m$ be the number of boxes in the first row of $T^{(0)}$. Define
\begin{align}\label{kappadefinition}
\kappa_T(0)=\sqrt{mp(m-p)}, \hspace{.3 in}
\kappa_T(i)=\sqrt{c_T(i)(c_T(i)+1)} \hspace{.5 in} (1\leq i\leq d).
\end{align}

Now let $N_0=n(n+1)$, and let $\Gamma^{\lambda}_0$ be the set of paths $T\in \Gamma^{\lambda}$ such that $s_0.T\neq \star$. Fix a function $f:\Gamma^{\lambda}_0\to \mathbb{C}$ satisfying the following condition:
\begin{align}
\hh f(T)f(s_0.T)=-\frac{(\kappa_T^2(0)+p^2\kappa^2_T(1))(N_0-\kappa^2_T(1))(N_0-\kappa^2_{s_0.T}(1))}{(\kappa^2_T(0)+p\kappa^2_T(1))^2((\kappa_T(0)-\kappa_{s_0.T}(0))^2+(\kappa_T(1)+\kappa_{s_0.T}(1))^2)}=:F_T. \label{Cal_2_f}
\end{align}
We will show the right hand side remains unchanged if one replaces $T$ with $s_0.T$, hence such a function exists. Let $\operatorname{Cl}_{d+1}$ be the subalgebra of $\mathcal{H}^{\operatorname{ev}}_{p,d}$ generated by $c_0,\dots,c_d$, and  $\operatorname{Cl}_{d+2}$ be the subalgebra of $\mathcal{H}^{\operatorname{od}}_{p,d}$ generated by $c_M,c_0,\dots,c_d$.

1) Let $\mathcal{D}^{\lambda}_f$ be the free module over $\operatorname{Cl}_{d+1}$ with basis $\{v_T\}_{T\in \Gamma^{\lambda}}$.

2) Let $\mathcal{E}^{\lambda}_f$ be the free module over $\operatorname{Cl}_{d+2}$ with basis $\{v_T\}_{T\in \Gamma^{\lambda}}$.

We impose the $\mathbb{Z}_2$-grading on $\mathcal{D}^{\lambda}_f$ and $\mathcal{E}^{\lambda}_f$ as follows: each vector $v_T$ has the same parity as the number of boxes in the first row of $T^{(0)}$. Extend this grading to all of $\mathcal{D}^{\lambda}_f$ and $\mathcal{E}^{\lambda}_f$, so that multiplication by a Clifford generator reverses the grading.

We now define an $\mathcal{H}^{\operatorname{ev}}_{p,d}$-module structure on $\mathcal{D}^{\lambda}_f$ and an $\mathcal{H}^{\operatorname{od}}_{p,d}$-module structure on $\mathcal{E}^{\lambda}_f$. Using Relations (\ref{cliffordrelations}, \ref{Ctwist1}, \ref{Ctwist2}), the remaining generators $z_0,\dots,z_d,x_1,s_1,\dots,s_{d-1}$ can be moved past the Clifford algebra, and act on $v_T$ directly. Therefore it is enough to define the action on each $v_T$. In the case of either $\mathcal{D}_f$ or $\mathcal{E}^{\lambda}_f$, generators $z_i$ act by $z_i.v_T=\kappa_T(i)v_T$ for $0\leq i\leq d$.

Declare $v_{\star}=0$, and the simple transpositions act by the following.
\begin{align}
s_i. v_T &= (-\frac{1}{\kappa_T(i)-\kappa_T(i+1)}+\frac{1}{\kappa_T(i)+\kappa_T(i+1)}c_ic_{i+1})v_T \notag \\
&\hh +\sqrt{1-\frac{1}{(\kappa_T(i)+\kappa_T(i+1))^2}-\frac{1}{(\kappa_T(i)-\kappa_T(i+1))^2}}v_{s_i.T}. \label{transpositionsaction}
\end{align}
Lastly, to define the action of $x_1$, let $\kappa=\sqrt{\kappa^2_T(0)+\kappa^2_T(1)}$. For $a,b\in \mathbb{C}$, let $D(a,b)$ be the following matrix
\begin{align}
D(a,b)=\begin{bmatrix}
a & b \\ b & -a
\end{bmatrix}, \label{CliffordD}
\end{align}

 and we introduce the following $2\times 2$ matrices for the path $T$:
\begin{align}
Z=&D(
\kappa_T(0)-\kappa_{s_0.T}(0), \kappa_T(1)+\kappa_{s_0.T}(1)),
\label{zdefinition}\\
A=\frac{N_0}{\kappa_T(0)^2+p\kappa_T(1)^2}&D(
p\kappa_T(1), -\kappa_T(0)) + \frac{\kappa_T(0)\kappa_T(1)}{\kappa_T(0)^2+p\kappa_T(1)^2} D(
\kappa_T(0) , \kappa_T(1)),
\\
 B= \frac{N_0}{\kappa_{s_0.T}(0)^2+p\kappa_{s_0.T}(1)^2}&D(
p\kappa_{s_0.T}(1), \kappa_{s_0.T}(0))+\frac{N_0}{\kappa_{s_0.T}(0)^2+p\kappa_{s_0.T}(1)^2}D(
\kappa_{s_0.T}(0) , -\kappa_{s_0.T}(1)).
 \label{Cal_2_M}
\end{align}
When $s_0.T\neq \star$, the action of $x_1$ on $v_T$ is given by the following  $4\times 4$ matrix on the subspace spanned by vectors $v_T$, $c_0c_1v_T$ ,$c_0v_{s_0.T}$ and $c_1v_{s_0.T}$:
\begin{align}
x_1(T)=\begin{bmatrix}
A
& f(s_0.T) Z\\
f(T) Z & B \end{bmatrix}.\label{Cal_2_x1better}
\end{align}
When $s_0.T=\star$, the action of $x_1$ on $\langle v_T,c_0c_1v_T \rangle$ is given by the upper left $2\times 2$ block of the above matrix.

One main result of this paper is as follows.
\begin{theorem}\label{relationshold}
The above construction gives a well-defined action of $\mathcal{H}^{\operatorname{ev}}_{p,d}$ on $\mathcal{D}^{\lambda}_f$ and $\mathcal{H}^{\operatorname{od}}_{p,d}$ on $\mathcal{E}^{\lambda}_f$. In other words, all the relations in $\mathcal{H}^{\operatorname{ev}}_{p,d}$ and $\mathcal{H}^{\operatorname{od}}_{p,d}$ are satisfied.
\end{theorem}
Checking all relations is a little involved. First let us start with some algebraic identities regarding the scalars in (\ref{kappadefinition}). For future references, when the path $T$ is clear from the context and $s_0.T\neq \star$, let us write $\kappa_0=\kappa_T(0)$, $\kappa_1=\kappa_T(1)$, $\kappa_0'=\kappa_{s_0.T}(0)$, $\kappa_1'=\kappa_{s_0.T}(1)$. In addition, let $\kappa^2=\kappa_0^2+\kappa_1^2$. These four scalars are relationed by following identities.
\begin{lemma}\label{kapparelations}
When $s_0.T\neq \star$, the following identities hold
\begin{align}
\kappa_0^2+\kappa_1^2=&(\kappa_0')^2+(\kappa_1')^2, \label{sumconstant}\\
\kappa_0^4+p^2\kappa_1^4=&p^3(p+1)\kappa_1^2+p(p+1)\kappa_0^2+2p\kappa_0^2\kappa_1^2, \label{degree4relation}\\
\kappa'_0=p \kappa_1 \sqrt{\frac{\kappa_0^2+\kappa_1^2}{\kappa_0^2+p^2\kappa_1^2}},
 \kappa'_1=\kappa_0 \sqrt{\frac{\kappa_0^2+\kappa_1^2}{\kappa_0^2+p^2\kappa_1^2}},&
 \kappa_0=p\kappa'_1 \sqrt{\frac{(\kappa'_0)^2+(\kappa'_1)^2}{(\kappa'_0)^2+p^2(\kappa'_1)^2}},
 \kappa_1=\kappa'_0 \sqrt{\frac{(\kappa'_0)^2+(\kappa'_1)^2}{(\kappa'_0)^2+p^2(\kappa'_1)^2}},\label{rewriteprimes}\\
 p(p+1)=&\frac{(\kappa_1+\kappa_1')^2\big((\kappa_0-\kappa_0')^2+(\kappa_1-\kappa_1')^2\big)}{(\kappa_0-\kappa_0')^2+(\kappa_1+\kappa_1')^2},\label{nokappas}\\
p(p+1)+\frac{4(m+1)(m-p)p}{1+p}=&\frac{(\kappa_1+\kappa_1')^2\big((\kappa_0+\kappa_0')^2+(\kappa_1-\kappa_1')^2\big)}{(\kappa_0+\kappa_0')^2+(\kappa_1+\kappa_1')^2}.
\label{anothernokappas}
\end{align}
Here, $m$ is the number of boxes in the first row of partition $T^{(0)}$, the target of the first edge in $T$. In addition, $m>p$, and the right hand side of condition (\ref{Cal_2_f}) is never zero.
\end{lemma}
\begin{proof}
To prove (\ref{sumconstant}), let $\lambda_1=T^{(0)}$ and  $\lambda_2=(s_0.T)^{(0)}$, which are the two partitions in Row 0 of paths $T$ and $s_0.T$. By definition, $T$ and $s_0.T$ travel through the same vertex in any other row, specifically parition $\mu$ in Row $1$. By Lemma~\ref{twopaths}, let $\lambda_1$ be the partition with a shorter first row, then $\mu$ is the following partition, where yellow boxes form the shape of $\alpha$, $\lambda_1$ is the partition without the green box, and $\lambda_2$ is the partition without the red box:

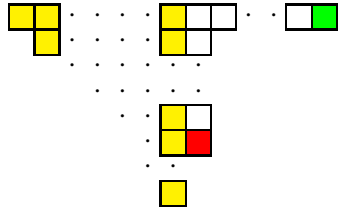
\begin{center}
\ytableausetup{smalltableaux}
\begin{ytableau}
*(yellow) & *(yellow) & \none[\cdot] & \none[\cdot] & \none[\cdot] & \none[\cdot] &  *(yellow) &  &  & \none[\cdot] & \none[\cdot] & & *(green)\\
\none & *(yellow) &   \none[\cdot] & \none[\cdot] & \none[\cdot] & \none[\cdot] & *(yellow) &  \\
\none & \none & \none[\cdot] & \none[\cdot] & \none[\cdot] &\none[\cdot]&\none[\cdot]&\none[\cdot]\\
\none & \none & \none & \none[\cdot] & \none[\cdot] &\none[\cdot]&\none[\cdot]&\none[\cdot]\\
\none & \none & \none & \none & \none[\cdot] &\none[\cdot]&*(yellow)&\\
\none & \none & \none & \none & \none &\none[\cdot]&*(yellow)& *(red) \\
\none & \none & \none & \none & \none &\none[\cdot]&\none[\cdot]\\
\none & \none & \none & \none & \none &\none  & *(yellow)\\
\end{ytableau}
\end{center}
If $\lambda_1$ has $m$ boxes in the first row, then the green box has content $m$. Using the fact that the skew shape $\lambda/\alpha$ has $p$ boxes, the red box has content $m-p$. Since all rows start with a box with content $0$, $m-p>0$ and $m>p$. These are the two boxes added to $\lambda_1$ or $\lambda_2$ to obtain $\mu$, therefore by the definition (\ref{kappadefinition}),
\begin{align}
\kappa_0^2&=mp(m-p), \hspace{.3 in} \kappa_1^2=m(m+1),\label{specificvalues1} \\
(\kappa_0')^2&=(m+1)p(m-p+1), \hspace{.3 in}(\kappa_{1}')^2=(m-p)(m-p+1). \label{specificvalues2}
\end{align}
and (\ref{sumconstant}) is a straightforward calculation. To see (\ref{degree4relation}), notice
\begin{align*}
\kappa_0^2-p\kappa_1^2&=-mp^2-mp=-mp(p+1), \hspace{.3 in} m=-\frac{\kappa_0^2-p\kappa_1^2}{p(p+1)}.
\end{align*}
and substituting the second equality in the expression for $p^2(p+1)^2\kappa_1^2=p^2(p+1)^2m(m+1)$. This leads to (\ref{degree4relation}).

To see (\ref{rewriteprimes}), notice
\begin{align}
\kappa_0\kappa_0'=p\sqrt{m(m+1)(m-p)(m-p+1)},  \hspace{.2 in}
\kappa_1\kappa_1'=\sqrt{m(m+1)(m-p)(m-p+1)}. \label{sameindex}
\end{align}
Therefore $\kappa_0\kappa_0'=p\kappa_1\kappa_1'$. The last equality in (\ref{rewriteprimes}) follows from substituting  $\kappa_0=\frac{p\kappa_1'}{\kappa_0'}\kappa_1$ in (\ref{sumconstant}) and solve for $\kappa_1$. The other equalities can be shown similarly.

To see (\ref{nokappas}), one can use (\ref{sameindex}) and the following
\begin{align*}
\kappa_1^2+(\kappa_1')^2&=m(m+1)+(m-p)(m-p+1), \hspace{.2 in} \kappa^2=m(p+1)(m-p+1)
\end{align*}
to substitute them in
\begin{align*}
\frac{(\kappa_1+\kappa_1')^2\big((\kappa_0-\kappa_0')^2+(\kappa_1-\kappa_1')^2\big)}{(\kappa_0-\kappa_0')^2+(\kappa_1+\kappa_1')^2}
=\frac{\big(\kappa_1^2+(\kappa_1')^2+2\kappa_1\kappa_1'\big)(2\kappa^2-2\kappa_0\kappa_0'-2\kappa_1\kappa_1')}{2\kappa^2-2\kappa_0\kappa_0'+2\kappa_1\kappa_1'}
\end{align*}
and
\begin{align*}
\frac{(\kappa_1+\kappa_1')^2\big((\kappa_0+\kappa_0')^2+(\kappa_1-\kappa_1')^2\big)}{(\kappa_0+\kappa_0')^2+(\kappa_1+\kappa_1')^2}
=\frac{\big(\kappa_1^2+(\kappa_1')^2+2\kappa_1\kappa_1'\big)(2\kappa^2+2\kappa_0\kappa_0'-2\kappa_1\kappa_1')}{2\kappa^2+2\kappa_0\kappa_0'+2\kappa_1\kappa_1'},
\end{align*}
and obtain (\ref{nokappas}) and (\ref{anothernokappas}) via a straightforward calculation. We also verify this using MAGMA codes which are included in Section~\ref{codeskapparelations}.

To prove the last claim, notice that $\kappa_0^2+p^2\kappa_1^2=m^2p+m^2p^2\neq 0$. On the other hand, the content of the red box is at least $1$ and at most $n-1$, therefore $1\leq m-p\leq n-1$, and by (\ref{specificvalues2}) $(\kappa_1')^2\neq n(n+1)$. Also, since $\lambda_1$ has at least $n+1$ boxes in the first row, $m\geq n+1$ and (\ref{specificvalues1}) implies $\kappa_1^2\neq n(n+1)$. Therefore the right hand side of (\ref{Cal_2_f}) is never zero.

\end{proof}

It is also helpful to give a reformulation of the action $x_1$ defined in (\ref{Cal_2_x1better}). This alternative construction is more technical to define, but will be helpful when we check all relations in $\mathcal{H}^{\operatorname{ev}}_{p,d}$ or $\mathcal{H}^{\operatorname{od}}_{p,d}$ . Recall the matrix $D(a,b)$ in (\ref{CliffordD}).

\begin{lemma}\label{preparation}
The following three quantities are equal
\begin{align}
&\frac{2\frac{N_0}{\kappa^2}(\kappa^2-\kappa_1^2-(\kappa_1')^2)+(\kappa_1^2+(\kappa_1')^2-p(p+1))}{2(\kappa_0\kappa_1+\kappa_0'\kappa_1')} \notag \\
=&\frac{\kappa_0\kappa_1}{\kappa_0^2+p\kappa_1^2}(\frac{N_0}{\kappa^2}(p-1)+1)
=\frac{\kappa_0' \kappa_1'}{(\kappa_0')^2+p(\kappa'_1)^2}(\frac{N_0}{\kappa^2}(p-1)+1):=c.\label{cdefinition}
\end{align}
Further let 
\begin{align}
Q=D(\kappa_1, -\kappa_0),\hspace{.2 in}
X=D(\kappa_0,\kappa_1), \hspace{.2 in}
R=D(\kappa_1', \kappa_0'),\hspace{.2 in}
Y =D(\kappa_0', -\kappa_1'). \label{matrixdefinition}
\end{align}
Recall the matrix $Z$ in (\ref{zdefinition}). The following is true, where the scalars represent the corresponding $2$-by-$2$ scalar matrices
\begin{align}
&Q^2=R^2=X^2=Y^2=\kappa^2,\hspace{.3 in} Z^2=(\kappa_0-\kappa_0')^2+(\kappa_1+\kappa_1')^2, \label{matrixidentity1} \\
&QX+XQ=YR+RY=XZ+ZY=ZX+YZ=QZ+ZR=RZ+ZQ=0.\label{matrixidentity2}
\end{align}
Also, the upper-left and lower-right block of the matrix $x_1(T)$ in (\ref{Cal_2_x1better}) can be rewritten, so that  
\begin{align}
x_1(T)=\begin{bmatrix}
 \frac{N_0}{\kappa^2}Q+cX& f(s_0.T)Z \\
 f(T)Z & \frac{N_0}{\kappa^2}R+cY
\end{bmatrix}.  \label{Cal_2_x1acts}
\end{align}
Furthermore, using the newly defined quantity $c$ in (\ref{cdefinition}), the condition $(\ref{Cal_2_f})$ on $f$ is equivalent to 
\begin{align}
f(T)f(s_0.T)\big((\kappa_0-\kappa_0')^2+(\kappa_1+\kappa_1')^2\big) =N_0-\frac{N_0^2}{\kappa^2}-c^2\kappa^2,  \label{Cal_2_f0}
\end{align}
and the right hand side $F_T$ of (\ref{Cal_2_f}) remains invariant when replacing $T$ with $s_0.T$, and therefore such an $f$ exists by taking $f(T)=\sqrt{F_T}$.
\end{lemma}

\begin{proof}
To see (\ref{matrixidentity1}), notice $D^2(a,b)=a^2+b^2$. The identities in (\ref{matrixidentity2}) are straightforward to check. The first equality in (\ref{cdefinition}) is obtained by first using (\ref{rewriteprimes}) in Lemma~\ref{kapparelations} to rewrite $\kappa_0'$ and $\kappa_1'$, and simplifying the result using (\ref{degree4relation}). To see the second equality in (\ref{cdefinition}) holds, we can use (\ref{specificvalues1}) and (\ref{specificvalues2}) in the proof of Lemma~\ref{kapparelations}, and obtain
\begin{align*}
\frac{\kappa_0\kappa_1}{\kappa_0^2+p\kappa_1^2}=\frac{m\sqrt{(m+1)(m-p)p}}{mp(2m-p+1)}= \frac{(m+1-p)\sqrt{(m+1)(m-p)p}}{(m+1-p)p(2m+1-p)}=\frac{\kappa_0' \kappa_1'}{(\kappa_0')^2+p(\kappa'_1)^2}.
\end{align*}
To check that the entries in $x_1(T)$ match those in (\ref{Cal_2_x1better}), notice
\begin{align*}
\frac{N_0}{\kappa^2}Q+cX &=\frac{N_0}{\kappa^2}\begin{bmatrix}
\kappa_1 & -\kappa_0\\
-\kappa_0 & -\kappa_1
\end{bmatrix}+\frac{\kappa_0\kappa_1}{\kappa_0^2+p\kappa_1^2}(\frac{N_0}{\kappa^2}(p-1)+1)\begin{bmatrix}
\kappa_0 & \kappa_1\\
\kappa_1 & -\kappa_0 
\end{bmatrix}\\
&=\frac{N_0}{\kappa^2(\kappa_0^2+p\kappa_1^2)}\begin{bmatrix}
\kappa_1(\kappa_0^2+p\kappa_1^2)+(p-1)\kappa_0^2\kappa_1 & -\kappa_0(\kappa_0^2+p\kappa_1^2)+(p-1)\kappa_0\kappa_1^2\\
-\kappa_0(\kappa_0^2+p\kappa_1^2)+(p-1)\kappa_0\kappa_1^2 & -\kappa_1(\kappa_0^2+p\kappa_1^2)-(p-1)\kappa_0^2\kappa_1
\end{bmatrix}\\
&\hh +\frac{\kappa_0\kappa_1}{\kappa_0^2+p\kappa_1^2}\begin{bmatrix}
\kappa_0 & \kappa_1\\
\kappa_1 & -\kappa_0 
\end{bmatrix}\\
&=\frac{N_0}{\kappa^2(\kappa_0^2+p\kappa_1^2)}\begin{bmatrix}
p\kappa_1(\kappa_0^2+\kappa_1^2) & -\kappa_0(\kappa_0^2+\kappa_1^2)\\
-\kappa_0(\kappa_0^2+\kappa_1^2) & -p\kappa_1(\kappa_0^2+\kappa_1^2)
\end{bmatrix}+\frac{\kappa_0\kappa_1}{\kappa_0^2+p\kappa_1^2}\begin{bmatrix}
\kappa_0 & \kappa_1\\
\kappa_1 & -\kappa_0 
\end{bmatrix}\\
&=\frac{N_0}{\kappa_0^2+p\kappa_1^2}\begin{bmatrix}
p\kappa_1 & -\kappa_0\\
-\kappa_0 & -p\kappa_1
\end{bmatrix}+\frac{\kappa_0\kappa_1}{\kappa_0^2+p\kappa_1^2}\begin{bmatrix}
\kappa_0 & \kappa_1\\
\kappa_1 & -\kappa_0 
\end{bmatrix},
\end{align*}
and the lower-right block can be checked similarly.

To check the last claim, we use (\ref{degree4relation}) in Lemma~\ref{kapparelations}:
\begin{align*}
\frac{(\kappa_0^2+p\kappa_1^2)^2+(p-1)^2\kappa_0^2\kappa_1^2}{\kappa^2}=\frac{\kappa_0^4+p^2\kappa_1^4+(p^2+1)\kappa_0^2\kappa_1^2}{\kappa_0^2+\kappa_1^2}=\kappa_0^2+p^2\kappa_1^2.
\end{align*}
Therefore
\begin{align*}
&\hh N_0-\frac{N_0^2}{\kappa^2}-c^2\kappa^2\\
&=\frac{1}{(\kappa_0^2+p\kappa_1^2)^2}(N_0(\kappa_0^2+p\kappa_1^2)^2-\frac{N_0^2}{\kappa^2}(\kappa_0^2+p\kappa_1^2)^2-\kappa_0^2\kappa_1^2(\frac{N_0}{\kappa^2}(p-1)+1)^2\kappa^2)\\
&=\frac{1}{(\kappa_0^2+p\kappa_1^2)^2}(-\frac{N_0^2}{\kappa^2}((\kappa_0^2+p\kappa_1^2)^2+(p-1)^2\kappa_0^2\kappa_1^2) +N_0((\kappa_0^2+p\kappa_1^2)^2-2(p-1)\kappa_0^2\kappa_1^2)-\kappa_0^2\kappa_1^2\kappa^2)\\
&=\frac{1}{(\kappa_0^2+p\kappa_1^2)^2}(-N_0^2(\kappa_0^2+p^2\kappa_1^2)+N_0(\kappa_0^4+p^2\kappa_1^4+2\kappa_0^2\kappa_1^2)-\kappa_0^2\kappa_1^2\kappa^2)\\
&=-\frac{1}{(\kappa_0^2+p\kappa_1^2)^2}(N_0-\kappa_1^2)(N_0(\kappa_0^2+p^2\kappa_1^2)-\kappa_0^2(\kappa_0^2+\kappa_1^2))\\
&=-\frac{\kappa_0^2+p^2\kappa_1^2}{(\kappa_0^2+p\kappa_1^2)^2}(N_0-\kappa_1^2)(N_0-\kappa_0^2\frac{\kappa_0^2+\kappa_1^2}{\kappa_0^2+p^2\kappa_1^2})=-\frac{\kappa_0^2+p^2\kappa_1^2}{(\kappa_0^2+p\kappa_1^2)^2}(N_0-\kappa_1^2)(N_0-(\kappa'_1)^2).
\end{align*}
and the last equality follows from (\ref{rewriteprimes}).

By Lemma~\ref{kapparelations}, $\kappa^2=\kappa_0^2+\kappa_1^2=(\kappa_0')^2+(\kappa_1')^2$, and by the first claim $c$ is constant for $T$ and $s_0.T$, (\ref{Cal_2_f0}) remains unchanged when we replace $T$ with $s_0.T$, therefore such an $f$ exists.
\end{proof}

To prove Theorem~\ref{relationshold}, let us check each relation in $\mathcal{H}^{\operatorname{od}}_{p,d}$ or $\mathcal{H}^{\operatorname{ev}}_{p,d}$ separately.

\begin{lemma}\label{firstrelation}
The relations $x_1 s_i = s_i x_1(2\leq i\leq d-1)$ are satisfied in the definition of $\mathcal{D}^{\lambda}_f$ and $\mathcal{E}^{\lambda}_f$.
\end{lemma}
\begin{proof}
For a fixed $i$, $2\leq i\leq d-1$, if $S=s_i.T$, then $T$ and $S$ share vertices at all rows except at Row $i$, therefore they share vertices at Row $0$ and Row $1$, $\kappa_0(T)=\kappa_0(S)$ and $\kappa_1(T)=\kappa_1(S)$. Since the entries in $x_1(T)$ are determined completely by $\kappa_0(T)$ and $\kappa_1(T)$, we have $x_1(T)=x_1(S)$. If $a_1,a_2,a_3,a_4$ are entries in the first column of this matrix, then 
\begin{align*}
x_1.v_T=(a_1+a_2c_0c_1)v_T+(a_3c_0+a_4c_1)v_{s_0.T}, \hspace{.2 in} x_1.v_S=(a_1+a_2c_0c_1)v_S+(a_3c_0+a_4c_1)v_{s_0.S}.
\end{align*}

On the other hand, if $s_0.T=R$, then $T$ and $R$ share the same vertices at Row $1$ and beyond, hence $\kappa_i(T)=\kappa_i(R)$ and $\kappa_{i+1}(T)=\kappa_{i+1}(R)$ for $2\leq i\leq d-1$. Let $b_1,b_2,b_3$ be the coefficients in the action of $s_i$ in (\ref{transpositionsaction}), since they are completely determined by $\kappa_i(T)$ and $\kappa_{i+1}(T)$,
\begin{align*}
s_i.v_T=(b_1+b_2c_ic_{i+1})v_T+b_3v_S, \hspace{.2 in} s_i.v_{R}=(b_1+b_2c_ic_{i+1})v_{R}+b_3v_{s_i.R}.
\end{align*}
Lastly,  as actions on the set $\Gamma^{\lambda}$ of paths, $s_i$ $(2\leq i\leq d)$ commutes with $s_0$, and
\begin{align*}
s_0.S=s_0.(s_i.T)=s_i(s_0.T)=s_i.R.
\end{align*}
 We compare the actions on either side:
\begin{align*}
&\hh x_1.(s_iv_T)=x_1((b_1+b_2c_ic_{i+1})v_T+b_3v_S)\\
&=(b_1+b_2c_ic_{i+1})((a_1+a_2c_0c_1)v_T+(a_3c_0+a_4c_1)v_{s_0.T})+b_3((a_1+a_2c_0c_1)v_S+(a_3c_0+a_4c_1)v_{s_0.S}),\\
&\hh s_i.(x_1v_T)=s_i((a_1+a_2c_0c_1)v_T+(a_3c_0+a_4c_1)v_{R})\\
&=(a_1+a_2c_0c_1)((b_1+b_2c_ic_{i+1})v_T+b_3v_S) +(a_3c_0+a_4c_1)((b_1+b_2c_ic_{i+1})v_{R}+b_3v_{s_i.R}).
\end{align*}
The two results are the same by comparison.
\end{proof}

\begin{lemma}
The relation $x_1^2=n(n+1)$ is satisfied in the definition of $\mathcal{D}^{\lambda}_f$ and $\mathcal{E}^{\lambda}_f$.
\end{lemma}
\begin{proof}
 On the subspace spanned by $\{v_T, c_0c_1v_T,c_0v_{s_0.T},c_1v_{s_0.T}\}$, $x_1^2$ acts as
\begin{align*}
x_1^2&=\begin{bmatrix}
\frac{N_0}{\kappa^2}Q+cX& f(s_0.T)Z \\
 f(T)Z & \frac{N_0}{\kappa^2}R+cY
\end{bmatrix}^2.
\end{align*}

Let us calculate the $2\times 2$ blocks individually. By (\ref{matrixidentity1}), (\ref{matrixidentity2}) and (\ref{Cal_2_f0}) in Lemma~\ref{preparation}, the upper left block is
\begin{align*}
(x_1^2)_{11}&=\frac{N_0^2}{\kappa^4}Q^2+c^2X^2+\frac{cN_0}{\kappa^2}(QX+XQ) +f(T)f(s_0.T)Z^2\\
&=\frac{N_0^2\kappa^2}{\kappa^4}+c^2\kappa^2+f(T)f(s_0.T)((\kappa_0-\kappa_0')+(\kappa_1+\kappa_1')^2)=n(n+1).
\end{align*}
Similarly, the lower right block is
\begin{align*}
(x_1^2)_{22}&=\frac{N_0^2}{\kappa^4}R^2+c^2X^2+\frac{cN_0}{\kappa^2}(RX+XR) +f(T)f(s_0.T)Z^2\\
&=\frac{N_0^2\kappa^2}{\kappa^4}+c^2\kappa^2+f(T)f(s_0.T)((\kappa_0-\kappa_0')+(\kappa_1+\kappa_1')^2)=n(n+1).
\end{align*}
The other two blocks are as follows
\begin{align*}
(x_1^2)_{12}&=f(s_0.T)(\frac{N_0}{\kappa^2}(QZ+ZR)+f(s_0.T)c^2(XZ+ZY))=0,\\
(x_1^2)_{21}&=f(T)(\frac{N_0}{\kappa^2}(ZQ+RZ)+f(T)c^2(ZX+YZ))=0.
\end{align*}
Therefore $x_1^2=N_0=n(n+1)$.
\end{proof}

To check the next two relations, notice
\begin{align*}
z_0.(c_0c_1 v_T)=-c_0c_1(z_0.v_T)=-\kappa_T(0)c_0c_1v_T.
\end{align*}
One can also obtain the eigenvalues for $c_0v_T$ and $c_1v_T$ via a similar fashion. In particular, let $d(r_1,\dots,r_t)$ be the diagonal matrix with diagonal entries $r_1,\dots,r_t$, then $z_0$ acts on the subspace spanned by $v_T$, $c_0c_1 v_T$, $c_0v_{s_0.T}$, $c_1 v_{s_0.T}$ via the matrix
\begin{align*}
d(\kappa_T(0),-\kappa_T(0),-\kappa_{s_0.T}(0),\kappa_{s_0.T}(0)).
\end{align*}
On the other hand, $z_1$ acts as $d(\kappa_T(1),-\kappa_T(1),\kappa_{s_0.T}(1),-\kappa_{s_0.T}(1))$, and $c_0c_1$ acts as $d(C,C)$, whose diagonal blocks are the $2$-by-$2$ matrix $C$ in (\ref{Cdefinition}).

\begin{lemma}\label{lasttrivial}
The relation $(z_0c_0c_1+z_1-x_1)x_1=-x_1(z_0c_0c_1+z_1-x_1)$ is satisfied in the definition of $\mathcal{D}^{\lambda}_f$ and $\mathcal{E}^{\lambda}_f$.
\end{lemma}

\begin{proof}
Under the above discussion, $z_0c_0c_1+z_1$ acts as the matrix $\begin{bmatrix}
Q & 0 \\ 0 & R
\end{bmatrix}$ where $Q$ and $R$ are defined in Lemma~\ref{preparation}. Therefore, using (\ref{matrixidentity1}) and (\ref{matrixidentity2}),
we obtain, 
\begin{align*}
&\hh (z_0c_0c_1+z_1)x_1+x_1(z_0c_0c_1+z_1)\\
&=\begin{bmatrix}
Q & 0 \\ 0 & R
\end{bmatrix}\begin{bmatrix}
\frac{N_0}{\kappa^2}Q+cX& f(s_0.T)Z \\
 f(T)Z & \frac{N_0}{\kappa^2}R+cY
\end{bmatrix}+\begin{bmatrix}
\frac{N_0}{\kappa^2}Q+cX& f(s_0.T)Z \\
 f(T)Z & \frac{N_0}{\kappa^2}R+cY
\end{bmatrix}\begin{bmatrix}
Q & 0 \\ 0 & R
\end{bmatrix}\\
&=\begin{bmatrix}
\frac{2N_0}{\kappa^2}Q^2+c(QX+XQ) & f(s_0.T) (QZ+ZR)\\
f(T)(ZQ+RZ) &  \frac{2N_0}{\kappa^2}Q^2+c(RY+YR)
\end{bmatrix}=2N_0=2x_1^2.
\end{align*}

\end{proof}

\begin{lemma}\label{firstnontrivial}
The relation $(x_1-z_1)^4=p(p+1)(x_1-z_1)^2$ is satisfied in the definition of $\mathcal{D}^{\lambda}_f$ and $\mathcal{E}^{\lambda}_f$.
\end{lemma}

\begin{proof}
Let $J=\begin{bmatrix}
1&0\\0&-1
\end{bmatrix}$. Based on the reformulation of $x_1$ in Lemma~\ref{preparation}, and the discussion before Lemma~\ref{lasttrivial}, $x_1-z_1$ acts on the space spanned by $\{v_T,c_0c_1 v_T,c_0v_{s_0.T},c_1 v_{s_0.T}\}$ via the following $4$-by-$4$ matrix:
\begin{align*}
x_1-z_1&=\begin{bmatrix}
 \frac{N_0}{\kappa^2}Q+cX-\kappa_1J& f(s_0.T)Z \\
 f(T)Z & \frac{N_0}{\kappa^2}R+cY-\kappa_1'J.
\end{bmatrix}
\end{align*}
Let us compute each $2\times 2$ block in $(x_1-z_1)^2=(A_{ij})$ individually, using (\ref{matrixidentity1}), (\ref{matrixidentity2}) and (\ref{Cal_2_f0}) developed in Lemma~\ref{preparation}. The scalars represent suitable scalar matrices:
\begin{align*}
A_{11}&= 
(\frac{N_0^2}{\kappa^4}Q^2+c^2X^2+f(T)f(s_0.T)Z^2)+\kappa_1^2-\frac{N_0c}{\kappa^2}(QX+XQ)\\
& -\frac{N_0\kappa_1}{\kappa^2}(JQ+QJ)-c\kappa_1(JX+XJ)=N_0+\kappa_1^2-\frac{2N_0\kappa_1^2}{\kappa^2}-2c\kappa_0\kappa_1,\\
A_{22}&=(\frac{N_0^2}{\kappa^4}R^2+c^2Y^2+f(T)f(s_0.T)Z^2)+(\kappa_1')^2-\frac{N_0c}{\kappa^2}(RX+XR)\\
& -\frac{N_0\kappa_1}{\kappa^2}(JR+RJ)-c\kappa_1(JY+YJ)=N_0+\kappa_1^2-\frac{2N_0(\kappa_1')^2}{\kappa^2}-2c\kappa_0'\kappa_1',\\
A_{12}&=f(s_0.T)(\frac{N_0^2}{\kappa^4}(QZ+ZR)+c(XZ+ZY)-(\kappa_1JZ+\kappa_1'ZJ))=-f(s_0.T)(\kappa_1JZ+\kappa_1'ZJ),\\
A_{21}&=f(T)(\frac{N_0^2}{\kappa^4}(ZQ+RZ)+c(ZX+YZ)-(\kappa_1'JZ+\kappa_1ZJ))=-f(T)(\kappa_1'JZ+\kappa_1ZJ).
\end{align*}
Therefore the relation becomes
\begin{align}
(x_1-z_1)^4-p(p+1)(x_1-z_1)^2=
\begin{bmatrix}
A_{11}^2+A_{12}A_{21} & (A_{11}+A_{22})A_{12}\\
 (A_{11}+A_{22})A_{21} & A_{22}^2+A_{21}A_{12}
\end{bmatrix}-p(p+1) \begin{bmatrix}
A_{11} & A_{12} \\
A_{21} & A_{22}
\end{bmatrix}. \label{allinone}
\end{align}
The upper-right and lower-left blocks can be checked as follows, using the first definition of $c$, in (\ref{cdefinition}) of Lemma~\ref{preparation}:
\begin{align*}
\hh A_{11}+A_{22}&=(N_0+\kappa_1^2-\frac{2N_0\kappa_1^2}{\kappa^2}-2c\kappa_0\kappa_1)+(N_0+(\kappa_1')^2-\frac{2N_0(\kappa_1')^2}{\kappa^2}-2c\kappa_0'\kappa_1')\\
&=\frac{2N_0}{\kappa^2}(\kappa^2-\kappa_1^2-(\kappa_1')^2)+(\kappa_1^2+(\kappa_1')^2)-2c(\kappa_0\kappa_1+\kappa_0'\kappa_1')=p(p+1).
\end{align*}
Now let us check the upper-left block and the remaining block can be checked similarly. Notice
\begin{align*}
JZJZ&=\begin{bmatrix}
1 & 0\\ 0 & -1
\end{bmatrix}
\begin{bmatrix}
\kappa_0-\kappa_0' & \kappa_1+\kappa_1'\\ \kappa_1+\kappa_1' & -(\kappa_0-\kappa_0')
\end{bmatrix}
\begin{bmatrix}
1 & 0\\ 0 & -1
\end{bmatrix}
\begin{bmatrix}
\kappa_0-\kappa_0' & \kappa_1+\kappa_1'\\ \kappa_1+\kappa_1' & -(\kappa_0-\kappa_0')
\end{bmatrix}\\
&=\begin{bmatrix}
\kappa_0-\kappa_0' & \kappa_1+\kappa_1'\\ -(\kappa_1+\kappa_1') & \kappa_0-\kappa_0'
\end{bmatrix}^2=(\kappa_0-\kappa_0')^2-(\kappa_1+\kappa_1')^2.
\end{align*}
Similarly, 
\begin{align*}
JZZJ=ZJJZ=(\kappa_0-\kappa_0')^2+(\kappa_1+\kappa_1')^2,
\hspace{.2 in}
ZJZJ=(\kappa_0-\kappa_0')^2-(\kappa_1+\kappa_1')^2.
\end{align*}
Using (\ref{nokappas}) in Lemma~\ref{kapparelations}, and the equivalent condition on $f$ in (\ref{Cal_2_f0}) of Lemma~\ref{preparation},
\begin{align*}
&\hh A_{12}A_{21}\\
&=f(T)f(s_0.T)(\kappa_1JZ+\kappa_1'ZJ)(\kappa_1'JZ+\kappa_1ZJ)\\
&=f(T)f(s_0.T)\left(2\kappa_1\kappa_1'((\kappa_0-\kappa_0')^2-(\kappa_1+\kappa_1')^2) +(\kappa_1^2+(\kappa_1')^2)((\kappa_0-\kappa_0')^2+(\kappa_1+\kappa_1')^2)\right)\\
&=f(T)f(s_0.T)\left((\kappa_1^2+\kappa_1')^2(\kappa_0-\kappa_0')^2+(\kappa_1^2+\kappa_1')^2(\kappa_1-\kappa_1')^2\right)\\
&=f(T)f(s_0.T)((\kappa_0-\kappa_0')^2+(\kappa_1+\kappa_1')^2)p(p+1)=p(p+1)(N_0-\frac{N_0^2}{\kappa^2}-c^2\kappa^2).
\end{align*}
Now we check the upper-left block of (\ref{allinone}) by rewriting $N_0=x\kappa^2$, and viewing the expression as a polynomial in $x$. In particular, the second expression for $c$ in (\ref{cdefinition}) becomes
\begin{align*}
c=\frac{\kappa_0\kappa_1}{\kappa_0^2+p\kappa_1^2}(x(p-1)+1),
\end{align*}
and
\begin{align*}
&\hh A_{11}^2+A_{12}A_{21}-p(p+1)A_{11}=\left(x\kappa^2+\kappa_1^2-2x\kappa_1^2-2\frac{\kappa_0^2\kappa_1^2}{\kappa_0^2+p\kappa_1^2}(x(p-1)+1)\right)^2\\
&+p(p+1)\left(x\kappa^2-x^2\kappa^2-\frac{\kappa^2\kappa_0^2\kappa_1^2}{(\kappa_0^2+p\kappa_1^2)^2}(x(p-1)+1)^2\right)\\
& -p(p+1)\left(x\kappa^2+\kappa_1^2-2x\kappa_1^2-\frac{2\kappa_0^2\kappa_1^2}{\kappa_0^2+p\kappa_1^2}(x(p-1)+1)\right).
\end{align*}
This is zero by checking coefficients of powers of $x$. We check the coefficient of $x^2$ and the other two can be checked in a similar fashion. For example, each coefficient of $x^i$ is a rational expression of $\kappa$, $\kappa_0$, $\kappa_1$ and $p$. Using (\ref{specificvalues1}) in the proof of Lemma~\ref{kapparelations},  one can rewrite it as a rational expression of $m$, $n$ and $p$, and check whether its denominator is the zero polynomial. Using this method one can check
\begin{align}
\left( \kappa^2-2\kappa_1^2-\frac{2\kappa_0^2\kappa_1^2}{\kappa_0^2+p\kappa_1^2}(p-1) \right)^2-p(p+1)\kappa^2-p(p+1)\frac{\kappa^2\kappa_0^2\kappa_1^2}{(\kappa_0^2+p\kappa_1^2)^2}(p-1)=0. \label{code1}
\end{align}
This can be done via a lengthy yet straightforward calculation, or by a few lines of codes in MAGMA. The codes are included in the Appendix.
\end{proof}

\begin{lemma}
The relation
\begin{align}
x_1 (s_1x_1s_1+(1-c_1c_2)s_1) &= (s_1x_1s_1+(1-c_1c_2)s_1) x_1 \label{lastrelation}
\end{align}
is satisfied in the definition of $\mathcal{D}^{\lambda}_f$ and $\mathcal{E}^{\lambda}_f$.
\end{lemma}

\begin{proof}
To verify the relation, we need to discuss $T$ by cases. In particular, the orbit of the action of $\{s_0,s_1\}$ on the set of paths plus $\star$, is always of cardinality $5$ and of the form of one of following two cases. We omit the staircase portion of the tableaux, and leave boxes empty if they are filled with integers $3$, $4$, \dots. The green boxes outline the partition in Row $0$ of $T$.

Case 1)
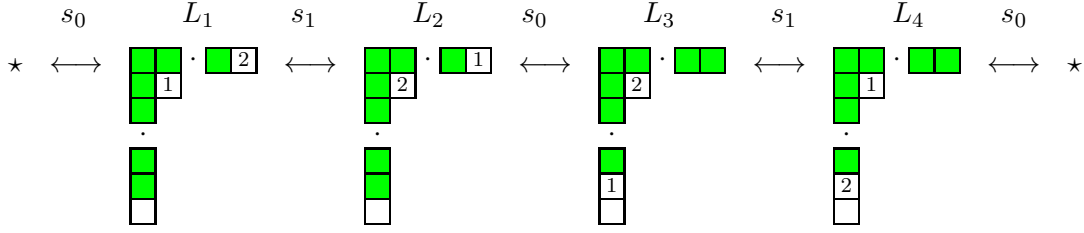
\begin{align*}
& \hspace{.1 in} s_0 \hspace{.5 in} L_1 \hspace{.4 in} s_1 \hspace{.5 in} L_2 \hspace{.4 in} s_0 \hspace{.5 in} L_3 \hspace{.5 in} s_1 \hspace{.5 in} L_4 \hspace{.4 in} s_0\\
\ytableausetup{smalltableaux} 
\star   \hspace{.1 in} &\longleftrightarrow  \hspace{.1 in}
\begin{ytableau}
*(green)&*(green) &  \none[\cdot]  &*(green) &2 \\
 *(green) &  1  \\
 *(green) \\
 \none[\cdot]\\
 *(green) \\
  *(green) \\
  \\
\end{ytableau} \hspace{.1 in}  \longleftrightarrow  \hspace{.1 in}
 \begin{ytableau}
*(green)&*(green) &  \none[\cdot]  &*(green) &1 \\
 *(green) &  2  \\
 *(green) \\
 \none[\cdot]\\
 *(green) \\
  *(green) \\
  \\
\end{ytableau} \hspace{.1 in}
\longleftrightarrow \hspace{.1 in}
\begin{ytableau}
*(green)&*(green) &  \none[\cdot]  &*(green) & *(green)\\
 *(green) &  2  \\
 *(green) \\
 \none[\cdot]\\
 *(green) \\
  1 \\
  \\
\end{ytableau} \hspace{.1 in}
\longleftrightarrow  \hspace{.1 in}
\begin{ytableau}
*(green)&*(green) &  \none[\cdot]  &*(green) & *(green) \\
 *(green) &  1  \\
 *(green) \\
 \none[\cdot]\\
 *(green) \\
  2 \\
  \\
\end{ytableau} \hspace{.1 in}\longleftrightarrow  \hspace{.1 in} \star
\end{align*}

Since $v_{\star}=0$ by definition, $v_{s_0.L_1}=v_{s_0.L_4}=0$. 
Let $m$ be the number of green boxes in the first row of $L_2$. By (\ref{specificvalues1}) and (\ref{specificvalues2}) in Lemma~\ref{kapparelations} and the definition of $\kappa_T(i)$ in (\ref{kappadefinition}), the eigenvalues for $z_0,z_1,z_2$ are as follows
\begin{table}[!h]
\caption{Eigenvalues for paths in Case 1}\label{firsteigenvalues}
\begin{tabular}{c|c|c|c}
\hline
  & $\kappa(0)$ & $\kappa(1)$  & $\kappa(2)$ \\
  $L_1$ & $\sqrt{mp(m-p)}$ & $\sqrt{N_0}$ & $\sqrt{m(m+1)}$\\
$L_2$ & $\sqrt{mp(m-p)}$ & $\sqrt{m(m+1)}$ &$\sqrt{N_0}$\\
$L_3$ & $\sqrt{(m+1)p(m+1-p)}$ & $\sqrt{(m-p)(m-p+1)}$&$\sqrt{N_0}$\\
$L_4$ &$\sqrt{(m+1)p(m+1-p)}$ & $\sqrt{N_0}$&$\sqrt{(m-p)(m-p+1)}$\\
\hline
\end{tabular}
\end{table}

Case 2)
\begin{align*}
& \hspace{.1 in} s_1 \hspace{.5 in} T_1 \hspace{.6 in} s_0 \hspace{.5 in} T_2 \hspace{.6 in} s_1 \hspace{.5 in} T_3 \hspace{.5 in} s_0 \hspace{.5 in} T_4 \hspace{.6 in} s_1\\
\ytableausetup{smalltableaux} 
\star   \hspace{.1 in} &\longleftrightarrow  \hspace{.1 in}
\begin{ytableau}
*(green)&*(green) &  \none[\cdot]  &*(green) &1 & 2 \\
 *(green)   \\
 *(green) \\
 \none[\cdot]\\
 *(green) \\
  *(green) \\
  *(green)\\
\end{ytableau}  \hspace{.1 in}  \longleftrightarrow  \hspace{.1 in}
 \ytableausetup{smalltableaux}
\begin{ytableau}
*(green)&*(green) &  \none[\cdot]  &*(green) & *(green) & 2 \\
 *(green)   \\
 *(green) \\
 \none[\cdot]\\
 *(green) \\
  *(green) \\
 1 \\
\end{ytableau}\hspace{.1 in}
\longleftrightarrow \hspace{.1 in}
\begin{ytableau}
*(green)&*(green) &  \none[\cdot]  &*(green) &*(green) & 1 \\
 *(green)  \\
 *(green) \\
 \none[\cdot]\\
 *(green) \\
  *(green) \\
  2\\
\end{ytableau} \hspace{.1 in}
\longleftrightarrow  \hspace{.1 in}
\begin{ytableau}
*(green)&*(green) &  \none[\cdot]  &*(green) &  *(green)&*(green)\\
 *(green)\\
 *(green) \\
 \none[\cdot]\\
  *(green) \\
 1\\
  2\\
\end{ytableau} \hspace{.1 in}\longleftrightarrow  \hspace{.1 in} \star
\end{align*}

Let $m$ be the number of green boxes in the first row of $T_3$, then the eigenvalues for $z_0,z_1,z_2$ are as follows.
\begin{table}[!h]
\caption{Eigenvalues for paths in Case 2}\label{secondeigenvalues}
 \begin{tabular}{c|c|c|c}
 \hline
  & $\kappa(0)$ & $\kappa(1)$  & $\kappa(2)$ \\
  $T_1$ & $\sqrt{(m-1)p(m-1-p)}$ & $\sqrt{m(m-1)}$ & $\sqrt{m(m+1)}$\\
$T_2$ & $\sqrt{mp(m-p)}$ & $\sqrt{(m-p-1)(m-p)}$ &$\sqrt{m(m+1)}$\\
$T_3$ & $\sqrt{mp(m-p)}$ & $\sqrt{m(m+1)}$&$\sqrt{(m-p-1)(m-p)}$\\
$T_4$ &$\sqrt{(m+1)p(m+1-p)}$ & $\sqrt{(m-p)(m-p+1)}$&$\sqrt{(m-p-1)(m-p)}$\\
\hline
\end{tabular}
\end{table}

Recall the definition of $\kappa_0$, $\kappa_1$, $\kappa_0'$, $\kappa_1'$ associated to path $T$ shortly before Lemma~\ref{kapparelations}. We now define new scalars $\kappa_0''$, $\kappa_0'''$, $\kappa_2$, $\kappa_2'$, $\kappa_2''$, $\kappa_2'''$ as eigenvalues for associated paths in the left half of following table. The eigenvalues for the right half can be deducted using the left half, using the fact that the action of $s_1$ interchanges the eigenvalues for $z_1$ and $z_2$. If any of the paths in the top row is $\star$, then the associated eigenvalues in that column are undefined.
\begin{table}[!h]
\caption{Notation of some eigenvalues}\label{kappatripleprimes}
\begin{tabular}{c||c|c|c|c||c|c|c}
\hline
 & $T$ & $s_0.T$ & $s_0s_1.T$  & $s_0s_1s_0.T$& $s_1.T$ & $s_1s_0.T$& $s_1s_0s_1.T$\\
$z_0$   & $\kappa_0$ & $\kappa_0'$& $\kappa_0''$ & $\kappa_0'''$& $\kappa_0$& $\kappa_0'$& $\kappa_0''$\\
$z_1$  & $\kappa_1$ & $\kappa_1'$ & $\kappa_2''$& $\kappa_2'''$  & $\kappa_2$& $\kappa_2$&$\kappa_1$\\
$z_2$  & $\kappa_2$ & $\kappa_2$ &$\kappa_1$&$\kappa_1'$& $\kappa_1$& $\kappa_1'$&$\kappa_2''$\\
\hline
\end{tabular}
\end{table}

To obtain explicit formulas for every entry in this table, we also list the following results of applying all possible moves to each of the eight paths.
\begin{center}
\begin{tabular}{c|c|c|c|c|c|c|c|c}
\hline
$T$&$L_1$&$L_2$&$L_3$&$L_4$&$T_1$&$T_2$&$T_3$&$T_4$\\
$s_0.T$&$\star$&$L_3$&$L_2$&$\star$&$T_2$&$T_1$&$T_4$&$T_3$\\
$s_0s_1.T$&$L_3$&$\star$&$\star$&$L_2$&$\star$&$T_4$&$T_1$&$\star$\\
$s_0s_1s_0.T$&$\star$&$\star$&$\star$&$\star$&$T_4$&$\star$&$\star$&$T_1$\\
\hline
\end{tabular}
\end{center}
By superimposing (the transpose) of this table with Tables~\ref{firsteigenvalues} and \ref{secondeigenvalues}, one obtain explicit expressions for all eigenvalues in Table~\ref{kappatripleprimes} and for all paths. For example, for $T_1$, 
\begin{align*}
\kappa_0'''=\kappa_{s_0s_1s_0.T_1}(0)=\kappa_{T_4}(0)=\sqrt{(m+1)p(m+1-p)}.
\end{align*}

Let $C_3$ be the Clifford algebra generated by $c_0, c_1,c_2$. For any path $T \in \Gamma^{\lambda}$, let $a_T,b_T\in C_3$ be the coefficients defined in (\ref{Cal_2_x1better}), and $d_T,e_T\in C_3$ be the ones in (\ref{transpositionsaction}), such that
\begin{align}
x_1.v_T=a_Tv_T+b_Tv_{s_0.T}, \hspace{.2 in}s_1.v_T=d_Tv_T+e_{s_1.T}v_T.  \label{defineabde}
\end{align}
Notice $e_T\in \mathbb{C}$. Specifically, these coefficients have explicit expressions for each relevant path: 
\begin{align*}
a_T&=\gamma+\delta c_0c_1=\frac{(N_0p+\kappa_0^2)\kappa_1}{\kappa_0^2+p\kappa_1^2}+\frac{(-N_0+\kappa_1^2)\kappa_0}{\kappa_0^2+p\kappa_1^2}c_0c_1,\\
a_{s_0.T}&=\gamma'+\delta' c_0c_1=\frac{(N_0p+(\kappa_0')^2)\kappa_1'}{(\kappa_0')^2+p(\kappa_1')^2}+\frac{(-N_0+(\kappa_1')^2)\kappa_0'}{(\kappa_0')^2+p(\kappa_1')^2}c_0c_1,\\
a_{s_1.T}&=\gamma_1+\delta_1 c_0c_1=\frac{(N_0p+\kappa_0^2)\kappa_2}{\kappa_0^2+p\kappa_2^2}+\frac{(-N_0+\kappa_2^2)\kappa_0}{\kappa_0^2+p(\kappa_2)^2}c_0c_1,\\
a_{s_1s_0.T}&=\gamma'''+\delta''' c_0c_1=\frac{(N_0p+(\kappa_0')^2)\kappa_2}{(\kappa_0')^2+p\kappa_2^2}+\frac{(-N_0+\kappa_2^2)\kappa_0'}{(\kappa_0')^2+p\kappa_2^2}c_0c_1,\\
a_{s_0s_1.T}&=\gamma_2+\delta_2 c_0c_1=\frac{(N_0p+(\kappa_0'')^2)\kappa_2}{(\kappa_0'')^2+p(\kappa_2'')^2}+\frac{(-N_0+(\kappa_2'')^2)\kappa_0''}{(\kappa_0'')^2+p(\kappa_2'')^2}c_0c_1.
\end{align*}
The quantities $b_T$:
\begin{align*}
b_T&=Ac_0+Bc_1=f(T)(A'c_0+B'c_1)=f(T)\big((\kappa_0-\kappa_0')c_0+(\kappa_1+\kappa_1')c_1\big),\\
b_{s_0.T}&=Dc_0+Ec_1=f(s_0.T)(-A'c_0+B'c_1)=f(s_0.T)\big(-(\kappa_0-\kappa_0')c_0+(\kappa_1+\kappa_1')c_1\big),\\
b_{s_1.T}&=f(s_1.T)(Fc_0+Gc_1)=f(s_1.T)\big((\kappa_0-\kappa_0'')c_1+(\kappa_2+\kappa_2'')c_1\big),\\
b_{s_0s_1.T}&=f(s_0s_1.T)(-Fc_0+Gc_1),\\
b_{s_1s_0.T}&=f(s_1s_0.T)(Hc_0+Ic_1)=f(s_1s_0.T)\big((\kappa_0'-\kappa_0''')c_0+(\kappa_2+\kappa_2''')c_1\big).
\end{align*}
The quantities $d_T$:
\begin{align*}
d_T=\alpha+\beta c_1c_2=-\frac{1}{\kappa_1-\kappa_2}+\frac{1}{\kappa_1+\kappa_2}c_1c_2, \hspace{.1 in}
d_{s_0.T}=\alpha'+\beta' c_1c_2=-\frac{1}{\kappa_1'-\kappa_2}+\frac{1}{\kappa_1'+\kappa_2}c_1c_2,\\
d_{s_1.T}=\alpha_1+\beta_1 c_1c_2=-\frac{1}{\kappa_2-\kappa_1}+\frac{1}{\kappa_2+\kappa_1}c_1c_2,\hspace{.1 in}
d_{s_0s_1.T}=\alpha_2+\beta_2 c_1c_2=-\frac{1}{\kappa_2''-\kappa_1}+\frac{1}{\kappa_2''+\kappa_1}c_1c_2,\\
d_{s_0s_1s_0.T}=\alpha_3+\beta_3c_1c_2=-\frac{1}{\kappa_2'''-\kappa_1'}+\frac{1}{\kappa_2'''+\kappa_1'}c_1c_2.
\end{align*}
The quantities $e_T$:
\begin{align*}
e_T=e_{s_1.T}&=\sqrt{1-\frac{1}{(\kappa_1-\kappa_2)^2}-\frac{1}{(\kappa_1+\kappa_2)^2}}, \hspace{.2 in}
e_{s_0.T}=e_{s_1s_0.T}&=\sqrt{1-\frac{1}{(\kappa_1'-\kappa_2)^2}-\frac{1}{(\kappa_1'+\kappa_2)^2}}.
\end{align*}

We now calculate the action of both sides of Relation~(\ref{lastrelation}) on a vector $v_T$. To shorten our notation, let $\phi_0,\phi_1: C_3 \to C_3$ be the $\mathbb{C}$-algebra homomorphism such that 
\begin{align}
\phi_0(c_0)=c_0, \hspace{.1 in}\phi_0(c_1)=-c_1, \hspace{.1 in} \phi_0(c_2)=c_2, \hspace{.1 in} \phi_1(c_0)=c_0, \hspace{.1 in} \phi_1(c_1)=c_2, \hspace{.1 in} \phi_1(c_2)=c_1. \label{twistmaps}
\end{align}
Then some of the relations can be written compactly as $x_1c_i=\phi_0(c_i)x_1$ and $s_1c_i=\phi_1(c_i)s_1$ for $i=0,1,2$. 
\begin{align*}
&\hh x_1s_1.v_T=x_1(d_Tv_T+e_Tv_{s_1.T})=\phi_0(d_T)(a_Tv_T+b_Tv_{s_0.T})+e_T(a_{s_1.T}v_{s_1.T}+b_{s_1.T}v_{s_0s_1.T}),\\
&x_1s_1x_1s_1.v_T=x_1s_1.\big(\phi_0(d_T)(a_Tv_T+b_Tv_{s_0.T})+e_T(a_{s_1.T}v_{s_1.T}+b_{s_1.T}v_{s_0s_1.T})\big)\\
=&x_1.\big(\phi_1\phi_0(d_T)\phi_1(a_T)(d_Tv_T+e_Tv_{s_1.T}) +\phi_1\phi_0(d_T)\phi_1(b_T)(d_{s_0.T}v_{s_0.T}+e_{s_0.T}v_{s_1s_0.T})\\
&+e_T\phi_1(a_{s_1.T})(d_{s_1.T}v_{s_1.T}+e_{s_1.T}v_T) +e_T\phi_1(b_{s_1.T})(d_{s_0s_1.T}v_{s_0s_1.T}+e_{s_0s_1.T}v_{s_1s_0s_1.T})\big)\\
=&\phi_0\phi_1\phi_0(d_T)\phi_0\phi_1(a_T)\phi_0(d_T)(a_Tv_T+b_Tv_{s_0.T}) +\phi_0\phi_1\phi_0(d_T)\phi_0\phi_1(a_T)e_T(a_{s_1.T}v_{s_1.T}+b_{s_1.T}v_{s_0s_1.T})\\
& + \phi_0\phi_1\phi_0(d_T)\phi_0\phi_1(b_T)\phi_0(d_{s_0.T})(a_{s_0.T}v_{s_0.T}+b_{s_0.T}v_T) \\
&+ \phi_0\phi_1\phi_0(d_T)\phi_0\phi_1(b_T)e_{s_0.T}(a_{s_1s_0.T}v_{s_1s_0.T}+b_{s_1s_0.T}v_{s_0s_1s_0.T})\\
&+e_T\phi_0\phi_1(a_{s_1.T})\phi_0(d_{s_1.T})(a_{s_1.T}v_{s_1.T}+b_{s_1.T}v_{s_0s_1.T})+e_T\phi_0\phi_1(a_{s_1.T})e_{s_1.T}(a_Tv_T+b_Tv_{s_0.T})\\
&+e_T\phi_0\phi_1(b_{s_1.T})\phi_0(d_{s_0s_1.T}) (a_{s_0s_1.T}v_{s_0s_1.T}+b_{s_0s_1.T}v_{s_1T})\\
&+e_T\phi_0\phi_1(b_{s_1.T})e_{s_0s_1.T}(a_{s_1s_0s_1.T}v_{s_1s_0s_1.T}+b_{s_1s_0s_1.T}v_{s_0s_1s_0s_1.T}).
\end{align*}
Similarly
\begin{align*}
&s_1x_1.v_T= s_1(a_Tv_T+b_Tv_{s_0.T})=\phi_1(a_T)(d_Tv_T+e_{T}v_{s_1.T})+\phi_1(b_T)(d_{s_0.T}v_{s_0.T}+e_{s_0.T}v_{s_1s_0.T}),\\
&s_1x_1s_1x_1.v_T=s_1x_1\big( \phi_1(a_T)(d_Tv_T+e_{T}v_{s_1.T})+\phi_1(b_T)(d_{s_0.T}v_{s_0.T}+e_{s_0.T}v_{s_1s_0.T})\big)\\
=&s_1\big(\phi_0\phi_1(a_T)\phi_0(d_T)(a_Tv_T+b_Tv_{s_0.T})+\phi_0\phi_1(a_T)e_T(a_{s_1.T}v_{s_1.T}+b_{s_1.T}v_{s_0s_1.T})\\
&+ \phi_0\phi_1(b_T)\phi_0(d_{s_0.T})(a_{s_0.T}v_{s_0.T}+b_{s_0.T}v_{T})+\phi_0\phi_1(b_T)e_{s_0.T}(a_{s_1s_0.T}v_{s_1s_0.T}+b_{s_1s_0.T}v_{s_0s_1s_0.T})\big)\\
=& \phi_1\phi_0\phi_1(a_T)\phi_1\phi_0(d_T)\phi_1(a_T)(d_Tv_T+e_Tv_{s_1.T}) \\
&+ \phi_1\phi_0\phi_1(a_T)\phi_1\phi_0(d_T)\phi_1(b_T)(d_{s_0.T}v_{s_0.T}+e_{s_0.T}v_{s_1s_0.T})\\
& + \phi_1\phi_0\phi_1(a_T)e_T\phi_1(a_{s_1.T})(d_{s_1.T}v_{s_1.T}+e_{s_1.T}v_T)\\
& +\phi_1\phi_0\phi_1(a_T)e_T\phi_1(b_{s_1.T})(d_{s_0s_1.T}v_{s_0s_1.T}+e_{s_0s_1.T}v_{s_1s_0s_1.T})\\
& + \phi_1\phi_0\phi_1(b_T)\phi_1\phi_0(d_{s_0.T})\phi_1(a_{s_0.T})(d_{s_0.T}v_{s_0.T}+e_{s_0.T}v_{s_1s_0.T})\\
& + \phi_1\phi_0\phi_1(b_T)\phi_1\phi_0(d_{s_0.T})\phi_1(b_{s_0.T})(d_Tv_T+e_Tv_{s_1.T})\\
&+ \phi_1\phi_0\phi_1(b_T)e_{s_0.T}\phi_1(a_{s_1s_0.T})(d_{s_1s_0.T}v_{s_1s_0.T}+e_{s_1s_0.T}v_{s_0.T})\\
& + \phi_1\phi_0\phi_1(b_T)e_{s_0.T}\phi_1(b_{s_1s_0.T})(d_{s_0s_1s_0.T}v_{s_0s_1s_0.T}+e_{s_0s_1s_0.T}v_{s_1s_0s_1s_0.T}).
\end{align*}
Relation (\ref{lastrelation}) can also be written as
\begin{align*}
x_1s_1x_1s_1-s_1x_1s_1x_1+(1+c_1c_2)x_1s_1-(1-c_1c_2)s_1x_1=0.
\end{align*}
For any of the eight paths listed in the beginning of the proof, $s_0s_1s_0s_1.T=s_1s_0s_1s_0.T=\star$, therefore $v_{s_0s_1s_0s_1.T}=v_{s_1s_0s_1s_0.T}=0$ by definition. We now act the left hand side on $v_T$, and claim that the coefficients of $v_T$, $v_{s_0.T}$, $v_{s_1.T}$, $v_{s_0s_1.T}$, $v_{s_1s_0.T}$, $v_{s_1s_0s_1.T}$, $v_{s_0s_1s_0.T}$ are zero.  To illustrate the steps needed, we compute the coefficient of $v_{s_1.T}$, and the other five coefficients can be checked via a similar computation. Using the explicit formulas of $a_T,b_T,d_T,e_T$ in (\ref{defineabde}) for all paths $T$, and the definition of $\phi_0$ and $\phi_1$ in (\ref{twistmaps}), the coefficients of $v_{s_1.T}$ is 
$e_T W $, where
\begin{align*}
W=&(1+c_1c_2)a_{s_1.T}-(1-c_1c_2)\phi_1(a_{T})+ \phi_0\phi_1\phi_0(d_T)\phi_0\phi_1(a_T)a_{s_1.T} +\phi_0\phi_1(a_{s_1.T})\phi_0(d_{s_1.T})a_{s_1.T}\\
&-\phi_1\phi_0\phi_1(a_T)\phi_1\phi_0(d_T)\phi_1(a_T)
 -\phi_1\phi_0\phi_1(a_T)\phi_1(a_{s_1.T})d_{s_1.T}+\phi_0\phi_1(b_{s_1.T})\phi_0(d_{s_0s_1.T})b_{s_0s_1.T}\\
& - \phi_1\phi_0\phi_1(b_T)\phi_1\phi_0(d_{s_0.T})\phi_1(b_{s_0.T})\\
=&(1+c_1c_2)(\gamma_1+\delta_1c_0c_1)-(1-c_1c_2)(\gamma+\delta c_0c_2) +(\alpha+\beta c_2c_1)(\gamma-\delta c_0c_2)(\gamma_1+\delta_1 c_0c_1)\\
&+(\gamma_1-\delta_1 c_0c_2)(\alpha_1+\beta_1c_1c_2)(\gamma_1+\delta_1c_0c_1)-(\gamma-\delta c_0c_1)(\alpha+\beta c_2c_1)(\gamma+\delta c_0c_2)\\
&-(\gamma-\delta c_0)(\gamma_1+\delta_1 c_0c_2)(\alpha_1+\beta_1 c_1c_2)\\
&+f(s_1.T)f(s_0s_1.T)(-Fc_0+Gc_2)(\alpha_2+\beta_2c_1c_2)(-Fc_0+Gc_1)\\
&-f(T)f(s_0.T)(-A'c_0+B'c_1)(\alpha'+\beta'c_2c_1)(-A'c_0+B'c_2).
\end{align*}
Here, we organize the terms using a $\mathbb{C}$-basis of $C_3$,  $W=W_0c_1c_2+W_1c_0c_2+W_2c_0c_1+W_3$, where
\begin{align*}
W_0=&f(s_1.T)f(s_0s_1.T)(-F^2\beta_2-G^2\alpha_2)-f(T)f(s_0.T)((A')^2\beta'+(B')^2\alpha')\\
&-\gamma\beta\gamma_1-\alpha\delta\delta_1-\gamma_1^2\beta_1 -\delta_1^2\alpha_1-\gamma^2\beta-\delta^2\alpha-\delta\delta_1\alpha_1-\gamma\gamma_1\beta_1+\gamma_1+\gamma,\\
W_1=&f(s_1.T)f(s_0s_1.T)(GF\alpha_2-FG\beta_2)-f(T)f(s_0.T)(-A'B'\beta'+A'B'\alpha')\\
&+\alpha\delta\gamma_1-\gamma\beta\delta_1+\delta\alpha\gamma_1 -\gamma_1\beta_1\delta_1+\delta\beta\gamma-\gamma\alpha\delta-\gamma\delta_1\alpha_1+\delta\gamma_1\beta_1+\delta_1-\delta,\\
W_2=&f(s_1.T)f(s_0s_1.T)(-GF\beta_2+GF\alpha_2)-f(T)f(s_0.T)(A'B'\alpha'-A'B'\beta')\\
&+\beta\delta\gamma_1+\alpha\gamma\delta_1-\delta_1\beta_1\gamma_1+\gamma_1\alpha_1\delta_1-\delta\alpha\gamma+\gamma\beta\delta-\delta\gamma_1\alpha_1-\gamma\delta_1\beta_1+\delta_1-\delta,\\
W_3=&f(s_1.T)f(s_0s_1.T)(F^2\alpha_2+G^2\beta_2)-f(T)f(s_0.T)((A')^2\alpha'+(B')^2\beta')\\
&+\alpha\gamma\gamma_1-\beta\delta\delta_1+\gamma_1^2\alpha_1 + \delta_1^2\beta_1-\gamma^2\alpha-\delta^2\beta-\gamma\gamma_1\alpha_1+\delta\delta_1\beta_1+\gamma_1-\gamma.
\end{align*}
and $f(T)f(s_0.T)$, $f(s_1.T)f(s_0s_1.T)$ are subject to condition (\ref{Cal_2_f}) by definition. It is straightforward to check that, one can use explicit expressions for $a,b,d,e$ given shortly after (\ref{defineabde}), and replace all symbols with expressions of $\kappa$'s. Then one can use the discussion shortly before (\ref{defineabde}) and Tables~\ref{firsteigenvalues},\ref{secondeigenvalues},\ref{kappatripleprimes}, to further replace $\kappa$'s with expressions of $m,n$ and $p$, and check that
\begin{align}
W_0=W_1=W_2=W_3=0 \label{s1Tiszero}
\end{align}
as rational expressions of $m,n$ and $p$ (in fact, we check their denominators are zero.) Since $s_1.T_1$ and $s_1.T_4$ are both $=\star$, we only need to check (\ref{s1Tiszero}) for the remaining six paths. This can be checked via MAGMA using codes that will be included in the Appendix.

\end{proof}

To verify the last relation, we review the combinatorially constructed modules in \cite[Section 5.1]{HKS}. Recall that $H_d$ is the affine Hecke-Clifford algebra introduced in Section~\ref{2.1}. Let $\operatorname{Cl}_d$ be the subalgebra generated by $c_1,\dots,c_d$. For given a skew shape $\lambda/\mu$, Hill-Kujawa-Sussan defined a module $\hat{H}^{\lambda/\mu}$, which is free over $\operatorname{Cl}_d$ with basis $\{u_T\}_{T\in I_{\lambda/\mu}}$. Here, $I_{\lambda/\mu}$ is the set of standard tableaux of shape $\lambda/\mu$, filled with with each of the entries $1,\dots,d$ exactly once. The $H_d$-module structure on $\hat{H}^{\lambda/\mu}$ is defined as follows: using relations (\ref{cliffordrelations}, \ref{affinecliffordtwist}), the generators $x_i$ and $s_i$ can be moved past $c_i$ and act on $u_T$ directly. The action of $x_i$ is given via $x_i.u_T=\kappa_T(i)u_T$ for $1\leq i\leq d$, similar as the construction given in the beginning of Section~\ref{4.2}, with the understanding that $x_i$ is a generator in $H_d$. Similarly, the action of $s_i \in H_d$ is given as (\ref{transpositionsaction}). They further checked that all relations in $H_d$ are satisfied:
\begin{prop}(\cite[Proposition 5.1.1]{HKS})\label{HKSaction}
The action of $x_i$ and $s_i$ above endows $\hat{H}^{\lambda/\mu}$ with a well-defined $H_d$-module structure.
\end{prop}

Since all relations in $H_d$ are already relations in $\mathcal{H}^{\operatorname{ev}}_{p,d}$ or $\mathcal{H}^{\operatorname{ev}}_{p,d}$, there exists following algebra maps $\phi_1:H_d \to \mathcal{H}^{\operatorname{ev}}_{p,d}$ and $\phi_2:H_d \to \mathcal{H}^{\operatorname{od}}_{p,d}$, sending $x_i\mapsto z_i$ for $1\leq i\leq d$, and $c_i$, $s_i$ to generators under the same name. Given a partition $\mu\in \mathcal{P}_0(\alpha)$, and let $\hat{D}^{\lambda/\mu}$ be the $\operatorname{Cl}_{d}$-module for $\phi_1(H_d)$ generated by vectors $\{v_T\}_{T\in I_{\lambda/\mu}}$. 
Similarly, let $\hat{E}^{\lambda/\mu}$ be the $\operatorname{Cl}_{d}$-module for $\phi_2(H_d)$ generated by vectors $\{v_T\}_{T\in I_{\lambda/\mu}}$. 
We have the following.

\begin{cor}\label{reallylastrelation}
The relations  $s_i {z}_i = {z}_{i+1}s_i-1+c_ic_{i+1}$, $(1\leq i\leq d-1)$ are satisfied in the definition of $\mathcal{D}^{\lambda}_f$ or  $\mathcal{E}^{\lambda}_f$.
\end{cor}
\begin{proof}
This is because relations $s_i {x}_i = {x}_{i+1}s_i-1+c_ic_{i+1}$, $(1\leq i\leq d-1)$ in (\ref{affinecliffordtwist0}) are satisfied for the $H_d$-action on $\hat{H}^{\lambda/\mu}$, by Proposition~\ref{HKSaction}. Moreover, the construction of $\hat{H}^{\lambda/\mu}$ and $\hat{D}^{\lambda/\mu}$ (as a subspace of $\mathcal{D}^{\lambda}_f$) implies that the map $\rho: \hat{H}^{\lambda/\mu} \to \hat{D}^{\lambda/\mu}$, $u_T\ \mapsto v_T$ intertwines the map $\phi$, i.e. $a.u_T=\phi(a)(\rho(u_T))$ for any generator $a\in H_d$, therefore the relation $\phi(s_i {x}_i )= \phi({x}_{i+1}s_i-1+c_ic_{i+1})$ is satisfied by the action of $\phi(H_d)$ on $\hat{D}^{\lambda/\mu}$. This is the relation in the lemma. In addition, $\mathcal{D}^{\lambda}_f\simeq \oplus_{\mu \in \mathcal{P}_0(\alpha,\beta)}\hat{D}^{\lambda/\mu}$ as a vector space, and each  $\hat{D}^{\lambda/\mu}$ is invariant under the action of $\phi(H_d)$. All relations $s_i {z}_i = {z}_{i+1}s_i-1+c_ic_{i+1}$ are inside $\phi(H_d)$, therefore it is enough to check these relations on each $\hat{D}^{\lambda/\mu}$. The argument for $\mathcal{E}^{\lambda}_f$ is similar.

\end{proof}

\begin{proof}[Proof of Theorem~\ref{relationshold}]
Since the acton of $s_i$ on a tableau $T$ only interchanges the integers $i$ and $i+1$ while fixing the other entries, both $v_T$ and $v_{s_i.T}$ have the same eigenvalue $\kappa_T(j)=\kappa_{s_i.T}(j)$ for $j\neq i,i+1$. Therefore in the action of $s_i$ in (\ref{transpositionsaction}), $s_i$ preserves the eigenvalue of $v_T$ for all $z_j$ with $j\neq i,i+1$. Relations (\ref{heckerelation3}) hold in the presentation given by Lemma~\ref{evenpresentation}. Similarly, the action of $s_0$ on $T$ fixes entries $2,3,\dots,d$, and $\kappa_T(j)=\kappa_{s_0.T}(j)$ for all $j\neq 0,1$, therefore the action of $x_1$ in (\ref{Cal_2_x1acts}) perserves the eigenvalue of $v_T$ for all $z_j$ with $j\geq 2$, Relation (\ref{between1}) holds. Relations (\ref{zscommute}) hold because $z_i$ acts semisimply. The Clifford relations hold by definition. The rest of the relations are satisfied by Lemma~\ref{firstrelation} through Corollary~\ref{reallylastrelation}.
\end{proof}

\subsection{Isomorphism condition and irreducibility}

The definition of $\mathcal{D}_f$ depends on the choice of a function $f$ satisfying condition (\ref{Cal_2_x1acts}). By the last claim in Lemma~\ref{preparation}, this function $f$ exists and can simply taken as $f(T)=f(s_0.T)=\sqrt{F_T}=\sqrt{F_{s_0.T}}$, therefore we have constructed at least one $\mathcal{D}_f$-module for each partition $\lambda$. Fixing $\lambda$, we now explore whether these modules are isomorphic for different choices of $f$. We first point out an obvious but important fact. Recall that $\Gamma^{\lambda}$ is the set of paths from $\alpha$ to $\lambda$.

\begin{lemma}\label{distincteigenvalues}
A tableau $T\in \Gamma^{\lambda}$ is uniquely determined by its eigenvalues $\kappa_T(1),\dots,\kappa_T(d)$. 
\end{lemma}
\begin{proof}
The path $T$ ends at the partition $\lambda$. Using $\kappa_T(d)=\sqrt{c_T(d)(c_T(d+1))}$ one can recover the content of integer $d$ and identify the last added box, as well as the partition prior to $\lambda$. Similarly, one can identify each of the consecutively added box and identify all partitions in Row $1,\dots,d$. The remaining partition is the one in Row $0$.
\end{proof}

We also need a few lemmas. Let ${e}_T(i)=\sqrt{1-\frac{1}{(\kappa_T(i)+\kappa_T(i+1))^2}-\frac{1}{(\kappa_T(i)-\kappa_T(i+1))^2}}$ be the coefficient in the action of $s_i$ in (\ref{transpositionsaction}).

\begin{lemma}\label{Cen_1_lem}
Fix a path $T\in \Gamma^{\lambda}$ and an integer $i$ ($1\leq i\leq d-1$). If $s_i.T\neq\star$, then the scalar  ${e}_T(i)$ is nonzero.
\end{lemma}
\begin{proof}
By clearing the denominators, it is equivalent to showing the following quantity is never zero, where $a=\kappa_T(i)$ and $b=\kappa_T(i+1)$.
\begin{align}
\mathcal{Y}_T(i)=&(\kappa_T(i)^2-\kappa_{T}(i+1)^2)^2-(\kappa_T(i)-\kappa_{T}(i+1))^2-(\kappa_T(i)+\kappa_{T}(i+1))^2 \label{ydefinition}\\
=&(a^2-b^2)^2-(a-b)^2-(a+b)^2=((a+b)^2-1)((a-b)^2-1)-1 \notag
\end{align}
if $s_i.T\neq\star$, then the boxes containing $i$ and $i+1$ are not adjacent to each other, therefore $|a-b|\geq 2$. Moreover, all boxes in a shifted tableau have contents at least $1$, hence one of $a,b$ is at least $1$ and $|a+b|\geq 3$. Therefore the above quantity is at least $3\cdot 1 -1 =2$.

\end{proof}
	
We also need the following well-known result in combinatorics of Young tableaux. Recall the definition of standard Young tableaux shortly before Section~\ref{4.2}. Given a tableau $T\in \Gamma^{\lambda}$, let $T^{\operatorname{row}}$ be the tableau of the same skew shape, whose entries are consecutive integers $1,2,3,\dots$ filled from left to right in each row, then from the top row to the bottom row.

\begin{lemma}\label{symmetrictransitive}
Given $T\in \Gamma^{\lambda}$, there exists a word $w$ in $s_1,\dots,s_{d-1}$ such that $w.T=T^{\operatorname{row}}$. In other words, the symmetric group $S_d$ acts transitively on the set of standard tableaux of the same skew shape.
\end{lemma}

The following theorem gives a condition on the functions $f,g$ to produce two isomorphic modules.

\begin{theorem}
Let $f,g$ be $\mathbb{C}$-valued functions on ${\Gamma^{\lambda}}$ which satisfy condition $(\ref{Cal_2_f})$, then $D_f^{\lambda}\simeq D_g^{\lambda}$ if and only if there exists a $\mathbb{C}^{\times}$-valued function $H$ on  $\Gamma^{\lambda}$, such that $H$ is constant for all paths that share the same first edge, and for any $T\in \Gamma^{\lambda}$, $s_0.T\neq \star$,
\begin{align}
\frac{f(T)}{g(T)}=\frac{H(T)}{H(s_0.T)}. \label{isomorphismcondition1}
\end{align}
equivalently, a more direct check is the following condition
\begin{align}
\frac{f(T)}{g(T)}=\frac{f(s_i.T)}{g(s_i.T)}, \hspace{.2 in} 2\leq i\leq d. \label{isomorphismcondition2}
\end{align}
\end{theorem}

\begin{proof}
$D_f^{\lambda}\simeq D_g^{\lambda}\Rightarrow (\ref{isomorphismcondition1})$: Let $\{v_T\}_{T\in \Gamma^{\lambda}}$, $\{w_T\}_{T\in \Gamma^{\lambda}}$ be the $\operatorname{Cl}_d$-bases of $\mathcal{D}^{\lambda}_f$ and $\mathcal{D}^{\lambda}_g$ in their construction. By Lemma~\ref{distincteigenvalues}, if $T\neq S$ are two paths, the list of eigenvalues for $z_0,\dots,z_d$ is different for $v_T$ and $v_S$. Moreover, by the ``Clifford twist relations'', $x_i$ anticommutes with $c_i$ and commutes with $c_j$ if $j\neq i$, therefore $c_0^{\epsilon_0}c_1^{\epsilon_1}\cdots c_d^{\epsilon_d}v_T$ has eigenvalues $(-1)^{\epsilon_0}\kappa_T(0)$, $(-1)^{\epsilon_1}\kappa_T(1)$, \dots, $(-1)^{\epsilon_d}\kappa_T(d)$, therefore each simultaneous eigenspace has dimension $1$.

If there exists a module isomorphism $\phi:D_f^{\lambda}\to  D_g^{\lambda}$, then $\phi$ preserves each simultaneous eigenspace, and  $\phi(v_T)=H(T)w_T$ for some $H(T)\in \mathbb{C}^{\times}$. Let $d_T,e_T$ be the coefficients in (\ref{transpositionsaction}) such that $s_i.v_T=d_Tv_T+e_Tv_{s_i.T}$, and similarly $s_i.w_T=d_Tw_T+e_Tw_{s_i.T}$. Whenever $s_i.T\neq\star$, 
\begin{align*}
\phi(s_i.v_T)&=d_T\phi(v_T)+e_T\phi(v_{s_i.T})=d_TH(T)w_T+e_TH(s_i.T)w_{s_i.T},\\
s_i.\phi(v_T)&=s_i(h(T)w_T)=H(T)(d_Tw_T+e_Tw_{s_i.T}).
\end{align*}
By comparing the coefficients of $w_{s_i.T}$, and the fact that $e_T\neq 0$ from Lemma~\ref{Cen_1_lem}, $H(s_i.T)=H(T)$ for any $1\leq i\leq d-1$. By applying all possible permutations, $H$ must be constant on all paths which go through the same partition at Row $0$, or equivalently, constant on all tableaux of the same skew shape.

Let $a_T,b'_T$ be the coefficients defined in (\ref{Cal_2_x1acts}) such that $x_1.v_T=a_Tv_T+f(T)b'_Tv_{s_0.T}$, and similarly $x_1.w_T=a_Tw_T+g(T)b'_Tw_{s_0.T}$. Whenever $s_0.T\neq \star$,
\begin{align*}
\phi(x_1.v_T)&=a_T\phi(v_T)+f(T)b'_T\phi(v_{s_0.T})=a_TH(T)w_T+f(T)b'_TH(s_0.T)w_{s_0.T},\\
x_1.\phi(v_T)&=x_1.(h(T)w_T)=H(T)(a_Tw_T+b'_Tg(T)w_{s_0.T}).
\end{align*}
In (\ref{Cal_2_x1acts}), $b_T'=(\kappa_0-\kappa_0')c_0+(\kappa_1+\kappa_1')c_1\neq 0$. By comparing the coefficients of $w_{s_0.T}$ we obtain ${g(T)}{H(T)}={f(T)}{H(s_0.T)}$. In Lemma~\ref{preparation} we showed $g(T)$ is never zero, and $H(s_0.T)$ is nonzero because $\phi$ is an isomorphism. Therefore we obtain the identity in the claim.

Conversely, if there exists such a function $H$
, then the map $\phi: \mathcal{D}_f^{\lambda}\to \mathcal{D}_g^{\lambda}$ defined by $v_T \mapsto H(T)w_T$ defines a module isomorphism, based on the calculations given above.

$ (\ref{isomorphismcondition1}) \Rightarrow (\ref{isomorphismcondition2})$: Condition (\ref{isomorphismcondition1}) implies that for all $2\leq i\leq d$,
\begin{align*}
\frac{f(s_i.T)}{g(s_i.T)}=\frac{H(s_i.T)}{H(s_0s_i.T)}=\frac{H(s_i.T)}{H(s_is_0.T)}=\frac{H(T)}{H(s_0.T)}=\frac{f(T)}{g(T)}.
\end{align*}
$  (\ref{isomorphismcondition2})\Rightarrow (\ref{isomorphismcondition1}) $: Any partition in Row $0$ of the Bratteli graph is uniquely determined by the number of boxes in the first row, that is, there is a bijection
\begin{align*}
\phi:\mathcal{P}_0 \to I=\{n+p,n+p-1,\dots,\operatorname{max}\{n+1,p+1\}\}, \hspace{.2 in}\phi(\mu)=\mu_1.
\end{align*}
Recall that $T^{(0)}$ is the target of the first edge in $T$, or equivalently the partition in $\mathcal{P}_0$ among vertices in $T$.  Let $h:I \to \mathbb{C}$ be the unique map (up to scalar multiples) determined by $\frac{h(i)}{h(i+1)}=\frac{f(T)}{g(T)}$ for all suitable $i$, where $T$ is a tableau such that $\phi(T^{(0)})=i$ and $\phi(s_0.T^{(0)})=i+1$. The function $h$ is well-defined, because even though the fraction $\frac{f(T)}{g(T)}$ depends on the choice of $T$, the first two edges in $T$ are fixed, and $T$ can only be paths related by operations of $s_2,\dots,s_{d-1}$, which leaves $\frac{f(T)}{g(T)}$ invariant by condition (\ref{isomorphismcondition2}).

Define $H$ to be the function on $\Gamma^{\lambda}$ such that $H(T)=h(i)$ if $\phi(T^{(0)})=i$ and $\phi(s_0.T^{(0)})=i+1$, then condition (\ref{isomorphismcondition1}) is satisfied by design.
\end{proof}

We now aim to show that the modules $\mathcal{D}^{\lambda}_f$ and $\mathcal{E}^{\lambda}_f$ are simple. Recall the definition of $T^{\operatorname{row}}$ shorlty before Lemma~\ref{symmetrictransitive}. Specifically, the set $\Gamma^{\lambda}$ contains tableaux of various skew shapes. We now show the analogue of Lemma~\ref{symmetrictransitive} once we include the action of $s_0$.

\begin{lemma}\label{Cal_3_comb}
The set $\{s_0,\dots,s_{d-1}\}$ acts on $\Gamma^{\lambda}\cup \{\star\}$ transitively.
\end{lemma}
\begin{proof}
Given two tableaux $T_1,T_2\in \Gamma^{\lambda}$, it is enough to show there exists a word $w$ in $s_0,\dots,s_{d-1}$ such that $w.T_1^{\operatorname{row}}=T_2^{\operatorname{row}}$. We first assume the unfilled boxes in $T_1^{\operatorname{row}}$ and $T_2^{\operatorname{row}}$ (or equivalently, the targets of the first edge in these two paths) differ by a single box in the first row. The other cases can be obtained from applying the following moves repeatly. In the picture we omit the staircase portion, as well as entries $4,5,6,\dots$, where empty boxes are labeled green.\\

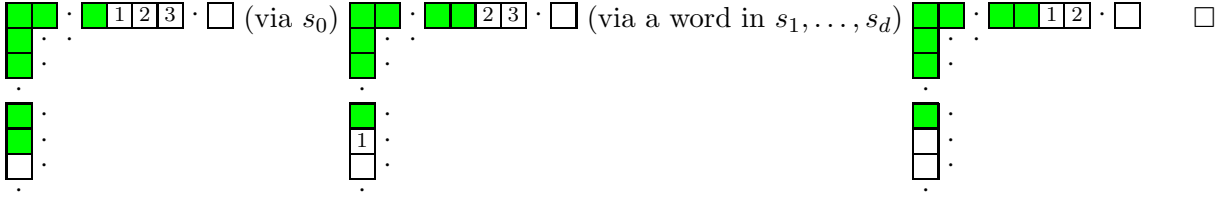
\ytableausetup{smalltableaux}
\begin{ytableau}
*(green)&*(green) &  \none[\cdot]  &*(green) &1 & 2& 3& \none[\cdot]  &\\
 *(green) &  \none[\cdot] &\none[\cdot]  \\
 *(green) & \none[\cdot]\\
 \none[\cdot]\\
 *(green) &\none[\cdot]\\
  *(green) & \none[\cdot]\\
  & \none[\cdot]\\
 \none[\cdot]
\end{ytableau}
(via $s_0$)
\begin{ytableau}
*(green)&*(green) &  \none[\cdot]  &*(green) & *(green) & 2& 3 & \none[\cdot]  &\\
 *(green) &  \none[\cdot] &\none[\cdot]  \\
 *(green) & \none[\cdot]\\
 \none[\cdot]\\
 *(green) & \none[\cdot]\\
 1 &\none[\cdot]\\
  & \none[\cdot]\\
 \none[\cdot]
\end{ytableau}
(via a word in $s_1,\dots,s_d$)
\begin{ytableau}
*(green)&*(green) &  \none[\cdot]  &*(green) & *(green) & 1& 2 & \none[\cdot]  &\\
 *(green) &  \none[\cdot] &\none[\cdot]  \\
 *(green) & \none[\cdot]\\
 \none[\cdot]\\
 *(green) & \none[\cdot]\\
  &\none[\cdot]\\
  & \none[\cdot]\\
 \none[\cdot]
\end{ytableau}
\end{proof}

\begin{theorem}\label{irreducibility}
The modules $\mathcal{D}^{\lambda}_f$ and $\mathcal{E}^{\lambda}_f$ are simple.
\end{theorem}
\begin{proof}
We first give the proof when $n$ is even. When $n$ is odd, the proof is very similar. Let $E=\{0,1\}^{d+1}$. For any sequence $\epsilon=(\epsilon_0,\dots,\epsilon_d)\in E$ and $T\in \Gamma^{\lambda}$ define

\begin{align}
P_{{T},\epsilon}&=\prod_{\sigma\in E, {S}\in{\Gamma^{\lambda}},
{S}\neq{T}} \frac{(z_0-(-1)^{\sigma_0}\kappa_0({S}))^2+\cdots+(z_d-(-1)^{\sigma_d}\kappa_d({S}))^2}{((-1)^{\epsilon_0}\kappa_0({T})-(-1)^{\sigma_0}\kappa_0({S}))^2+\cdots+((-1)^{\epsilon_d}\kappa_d({T})-(-1)^{\sigma_d}\kappa_d({S}))^2}\notag \\
&\hh \cdot \prod_{\sigma \in E, \sigma\neq \epsilon} \frac{(z_0-(-1)^{\sigma_0}\kappa_0({T}))^2+\cdots+(z_d-(-1)^{\sigma_d}\kappa_d({T}))^2}{((-1)^{\epsilon_0}\kappa_0({T})-(-1)^{\sigma_0}\kappa_0({T}))^2+\cdots+((-1)^{\epsilon_d}\kappa_d({T})-(-1)^{\sigma_d}\kappa_d({T}))^2} \label{Cal_3_proj}
\end{align}

Let $c^{\epsilon}=c_1^{\epsilon_1}\cdots c_d^{\epsilon_d}$. Since $z_i. (c^{\epsilon}v_{T})=(-1)^{\epsilon_i} c^{\epsilon}v_{T}$, it follows that $P_{{T},\epsilon}:\mathcal{D}^{\lambda}_f\to \mathcal{D}^{\lambda}_f$ acts as a projection onto the space spanned by $c^{\epsilon}v_{T}$:
\begin{align*}
P_{{T},\epsilon}. (c^{\sigma}v_S)=0, \hspace{.1 in} (S\neq T) \hspace{.3 in} P_{{T},\epsilon}.(c^{\sigma}v_T)=0, \hspace{.1 in}(\sigma\neq \epsilon) \hspace{.3 in} P_{{T},\epsilon}.(c^{\epsilon}v_T)=c^{\epsilon}v_T
\end{align*}

Assume on the contrary that $\mathcal{D}^{\lambda}_f$ is not simple. Let $W$ be a proper submodule of $\mathcal{D}^{\lambda}$, then there must exist ${T}\in {\Gamma}$ and $\sigma \in E$ such that $P_{T,\sigma}w \neq 0$, and  $_{T,\sigma}w=ac^{\sigma}v_T$ for some $a\neq 0$. Recall the scalar $e_T(i)$ defined in Lemma~\ref{Cen_1_lem} based on (\ref{transpositionsaction}), so that by definition if $s_i.T\neq \star$,
\begin{align*}
\left(s_i+\frac{1}{\kappa_T(i)-\kappa_{T}(i+1)}-\frac{1}{\kappa_T(i)+\kappa_T(i+1)}c_ic_{i+1} \right) v_T=e_T(i)v_{s_i.T}
\end{align*} 
In Lemma~\ref{Cen_1_lem} we showed $e_T(i)\neq 0$, therefore $v_{s_i.T}\in W$ if $v_{T}\in W$, for any $1\leq i\leq d-1$.

On the other hand, by (\ref{Cal_2_x1acts}) if $s_0.T\neq \star$,
\begin{align*}
&\hh ((\kappa_T(0)-\kappa'_T(0))c_0+(\kappa_T(1)+\kappa'_T(1))c_1)\\
&\hh \cdot (x_1-\frac{N_0}{\kappa^2}(\kappa_T(1)-\kappa_T(0)c_0c_1)-c(\kappa_T(0)+\kappa_T(1)c_0c_1))v_T\\
&= ((\kappa_T(0)-\kappa'_T(0))^2+(\kappa_T(1)+\kappa'_T(1))^2)v_{s_0.T}
\end{align*}
The coefficient of $v_{s_0.T}$ is nonzero since  $\kappa_T(1)\neq 0$, therefore $v_{s_0.T}\in W$ if $v_{T}\in W$.

By Lemma~\ref{Cal_3_comb}, the set $\{s_0,\dots,s_{d-1}\}$ acts transitively on $\Gamma^{\lambda}$, therefore $v_T\in W$ for all $T\in \Gamma^{\lambda}$, and $c^{\epsilon}v_T\in W$ for all $\epsilon\in E$. This contradicts the fact that $W$ is proper.

\end{proof}

\subsection{Classfying calibrated modules}
In fact, the action of $x_1$ and transpositions $s_i$ in the definition of $\mathcal{D}^{\lambda}_f$  and $\mathcal{E}^{\lambda}_f$ are determined by the action of $z_0,\dots,z_d$, up to a choice of the function $f$.
\begin{theorem}\label{classify}
Assume $n$ is even, and furthermore, $n^2(n+1)^2+p^2(p+1)^2$ is not a perfect square. Fix $\lambda$ in Row $d$ of the Bratteli diagram $\Gamma$. Let $\mathcal{W}$ be a module for $\mathcal{H}^{\operatorname{ev}}_{p,d}$. If $\mathcal{W}$ is a free module over $\operatorname{Cl}_{d+1}$ with basis $\{v_T\}_{T\in \Gamma^{\lambda}}$, and $z_i.v_T=\kappa_T(i)v_T$ for all $T\in \Gamma^{\lambda}$ and $0\leq i\leq d$, then $W\simeq \mathcal{D}^{\lambda}_f$ for some choice of $f$.
\end{theorem}

We break down the proof into two lemmas.
\begin{lemma}\label{x1determined}
Let $\mathcal{W}$ be an $\mathcal{H}^{\operatorname{ev}}_{p,d}$-module satisfying the condition in Theorem~\ref{classify}, then the action of $x_1$ is given by (\ref{Cal_2_x1acts}) for some choice of the function $f$. If $s_0.T\neq \star$, the parity of $v_T$ differs from the parity of $v_{s_0.T}$.
\end{lemma}
\begin{proof}

Since $x_1z_j=z_jx_1$ for $j\geq 2$, $x_1$ preserves the subspace whose eigenvalues are $\kappa_T(2),\dots,\kappa_T(d)$ for $z_2,\dots,z_d$. 	These eigenvalues determine the last $d-1$ edges in $T$, and $T$ has fixed vertices except at Row $0$. Therefore this simultaneous eigenspace is spanned by vectors $v_T$, $c_0c_1v_T$, $c_0v_S$, $c_1v_S$, $v_S$, $c_0c_1v_S$, $c_0v_T$, $c_1v_T$ over $\mathbb{C}$, where $S=s_0.T$ (these vectors are arranged in this order for parity considerations). The linear maps will be given as matrices in this basis. Let $d(X_1,\dots,X_t)$ be the block diagonal matrix with diagonal blocks $X_1,\dots,X_t$. Recall the matrix $C$ in \ref{Cdefinition}, the matrix $D(a,b)$ in (\ref{CliffordD}) for $a,b\in \mathbb{C}$ and let $H=D(1,0)$. Then it is straightforward to calculate that the matrix associated to the action of the following elements.
\begin{align*}
c_0c_1=d(C,C,C,C), \hspace{.3 in} 
z_0 =d(\kappa_0 H, -\kappa_0' H,\kappa_0'H ,-\kappa_0 H), \hspace{.3 in}
z_1 = d(\kappa_1 H , \kappa_1' H , \kappa_1'H , \kappa_1 H)
\end{align*}
Therefore $z_0c_0c_1+z_1=d(P_1 ,P_2)$,
where
\begin{align}
P_1= d(
\kappa_0 HJ+\kappa_1 H ,-\kappa_0'HJ+\kappa_1'H), \hspace{.3 in} P_2= d(
\kappa_0' HJ+\kappa_1' H, -\kappa_0 HJ+\kappa_1 H) \label{Pdefinition}
\end{align}

Assume 
$
x_1=\begin{bmatrix}
A & B\\ C & D
\end{bmatrix}$, where $A=\begin{bmatrix}
A_{11} & A_{12} \\A_{21} & A_{22}
\end{bmatrix}
$
with each $A_{ij}$ being a $2\times 2$ matrix, and the notation is similar for Blocks $B$, $C$ and $D$. Now we use some of the relations to determine the action of $x_1$.

1) First we use the relation $x_1(c_0c_1)+(c_0c_1)x_1=0$. The matrix form of this relation implies $JA_{ij}+A_{ij}J=0$. If  $A_{ij}=\begin{bmatrix}
a & b \\ c & d
\end{bmatrix}$, then
\begin{align*}
\begin{bmatrix}
-c & -d \\ a & b 
\end{bmatrix}
+\begin{bmatrix}
b & -a \\ d & -c 
\end{bmatrix}=0
\end{align*}
therefore $a=-d$ and $b=c$, and $A_{ij}$ is in the form of $D(a,b)$. Similarly, $B_{ij}$, $C_{ij}$ and $D_{ij}$ are also of this form (with potentially different scalars $a,b$).

2) Next, we combine relations (\ref{importantrelation}) and (\ref{xeigenvalues}) to obtain
\begin{align*}
x_1(z_0c_0c_1+z_1)+(z_0c_0c_1+z_1)x_1=2x_1^2=2n(n+1)=2N_0
\end{align*}
Using the matrix  for $z_0c_0c_1+z_1$ in (\ref{Pdefinition}), and the blocks $A,B,C,D$ in the matrix for $x_1$,
\begin{align}
\begin{bmatrix}
A & B\\ C & D
\end{bmatrix}
\begin{bmatrix}
P_1 & 0 \\ 0 & P_2
\end{bmatrix}+\begin{bmatrix}
P_1 & 0 \\ 0 & P_2
\end{bmatrix}\begin{bmatrix}
A & B\\ C & D
\end{bmatrix}=2N_0 \label{blocksinx1condition}
\end{align}
therefore 
\begin{align}
P_1A+AP_1=2n(n+1),\hspace{.1 in}
P_1B+BP_2 =0,\hspace{.1 in}
P_2C+CP_1=0, \hspace{.1 in} P_2D+DP_2=2n(n+1) \label{blocksinx1}
\end{align}
We now use the first equation:
\begin{align*}
\begin{bmatrix}
\kappa_0 HJ+\kappa_1 H & 0 \\ 0 & -\kappa_0'HJ+\kappa_1'H
\end{bmatrix} \begin{bmatrix}
A_{11} & A_{12} \\ A_{21} & A_{22}
\end{bmatrix} +\begin{bmatrix}
A_{11} & A_{12} \\ A_{21} & A_{22}
\end{bmatrix}\begin{bmatrix}
\kappa_0 HJ+\kappa_1 H & 0 \\ 0 & -\kappa_0'HJ+\kappa_1'H
\end{bmatrix}\\
=2N_0
\end{align*}
which is equivalent to
\begin{align}
(\kappa_0 HJ+\kappa_1 H)A_{11}+A_{11}(\kappa_0 HJ+\kappa_1 H)=&2N_0 \label{smallblock1}\\
(\kappa_0 HJ+\kappa_1 H)A_{12}+A_{12}( -\kappa_0'HJ+\kappa_1'H)=&0\label{smallblock2}\\
 (-\kappa_0'HJ+\kappa_1'H)A_{21}+A_{21}  (\kappa_0 HJ+\kappa_1 H)=&0\label{smallblock3}\\
  (-\kappa_0'HJ+\kappa_1'H)A_{22}+A_{22}(-\kappa_0'HJ+\kappa_1'H)=&2N_0 \label{smallblock4}
\end{align}
By the claim in 1), $A_{11}=D(a,b)$ for some $a,b\in \mathbb{C}$. Once computed more explicitly, (\ref{smallblock1}) becomes
\begin{align*}
\kappa_0 D(
-b,a)+\kappa_1  D(
a, b)+\kappa_0  D(
-b, -a)
+\kappa_1  D(
a, -b)=2N_0, \hspace{.1 in}
a\kappa_1-b\kappa_0&=N_0
\end{align*}
therefore $
a= \frac{N_0}{\kappa^2}\kappa_1 + c \kappa_0$ and $
b=-\frac{N_0}{\kappa^2}\kappa_0 + c \kappa_1$ for some $c\in \mathbb{C}$. Using the matrices $Q$ and $X$ in (\ref{matrixdefinition}) of Lemma~\ref{preparation}, $A_{11}=\frac{N_0}{\kappa^2}Q+cX$ for some undetermined $c\in\mathbb{C}$. By using (\ref{smallblock2}), (\ref{smallblock3}), (\ref{smallblock4}) in similar calculations, we obtain
\begin{align}
A_{11}=\frac{N_0}{\kappa^2}Q+cX,  \hspace{.1 in} A_{12}=e Z,  \hspace{.1 in}
A_{21}= fZ,  \hspace{.1 in}  A_{22}= \frac{N_0}{\kappa^2}R+dY \label{blocksinA}
\end{align}
for the matrices $X$, $Y$, $R$, $Q$, $Z$ defined in (\ref{matrixdefinition}), $Z$ in (\ref{zdefinition}) and some undetermined $c,e,f,d \in \mathbb{C}$.

The steps in determining the block $D$ are similar to that of determining $A$, using the last identity in (\ref{blocksinx1}).  
Define matrices $Q',R',X',Y',Z'$ similar to those in (\ref{matrixdefinition}), by setting $\kappa_0=\kappa_{s_0.T}(0)=\kappa_0'$ and $\kappa_1=\kappa_{s_0.T}(1)=\kappa_1'$ in (\ref{matrixdefinition}), we have
\begin{align*}
D_{11}=\frac{N_0}{\kappa^2}Q'+c'X', \hspace{.1 in}D_{12}=e' Z',\hspace{.1 in}D_{21}= f'Z, \hspace{.1 in} D_{22}= \frac{N_0}{\kappa^2}R'+d'Y'
\end{align*}
for some undetermined $c',e',f',d'\in \mathbb{C}$.

Using the remaining identities in (\ref{blocksinx1}), we can also determine the Blocks $B$ and $C$. In particular, let 
\begin{align}
E=D(
\kappa_0+\kappa_0',\kappa_1+\kappa_1'), \hspace{.1 in} F=D(
-(\kappa_0+\kappa_0'), \kappa_1+\kappa_1') \label{EFdefinition}
\end{align}
then by repeating similar steps for the upper right and lower left blocks in (\ref{blocksinx1condition}),
\begin{align*}
 B_{11} = b_{11} E, \hspace{.1 in}  B_{12} = b_{12} \begin{bmatrix}
0 & 1 \\ 1 & 0
\end{bmatrix}, \hspace{.1 in}
B_{21}= b_{21} \begin{bmatrix}
0 & 1 \\ 1 & 0
\end{bmatrix}, \hspace{.1 in} B_{22} = b_{22} F\\
 C_{11} = c_{11} E, \hspace{.1 in}  C_{12} = c_{12} \begin{bmatrix}
0 & 1 \\ 1 & 0
\end{bmatrix},\hspace{.1 in}
C_{21}= c_{21} \begin{bmatrix}
0 & 1 \\ 1 & 0
\end{bmatrix}, \hspace{.1 in} C_{22} = c_{22} F
\end{align*}
for some undetermined complex numbers $b_{11}$, $b_{12}$, $b_{21}$, $b_{22}$, $c_{11}$, $c_{12}$, $c_{21}$, $c_{22}$. 

3) We now use relations (\ref{xeigenvalues}) and (\ref{yeigenvalues}) to determine the grading on $\mathcal{W}$. We claim that $v_{T}$ and $v_{s_0.T}$ must have different parities if $s_0.T \neq \star$. Since these vectors are homogeneous by assumption, assume on the contrary they have the same parity. Then $x_1$ acts invariantly on the subspace spanned by  $v_T$, $c_0c_1v_T$, $v_S$, $c_0c_1 v_S$. Using the matrix identities (\ref{matrixidentity1}) and (\ref{matrixidentity2}) in Lemma~\ref{preparation},  on this subspace $x_1^2=N_0$ acts by the $4\times 4$ matrix
\begin{align*}
\begin{bmatrix}
A_{11} & B_{11}\\
C_{11} & D_{11}
\end{bmatrix}^2
=\begin{bmatrix}
\frac{N_0}{\kappa^2}Q+cX & b_{11}E\\
c_{11}E & \frac{N_0}{\kappa^2}Q'+c'X'
\end{bmatrix}^2=\begin{bmatrix}
\frac{N_0^2}{\kappa^2}+c^2\kappa^2+b_{11}c_{11}E^2 & b_{11}(c-c')XE \\
c_{11}(c-c')EX & \frac{N_0^2}{\kappa^2}+(c')^2\kappa^2b_{11}c_{11}E^2
\end{bmatrix}
\end{align*}
Using $E$ in (\ref{EFdefinition}) and the fact that $E^2=(\kappa_0+\kappa_0')^2+(\kappa_1+\kappa_1')^2$, we have
\begin{align}
b_{11}c_{11}=\frac{N_0- \frac{N_0^2}{\kappa^2}-c^2\kappa^2}{(\kappa_0+\kappa_0')^2+(\kappa_1+\kappa_1')^2} \label{b11c11}
\end{align}
We now use the relation $(x_1-z_1)^4=p(p+1)(x_1-z_1)^2$ to reach a contradiction. Recall that $H=D(1,0)$, then $z_1$ acts on the subspace spanned by  $v_T$, $c_0c_1v_T$, $v_S$, $c_0c_1 v_S$ via the matrix $z_1=d(\kappa_1H,\kappa_1'H)$. Furthermore, let $x_1z_1+z_1x_1$ act by the matrix 
$
G=\begin{bmatrix}
G_{11} & G_{12} \\G_{21} & G_{22}
\end{bmatrix}$, each $G_{ij}$ being a $2\times 2$ block matrix. In Part 2) we determined the blocks $A_{11}$, $B_{11}$, $C_{11}$, $D_{11}$ up to some scalars, and explicitly
\begin{align}
G&=\begin{bmatrix}
A_{11} & B_{11} \\ C_{11} & D_{11}
\end{bmatrix}\begin{bmatrix}
\kappa_1H & 0 \\ 0 &\kappa_1'H
\end{bmatrix} +\begin{bmatrix}
\kappa_1H & 0 \\ 0 &\kappa_1'H
\end{bmatrix} \begin{bmatrix}
A_{11} & B_{11} \\ C_{11} & D_{11}
\end{bmatrix}\\
G_{11}&=\kappa_1\left(\frac{N_0}{\kappa^2}(HQ+QH)+c(HX+XH)\right)=2\kappa_1\left(\frac{N_0}{\kappa^2}\kappa_1+c\kappa_0\right)\label{G11}\\
G_{22}&=\kappa_1'\left(\frac{N_0}{\kappa^2}(HQ'+Q'H)+c'(HX'+X'H)\right)=2\kappa_1'\left(\frac{N_0}{\kappa^2}\kappa_1'-c\kappa_0'\right)\label{G22}\\
G_{12}&=b_{11}(\kappa_1HE+\kappa_1'EH)=b_{11} \begin{bmatrix}
(\kappa_0+\kappa_0')(\kappa_1+\kappa_1') & \kappa_1^2-(\kappa_1')^2\\  (\kappa_1')^2-\kappa_1^2 & (\kappa_0+\kappa_0')(\kappa_1+\kappa_1')
\end{bmatrix}\label{G12}\\
G_{21}&=c_{11}(\kappa_1'HE+\kappa_1 EH)=c_{11}\begin{bmatrix}
(\kappa_0+\kappa_0')(\kappa_1+\kappa_1') & (\kappa_1')^2-\kappa_1^2 \\ \kappa_1^2- (\kappa_1')^2 & (\kappa_0+\kappa_0')(\kappa_1+\kappa_1')
\end{bmatrix} \label{G21}
\end{align}
Here, the first equality in (\ref{G11}) holds because from (\ref{blocksinA}),
\begin{align*}
G_{11}=&\kappa_1(AH+HA)=\kappa_1(\frac{N_0}{\kappa^2}Q+cX)H+\kappa_1H(\frac{N_0}{\kappa^2}Q+cX)\\
=&\kappa_1\left(\frac{N_0}{\kappa^2}(HQ+QH)+c(HX+XH)\right)
\end{align*}
Then using $H=D(1,0)$ from the beginning of this proof, $Q=D(\kappa_1,-\kappa_0)$, $X=D(\kappa_0,\kappa_1)$ from (\ref{matrixdefinition}), where $
D(a,b)=\begin{bmatrix}
a & b \\ b & -a
\end{bmatrix}$ from (\ref{CliffordD}), it is a straightforward calculation that 
\begin{align*}
HQ+QH=D(1,0)D(\kappa_1,-\kappa_0)+D(\kappa_1,-\kappa_0)D(1,0)=2\kappa_1,\\
HX+XH=D(1,0)D(\kappa_0,\kappa_1)+D(\kappa_0,\kappa_1)D(1,0)=2\kappa_0,
\end{align*}
where the scalars on the right hand sides represent scalar matrices. Hence the second equality in (\ref{G11}) follows. The other blocks in $G$ can be computed similarly.

In particular, if scalars represent the proper $2\times 2$ scalar matrices,
\begin{align}
G_{12}G_{21}=&b_{11}c_{11}\big((\kappa_0+\kappa_0')^2(\kappa_1+\kappa_1')^2+(\kappa_1^2-(\kappa_1')^2)^2\big)\\
=&b_{11}c_{11}(\kappa_1+\kappa_1')^2\big((\kappa_0+\kappa_0')^2+(\kappa_1-\kappa_1')^2\big)\label{productofG12G21}
\end{align}
Using (\ref{anothernokappas}) in Lemma~\ref{kapparelations} and (\ref{b11c11}) above, we have
\begin{align*}
G_{12}G_{21}=&\left(N_0-\frac{N_0^2}{\kappa^2}-c^2\kappa^2\right)\frac{(\kappa_1+\kappa_1')^2\big((\kappa_0+\kappa_0')^2+(\kappa_1-\kappa_1')^2\big)}{(\kappa_0+\kappa_0')^2+(\kappa_1+\kappa_1')^2}\\
=&\left(N_0-\frac{N_0^2}{\kappa^2}-c^2\kappa^2\right)\left(p(p+1)+\frac{4(m+1)(m-p)p}{1+p}\right) 
\end{align*}
To use relation $(x_1-z_1)^4=p(p+1)(x_1-z_1)^2$, notice 
\begin{align*}
(x_1-z_1)^2&=x_1^2+z_1^2-G=N_0+\begin{bmatrix}
\kappa_1^2 I_2 & 0 \\ 0 & (\kappa_1')^2 I_2
\end{bmatrix}-G  
\end{align*}
Therefore
\begin{align}
\left(N_0+\begin{bmatrix}
\kappa_1^2 I_2 & 0 \\ 0 & (\kappa_1')^2 I_2
\end{bmatrix}-G\right)^2-p(p+1)\left(N_0+\begin{bmatrix}
\kappa_1^2 I_2 & 0 \\ 0 & (\kappa_1')^2 I_2
\end{bmatrix}-G\right)&=0 \label{quadraticoriginalform}
\end{align}
Or more explicitly,
\begin{align}
&\hh N_0^2-p(p+1)N_0+\begin{bmatrix}
\kappa_1^4 I_2 & 0 \\ 0 & (\kappa_1')^4 I_2
\end{bmatrix}+G^2+(2N_0-p(p+1))\begin{bmatrix}
\kappa_1^2 I_2 & 0 \\ 0 & (\kappa_1')^2 I_2
\end{bmatrix} \notag \\
&\hh -(2N_0-p(p+1))G -\left(\begin{bmatrix}
\kappa_1^2 I_2 & 0 \\ 0 & (\kappa_1')^2 I_2
\end{bmatrix}G+G\begin{bmatrix}
\kappa_1^2 I_2 & 0 \\ 0 & (\kappa_1')^2 I_2
\end{bmatrix}\right)=0. \label{mostcomplicated}
\end{align}
The top right block is as follows,
\begin{align}
G_{11}G_{12}+G_{12}G_{22} -\big(2N_0-p(p+1)\big)G_{12}-(\kappa_1^2+(\kappa_1')^2)G_{12}=0 \label{toprightincomplicated}
\end{align}

We now assume $G_{12}\neq 0$ (the case when $G_{12}=0$ will be discussed at the end of part 3). Using (\ref{G11}) and (\ref{G22}),
\begin{align*}
2\kappa_1\left(\frac{N_0}{\kappa^2}\kappa_1+c\kappa_0\right)+2\kappa_1'\left(\frac{N_0}{\kappa^2}\kappa_1'-c\kappa_0'\right)-2N_0+p(p+1)-\big(\kappa_1^2+(\kappa_1')^2\big)=0
\end{align*}
Therefore
\begin{align}
c=\left(\frac{N_0}{\kappa^2}\right)\frac{\kappa^2-\kappa_1^2-(\kappa_1')^2}{\kappa_0\kappa_1-\kappa_0'\kappa_1'}+\frac{\kappa_1^2+(\kappa_1')^2-p(p+1)}{2(\kappa_0\kappa_1-\kappa_0'\kappa_1')} \label{constantcthirdtime}
\end{align}
We now check the top left block of (\ref{mostcomplicated}): using (\ref{G11}) and (\ref{G22}), the following quantity must be zero.
\begin{align}
& N_0^2-p(p+1)N_0+\kappa_1^4+G_{11}^2+G_{12}G_{21} +\big(2N_0-p(p+1)\big)\kappa_1^2-\big(2N_0-p(p+1)+2\kappa_1^2\big)G_{11}\label{topleftiszero}\\
=&N_0^2-p(p+1)N_0+\kappa_1^4+\left(\frac{2N_0\kappa_1^2}{\kappa^2}+2c\kappa_0\kappa_1\right)^2\notag\\
&+\left(N_0-\frac{N_0^2}{\kappa^2}-c^2\kappa^2\right)\left(p(p+1)+\frac{4(m+1)(m-p)p}{1+p}\right)\notag\\
&+\big(2N_0-p(p+1)\big)\kappa_1^2-\big(2N_0-p(p+1)+2\kappa_1^2\big)\left(\frac{2N_0\kappa_1^2}{\kappa^2}+2c\kappa_0\kappa_1\right)=a_2N_0^2+a_1 N_0 +a_0\notag
\end{align}
and the last equality is obtained by rearranging the terms as a polynomial in $N_0$, with coefficients $a_2$, $a_1$, $a_0 \in \mathbb{Z}(m,p)$, after subsituting all $\kappa$ values via (\ref{specificvalues1}) and (\ref{specificvalues2}) in Lemma~\ref{kapparelations}, and the expression for $c$ in (\ref{constantcthirdtime}).

In particular, $a_1=0$. This is because, by (\ref{constantcthirdtime})
\begin{align}
&a_1=-p(p+1)+4\kappa_0\kappa_1\left( \frac{2\kappa_1^2}{\kappa^2}+\frac{2\kappa_0\kappa_1\big(\kappa^2-\kappa_1^2-(\kappa_1')^2\big)}{\kappa^2(\kappa_0\kappa_1-\kappa_0'\kappa_1')} \right)\cdot \frac{\kappa_1^2+(\kappa_1')^2-p(p+1)}{2(\kappa_0\kappa_1-\kappa_0'\kappa_1')} \label{a1iszero} \\
&+\left(1-\frac{2\kappa^2\big(\kappa^2-\kappa_1^2-(\kappa_1')^2\big)\big(\kappa_1^2+(\kappa_1')^2-p(p+1)\big)}{2\kappa^2(\kappa_0\kappa_1-\kappa_0'\kappa_1')^2}\right)\cdot\left( p(p+1)+\frac{4(m+1)(m-p)p}{p+1} \right) \notag \\
&+2\kappa_1^2-4\kappa_0\kappa_1\frac{\kappa_1^2+(\kappa_1')^2-p(p+1)}{2(\kappa_0\kappa_1-\kappa_0'\kappa_1')}-\big(-p(p+1)+2\kappa_1^2\big)\left( \frac{2\kappa_1^2}{\kappa^2}+\frac{2\kappa_0\kappa_1\big(\kappa^2-\kappa_1^2-(\kappa_1')^2\big)}{\kappa^2(\kappa_0\kappa_1-\kappa_0\kappa_1')} \right) \notag
\end{align}
This is zero by a direct calculation. We also include MAGMA codes as verification in Section~\ref{a1iszerocodes}.

We also claim that 
\begin{align}
a_2=&1+\left( \frac{2\kappa_1^2}{\kappa^2}+2\kappa_0\kappa_1\frac{\kappa^2-\kappa_1^2-(\kappa_1')^2}{\kappa^2(\kappa_0\kappa_1-\kappa_0'\kappa_1')}\right)\frac{1}{\kappa^2}\left( p(p+1)+\frac{4(m+1)(m-p)p}{1+p}\right) \notag\\
&-2\left( \frac{2\kappa_1^2}{\kappa^2}+2\kappa_0\kappa_1\frac{\kappa^2-\kappa_1^2-(\kappa_1')^2}{\kappa^2(\kappa_0\kappa_1-\kappa_0'\kappa_1')} \right)=\frac{(2m-p+1)^2(m-p)(m+1)}{(p+1)^2(m-p+1)m} \label{a2iszero}
\end{align}
Again, this is a straightforward calculation subsituting all $\kappa$ values via (\ref{specificvalues1}). We also include MAGMA codes as a computer check in Section~\ref{a1iszerocodes}.

We also claim that 
\begin{align}
a_0=&\kappa_1^4+\left(2\kappa_0\kappa_1 \frac{\kappa_1^2+(\kappa_1')^2-p(p+1)}{2(\kappa_0\kappa_1-\kappa_0'\kappa_1')} \right)^2-p(p+1)\kappa_1^2\notag\\
&-\kappa^2\left(\frac{\kappa_1^2+(\kappa_1')^2-p(p+1)}{2(\kappa_0\kappa_1-\kappa_0'\kappa_1')}\right)^2 \left(p(p+1)+\frac{4(m+1)(m-p)p}{p+1}\right)\notag\\
&-\big(p(p+1)+2\kappa_1^2\big)\left(2\kappa_0\kappa_1 \frac{\kappa_1^2+(\kappa_1')^2-p(p+1)}{2(\kappa_0\kappa_1-\kappa_0'\kappa_1')} \right)=\frac{-4p^2m(m+1)(m-p)(m-p+1)}{(p-1)^2}\label{a0iszero}
\end{align}
This is a straightforward calculation subsituting all $\kappa$ values via (\ref{specificvalues1}). We also include MAGMA codes as a computer check in Section~\ref{a1iszerocodes}.

Combining the expressions for $a_0$ and $a_2$ in (\ref{a0iszero}) and (\ref{a2iszero}), we have
\begin{align*}
N_0=\sqrt{\frac{-a_0}{a_2}}=\frac{2mp(m-p+1)(p+1)}{(p-1)(2m-p+1)}
\end{align*}
Or in other words, 
\begin{align*}
2m(m-p+1)=(p-1)(2m-p+1)\frac{N_0}{p(p+1)}
\end{align*}
By arranging this into a polynomial in $m$, it has rational coefficients with discriminant 
\begin{align*}
4(p-1)^2\left(1+\frac{N_0}{p(p+1)}\right)^2-8(p-1)^2\frac{N_0}{p(p+1)}=\frac{4(p-1)^2}{p^2(p+1)^2}\big(p^2(p+1)^2+n^2(n+1)^2\big)
\end{align*}
since $p^2(p+1)^2+n^2(n+1)^2$ is not a perfect square by the assumption in Theorem~\ref{classify}, the solution $m$ will never be rational, contradicting the fact that $m$ is an integer, and relation (\ref{topleftiszero}) is never satisfied.

On the other hand, if $G_{12}=0$ in (\ref{toprightincomplicated}), then by (\ref{productofG12G21}),  either $b_{11}$ or $c_{11}$ is zero. Without a loss of generality, assume $b_{11}=0$, then by the results in Part 2) and (\ref{matrixidentity1}), (\ref{matrixidentity2}) in Lemma~\ref{preparation}, the upper-left block in $x_1^2=N_0$ is
\begin{align}
N_0=A_{11}^2=(\frac{N_0}{\kappa^2}Q+cX)^2=\frac{N_0}{\kappa^2}+c^2\kappa^2, \hspace{.2 in} c^2=\frac{N_0}{\kappa^2}-\frac{N_0^2}{\kappa^4} \label{cfourthtime}
\end{align}
Also, $G_{12}=0$ from (\ref{G12}), therefore $G^2$ is block diagonal, with upper-left block given by $G_{11}^2$. Then the upper-left block in (\ref{quadraticoriginalform}) becomes
\begin{align}
(N_0+\kappa_1^2-G_{11})^2-p(p+1)(N_0+\kappa_1^2-G_{11})=&0\\
G_{11}=\begin{cases} N_0+\kappa_1^2\\
 N_0+\kappa_1^2-p(p+1) \end{cases} \label{G11possibilities}
\end{align}
Using (\ref{cfourthtime}), (\ref{G11}), and the definition $\kappa^2=\kappa_0^2+\kappa_1^2$ shortly before Lemma~\ref{kapparelations}, the first possibility in (\ref{G11possibilities}) becomes
\begin{align*}
4 \left( \frac{N_0}{\kappa^2}-\frac{N_0^2}{\kappa^4} \right)\kappa_0^2\kappa_1^2=&\left( \left(1-\frac{2\kappa_1^2}{\kappa^2}N_0\right)+\kappa_1^2\right)^2\\
\left( \frac{4\kappa_0^2\kappa_1^2}{\kappa^4}+\left( 1-\frac{2\kappa_1^2}{\kappa^2} \right)^2 \right)N_0^2-\left( 2\kappa_1^2\left(1-\frac{2\kappa_1^2}{\kappa^2} \right)-\frac{4\kappa_0^2\kappa_1^2}{\kappa^2} \right)N_0+\kappa_1^4=&0\\
(N_0+\kappa_1^2)^2=&0
\end{align*}
This contradicts $N_0=n(n+1)>0$. Similarly, the second possibility in (\ref{G11possibilities}) reduces to
\begin{align*}
\big(N_0+\kappa_1^2-p(p+1) \big)^2=-2p(p+1)\frac{2\kappa_1^2}{\kappa^2}N_0
\end{align*}
which is also a contradiction.

4) Since $v_T$ and $v_{s_0.T}$ have different parities, the blocks $B$ and $C$ in $x_1$ act as zero, and $x_1$ acts invariantly on the subspace spanned by $\{v_T, c_0c_1v_T, c_0v_S, c_1v_S\}$ via the matrix $A$, whose blocks $A_{11}$, $A_{12}$, $A_{21}$, $A_{22}$ were determined in Part 2). 

We now use the relation $(x_1-z_1)^4-p(p+1)(x_1-z_1)^2=0$. Similar to Part 3), let $L=\begin{bmatrix}
L_{11} & L_{12} \\L_{21} & L_{22}
\end{bmatrix}$ be the matrix of $x_1z_1+z_1x_1$ acting on the subspace spanned by $\{v_T, c_0c_1v_T, c_0v_S, c_1v_S\}$. Recall $H=D(1,0)$, and $z_1$ acts as $d(\kappa_1H,\kappa_1'H)$ on this subspace. A straightforward calculation shows that 
\begin{align}
L_{11}&=\kappa_1(\frac{N_0}{\kappa^2}(KQ+QK)+c(KX+KX))=2\kappa_1(\frac{N_0}{\kappa^2}\kappa_1+c\kappa_0) \label{L11}\\
L_{22}&=\kappa_1'(\frac{N_0}{\kappa^2}(KR+RK)+c(KY+YK))=2\kappa_1'(\frac{N_0}{\kappa^2}\kappa_1'+c\kappa_0')\label{L22}\\
L_{12}&=e(\kappa_1'ZH+\kappa_1HZ)=e(\kappa_1+\kappa_1')\begin{bmatrix}
\kappa_0-\kappa_0' & \kappa_1-\kappa_1' \\-(\kappa_1-\kappa_1') & \kappa_0-\kappa_0' 
\end{bmatrix}\label{L12}\\
L_{21}&=f(\kappa_1ZH+\kappa_1'HZ)=f(\kappa_1+\kappa_1')\begin{bmatrix}
\kappa_0-\kappa_0' & -(\kappa_1-\kappa_1') \\\kappa_1-\kappa_1' & \kappa_0-\kappa_0' 
\end{bmatrix}\label{L21}
\end{align}
In particular,
\begin{align*}
L_{12}L_{21}=ef(\kappa_1+\kappa_1')^2\big((\kappa_0-\kappa_0')^2+(\kappa_1-\kappa_1')^2\big)
\end{align*}

On one hand, $ef\neq 0$. If not, then the same argument at the end of Part 3) applies and provides a contradiction. We now use $(x_1-z_1)^4-p(p+1)(x_1-z_1)^2=0$ and observe that 
\begin{align}
(x_1-z_1)^2&=x_1^2+z_1^2-L=N_0+\begin{bmatrix}
\kappa_1^2 I_2 & 0 \\ 0 & (\kappa_1')^2 I_2
\end{bmatrix}-L  \label{mostcomplicatedagain}
\end{align}
and (\ref{mostcomplicated}) is true by replacing $G$ with $L$. Since $ef\neq 0$, we assume $e\neq 0$ without a loss of generality. The upper right block is given by the following, similar to (\ref{toprightincomplicated}):
\begin{align}
L_{11}L_{12}+L_{12}L_{22} -\big(2N_0-p(p+1)\big)L_{12}-(\kappa_1^2+(\kappa_1')^2)L_{12}=0
\end{align}
If $f\neq 0$, we use the lower left of (\ref{mostcomplicatedagain}) to reach the same conclusion. The matrix $L_{12}$ in (\ref{L12}) is not the zero matrix, therefore the scalar 
\begin{align*}
L_{11}+L_{22} -\big(2N_0-p(p+1)\big)-(\kappa_1^2+(\kappa_1')^2)=&0\\
\left(\frac{2\kappa_1^2}{\kappa^2}N_0+2c\kappa_0\kappa_1\right)+\left(\frac{2(\kappa_1')^2}{\kappa^2}N_0+2c\kappa_0'\kappa_1'\right)-\big(2N_0-p(p+1) \big)-\big(\kappa_1^2+(\kappa_1')^2\big)=&0\\
\frac{2\frac{N_0}{\kappa^2}(\kappa^2-\kappa_1^2-(\kappa_1')^2)+\big(\kappa_1^2+(\kappa_1')^2-p(p+1)\big)}{2(\kappa_0\kappa_1+\kappa_0'\kappa_1')}=&c
\end{align*}
This is consistent with the definition of the constant $c$ in (\ref{cdefinition}), Lemma~\ref{preparation}, in which we showed that the upper-left block in $x_1$ given by $A_{11}=\frac{N_0}{\kappa^2}Q+cX$ is equivalent to the construction (\ref{Cal_2_x1better}). To determine the remaining blocks, let us first show that the constants $e$ and $f$ satisfy the condition in (\ref{Cal_2_f}), or equivalently, (\ref{Cal_2_f0}) in Lemma~\ref{preparation}. This is because $x_1^2=N_0$ should act as
\begin{align*}
&\begin{bmatrix}
\frac{N_0}{\kappa^2}Q+cX & eZ \\ fZ & \frac{N_0}{\kappa^2}+dY
\end{bmatrix}^2\\
=&
\begin{bmatrix}
(\frac{N_0}{\kappa^2}Q+cX)^2 +  efZ^2 & e((\frac{N_0}{\kappa^2}Q+cX)Z+Z(\frac{N_0}{\kappa^2}R+dY)) \\
f (Z(\frac{N_0}{\kappa^2}Q+cX)+(\frac{N_0}{\kappa^2}R+dY)Z) & (\frac{N_0}{\kappa^2}R+dY)^2+  efZ^2
\end{bmatrix}^2\\
=&
\begin{bmatrix}
\frac{N_0^2}{\kappa^2}+c\kappa^2 +  efZ^2 & e(c-d)XZ \\
f (c-d)ZX & \frac{N_0^2}{\kappa^2}R+d\kappa^2+  efZ^2
\end{bmatrix}
\end{align*}
The upper-right and lower-left block implies $c=d$, and the block $A_{22}$ in $x_1$ also follows the form given in (\ref{Cal_2_x1acts}). The upper-left block implies that 
\begin{align*}
\frac{N_0^2}{\kappa^2}+c\kappa^2 +  ef \big((\kappa_0-\kappa_0')^2+(\kappa_1+\kappa_1')^2 \big) =N_0
\end{align*}
and the product $ef$ satisfies condition (\ref{Cal_2_f0}). In Lemma~\ref{preparation} we showed this is equivalent to (\ref{Cal_2_f}) in the original definition.
\end{proof}

In order to determine the action of the simple transpositions, we review the following intertwiners introduced by Nazarov in \cite[Section~3]{Naz}.
\begin{align*}
\Phi_i = s_i (z_i^2-z_{i+1}^2)+(z_i+z_{i+1})-c_0c_1(z_i-z_{i+1})
\end{align*}
and it is straightforward to check that   
\begin{align*}
\Phi_i z_i&=z_{i+1}\Phi_i, \hspace{.3 in}
\Phi_i z_{i+1}=z_i \Phi_i, \hspace{.3 in}
\Phi_i z_j =z_j \Phi_i \hspace{.1 in} (j\neq i,i+1)\\
\Phi_i^2 &= 2(z_i^2+z_{i+1}^2)-(z_i^2-z_{i+1}^2)^2= -(z_i^2-z_{i+1}^2)^2+(z_i-z_{i+1})^2+(z_i+z_{i+1})^2
\end{align*}
Also recall $\mathcal{Y}_T(i)$ defined in (\ref{ydefinition}) of Lemma~\ref{Cen_1_lem}. Since the action of $s_i$ on $T$ interchanges the eigenvalues for $z_i$ and  $z_{i+1}$, it follows that 
$\mathcal{Y}_{T}(i)=\mathcal{Y}_{s_i.T}(i)$ and $\Phi^2_i v_T =-\mathcal{Y}_{T}(i) v_T$.

Based on the assumptions in Theorem~\ref{classify}, let $\mathcal{W}$ be a free $\operatorname{Cl}_{d+1}$-module with basis $\{w_{T}\}_{T\in \Gamma^{\lambda}}$, with $z_i.w_T=\kappa_T(i)w_T$. Using the above intertwiners, we have
\begin{align*}
\kappa_T(i)\Phi_i w_T&=\Phi_i (z_i w_T)= z_{i+1}(\Phi_i w_T)\\
\kappa_T(i+1)\Phi_i w_T&=\Phi_i (z_{i+1} w_T)= z_{i}(\Phi_i w_T)\\
\kappa_T(j)\Phi_i w_T&=\Phi_i (z_j w_T)= z_{j}(\Phi_i w_T) \hspace{.5 in} j\neq i,i+1
\end{align*}
Therefore $\Phi_i w_T$ has eigenvalues $\kappa_T(0),\dots,\kappa_T(i+1),\kappa_T(i),\dots,\kappa_T(d)$ under the action of $z_0,\dots,z_d$. On the other hand, the simultaneous eigenspaces for $\kappa_T(0),\dots,\kappa_T(i+1),\kappa_T(i),\dots,\kappa_T(d)$ are all one-dimensional: the list of eigenvalues for $w_T$ is distinct by Lemma~\ref{distincteigenvalues}, with all positive eigenvalues. A general vector $c_0^{\epsilon_0}c_1^{\epsilon_1}\cdots c_d^{\epsilon_d}v_T$ has eigenvalue $(-1)^{\epsilon_i}\kappa_T(i)$ for $z_i$, whose list of eigenvalues is distinct from all other choices of the tuple $(\epsilon_0,\epsilon_1,\dots,\epsilon_d)$.

Since $w_{s_i.T}$ is a vector with the same list of eigenvalues as $\Phi_i w_T$, $\Phi_i w_T$ is a scalar multiple of $w_{s_i.T}$. By a similar argument, $\Phi_i w_{s_i.T}$ is also a scalar multiple of $w_T$. Recall the nonzero constants $\mathcal{Y}_T(i)$ defined in (\ref{ydefinition}). Let $a_T(i)\in \mathbb{C}$ be the unique nonzero scalar such that $\Phi_i w_T={a_T(i)}\sqrt{-\mathcal{Y}_T(i)} w_{s_i.T}$. Even though the simple transpositions $s_i$ may not act exactly by (\ref{transpositionsaction}), we can use scalar multiples of $w_T$ to produce a new $\operatorname{Cl}_{d+1}$-basis on which $s_i$ acts as (\ref{transpositionsaction}). 

Recall that the generators $s_1,\dots,s_{d-1}$ satisfy exactly those relations for the symmetric group in (\ref{symmetricgrouprelations}), therefore they generate the symmetric group $S_d$, which is a Coxeter group. For an element $w\in S_d$, define the \emph{length} $\ell(w)$ of $w$ to be the smallest integer $t$ such that $w$ can be written as $w=s_{i_1}s_{i_2}\cdots s_{i_t}$, where $1\leq i_1,\dots,i_t\leq d-1$, and we call such an expression \emph{reduced}. In particular, it is straightfoward to see that two reduced expressions of the same $w\in S_d$ are related only by relations in (\ref{symmetricgrouprelations}) excluding $s_i^2=1$, $(1\leq i\leq d-1)$. For other material related to this definition, we refer readers to \cite[Section~1.6]{Hum}.

\begin{lemma}\label{rescalingvectors}
There is a well-defined map $\phi:\Gamma^{\lambda}\to \mathbb{C}^{\times}$ such that $\sqrt{-1}\phi({s_i.T}) =a_T(i)\phi(T)$ for any $1\leq i\leq d$ and any $T\in \Gamma^{\lambda}$ with $s_i.T\neq \star$. 
\end{lemma}

\begin{proof}
If $T,S\in \Gamma^{\lambda}$ are tableaux of the same skew shape (or equivalently the first edge in either path ends at the same partition $T^{(0)}=S^{(0)}$), by Lemma~\ref{symmetrictransitive}, they are related by an element in the symmetric group $S_d$. Conversely, the action of $S_d$ does not affect $T^{(0)}$ or $S^{(0)}$, therefore if $T$ and $S$ are related by the action of $S_d$, they must be tableaux of the same skew shape. We now fix a skew shape and focus only on tableaux with such skew shape. Fix any $T_0$, and let $\phi({T_0})=1$. Let $S$ be another tableau with the same skew shape, and $S=w.T_0$ for a reduced word $w\in S_d$. Define $\phi$ on $S$ by induction on the length of $w$: $\sqrt{-1}\phi(s_i.T)=a_T(i){T}$. We claim that this map is well-defined independent of the reduced expression. In other words, it is invariant under the Coxeter relations: $\phi({s_is_j.T})=\phi({s_js_i.T})$ for $|i- j|>1$ and $\phi({s_is_{i+1}s_i.T})=\phi({s_{i+1}s_is_{i+1}.T})$ for $1\leq i\leq d-2$.

1) If $|i-j|>1$, then $\Phi_i\Phi_j=\Phi_j\Phi_i$. By the definition of $a_T(i)$,
\begin{align*}
\Phi_i\Phi_j w_T  &=\Phi_i ({a_T(j)}\sqrt{-\mathcal{Y}_T(j)} w_{s_j.T})={a_T(j) a_{s_j.T}(i)} \sqrt{(-\mathcal{Y}_T(j))(-\mathcal{Y}_{s_j.T}(i))}w_{s_is_j.T}\\
\Phi_j \Phi_i w_T  &=\Phi_j ({a_T(i)}\sqrt{-\mathcal{Y}_T(i)} w_{s_i.T})={a_T(i) a_{s_i.T}(j)} \sqrt{(-\mathcal{Y}_T(i))(-\mathcal{Y}_{s_i.T}(j))}w_{s_js_i.T}
\end{align*}
Since $\mathcal{Y}_{s_i.T}(j)=\mathcal{Y}_{T}(j)$ and $\mathcal{Y}_T(i)=\mathcal{Y}_{s_j.T}(i)$,  it follows that $a_{s_j.T}(i)a_T(j) = a_{s_i.T}(j)a_T(i)$ for any tableau $T$, therefore 
\begin{align*}
\phi({s_is_j .T}) =  -a_{s_j.T}(i)   a_{T}(j) w_T=-a_{s_i.T}(j)a_T(i)=\phi( {s_js_i .T}) 
\end{align*}

2) Similarly, 
\begin{align}
\Phi_{i+1}\Phi_i\Phi_{i+1}w_T&={a_{s_i s_{i+1}.T}(i+1) a_{s_{i+1}.T}(i)a_T(i+1)}\\
\cdot& \sqrt{(-\mathcal{Y}_{s_is_{i+1}.T}(i+1))(-\mathcal{Y}_{s_{i+1}.T}(i))(-\mathcal{Y}_T(i+1))}w_T \label{checkbraid1}\\
\Phi_{i}\Phi_{i+1}\Phi_{i}w_T&={a_{s_{i+1} s_{i}.T}(i) a_{s_{i}.T}(i+1)a_T(i)}\sqrt{(-\mathcal{Y}_{s_{i+1}s_{i}.T}(i))(-\mathcal{Y}_{s_{i}.T}(i+1))(-\mathcal{Y}_T(i))} w_T \label{checkbraid2}
\end{align}

On the other hand, we claim that the following is true
\begin{align*}
\mathcal{Y}_{T}(i) \mathcal{Y}_{s_i.T}(i+1) \mathcal{Y}_{s_{i+1}s_i.T}(i)=\mathcal{Y}_{T}(i+1) \mathcal{Y}_{s_{i+1}.T}(i) \mathcal{Y}_{s_is_{i+1}.T}(i+1)
\end{align*}

Let $\kappa_i=\kappa_T(i),\kappa_{i+1}=\kappa_T(i+1),\kappa_{i+2}=\kappa_T(i+2)$, then the eigenvalues for various paths are as follows
\begin{center}
\begin{tabular}{c|c|c|c}
\hline
  & $\kappa(i)$ & $\kappa(i+1)$  & $\kappa(i+2)$ \\
  $T$ & $\kk_i$ & $\kk_{i+1}$ & $\kk_{i+2}$\\
$s_i.T$ & $\kk_{i+1}$&$\kk_i$ &$\kk_{i+2}$\\
$s_{i+1}.T$ & $\kk_i$ & $\kk_{i+2}$&$\kk_{i+1}$\\
$s_is_{i+1}.T$ &$\kk_{i+2}$ & $\kk_i$&$\kk_{i+1}$\\
$s_{i+1}s_{i}.T$ &$\kk_{i+1}$ &$\kk_{i+2}$ & $\kk_i$\\
\hline 
\end{tabular}
\end{center}

If $\mathcal{Y}(a,b)=(a^2-b^2)^2-(a-b)^2-(a+b)^2$, then both sides are equal to the quantity of $\mathcal{Y}(\kk_i, \kk_{i+1})\mathcal{Y}(\kk_{i+1},\kk_{i+2})\mathcal{Y}(\kk_i,\kk_{i+2})$.

By comparing (\ref{checkbraid1}) and (\ref{checkbraid2}), $a_{s_i s_{i+1}.T}(i+1) a_{s_{i+1}.T}(i)a_T(i+1)=a_{s_{i+1} s_{i}.T}(i) a_{s_{i}.T}(i+1)a_T(i)$. Therefore
\begin{align*}
&\hh \phi({s_is_{i+1}s_i .T})=\sqrt{-1}a_{s_{i+1} s_{i}.T}(i) a_{s_{i}.T}(i+1)a_T(i) w_T\\
&=\sqrt{-1}a_{s_i s_{i+1}.T}(i+1) a_{s_{i+1}.T}(i)a_T(i+1) w_T =\phi({s_{i+1}s_i s_{i+1} .T}),
\end{align*}
since $\phi$ satisfies the required condition by design, we have proved the lemma.

\end{proof}

\begin{lemma}\label{sidetermined}
Let $\mathcal{W}$ be a $\mathcal{H}^{\operatorname{ev}}_{p,d}$-module satisfying the condition in Theorem~\ref{classify}, and $\{w_T\}$ be the given $\operatorname{Cl}_{d+1}$-basis. Then after rescaling the vectors $v_T=\phi(T)w_T$ using the map $\phi$ in Lemma~\ref{rescalingvectors}, the action of $s_i$ is given by (\ref{transpositionsaction}). Moreover, if $s_i.T \neq \star$, then $v_T$ and $v_{s_i.T}$ have the same parity.
\end{lemma}

\begin{proof}
First, observe
\begin{align*}
\Phi_i^2w_T=\Phi_i\left( {a_T(i)}\sqrt{-\mathcal{Y}_T(i)} w_{s_0.T}\right ) =-{a_T(i)a_{s_i.T}(i)}\mathcal{Y}_T(i)w_T.
\end{align*}

Since $\Phi_i^2$ acts on $v_T$ by the scalar $-\mathcal{Y}_T(i)$, it follows that $a_T(i)a_{s_i.T}(i)=1$. 

By the definition of $a_T(i)$ shortly before Lemma~\ref{rescalingvectors}, 
\begin{align*}
&{a_T(i)}\sqrt{-\mathcal{Y}_T(i)}w_T=\Phi_i.w_T\\
=&s_i(z_i^2-z_{i+1}^2)w_T+(z_i+z_{i+1})w_T-c_0c_1(z_i-z_{i+1})w_T\\
=&s_i (\kappa_T(i)^2-\kappa_T(i+1)^2)w_T+(\kappa_T(i)+\kappa_T(i+1))w_T-c_0c_1(\kappa_T(i)-\kappa_T(i+1))w_T.
\end{align*}
Therefore
\begin{align*}
s_i.w_T=&-\frac{1}{\kappa_T(i)-\kappa_T(i+1)}w_T+\frac{1}{\kappa_T(i)+\kappa_{T}(i+1)} c_0c_1w_T+\frac{{a_T(i)}\sqrt{-\mathcal{Y}_T(i)}}{\kappa_T(i)^2-\kappa_T(i+1)^2}w_{s_i.T}.
\end{align*}
We now multiply both sides by $\phi(T)$, then use the fact that $\phi(T)a_T(i)=\sqrt{-1}\phi({s_i.T})$, as well as $v_T=\phi(T)w_T$ for both $T$ and $s_i.T$,
\begin{align*}
s_i.v_T=-\frac{1}{\kappa_T(i)-\kappa_T(i+1)}v_T+\frac{1}{\kappa_T(i)+\kappa_{T}(i+1)} c_0c_1v_T+\frac{\sqrt{\mathcal{Y}_T(i)}}{\kappa_T(i)^2-\kappa_T(i+1)^2}v_{s_i.T}.
\end{align*}
This is consistent with (\ref{transpositionsaction}). Since $s_i$ is odd, the right hand side is a homogeneous vector with the same parity as $v_T$, therefore $v_{s_i.T}$ has the same parity as $v_T$.
\end{proof}

\begin{proof}[Proof of Theorem~\ref{classify}]
This is simply a result of Lemma~\ref{x1determined} and \ref{sidetermined}.
\end{proof}

By definition, $\mathcal{H}^{\operatorname{od}}_{p,d}\simeq\mathcal{H}^{\operatorname{ev}}_{p,d}\otimes \operatorname{Cl}_1 $, and it is known that their module categories are related by taking the Clifford twist of each other. For a superalgebra $A$ and its module $W$, let $\Pi W$ be the module with the same underlying vector space with reversed grading. For $v\in W$, if we use $\Pi v$ to denote the same vector $v$ in the module $\Pi W$, a homogeneous $x \in A$ acts on $\Pi W$ via $x.(\Pi v)=(-1)^{\overline{x}}\Pi(x.v)$. The theory of taking the Clifford twist of a category is well-developed by Kang-Kashiwara-Tsuchioka \cite{Kashiwara} and Brundan-Davidson \cite{Brundan3}. In particular, for a superalgebra $A$, recall the superalgebra $A\otimes \operatorname{Cl}_1$ defined in Section~\ref{2.4}. Let $A$-{smod} be the category whose objects are $A$-supermodules and whose morphisms are linear combinations of homogeneous $A$-supermodule homomorphisms (we will drop the prefix ``super'' and simply refer to them as modules and module homomorphisms). 

In addition, recall the definition of $\q(n)$-modules of Type Q and Type M in Lemma~\ref{Pre_3_superschur}. This notion can be generalized to a given module $V$ for an arbitrary superalgebra $A$: $V$ is defined to be of Type M if and only if $\operatorname{End}_A(V)$ is one-dimensional, and Type Q if and only if $\operatorname{End}_A(V)$ is spanned by ${\operatorname{id}}_V$ and an odd map. Then $A$-{smod} and $A\otimes \operatorname{Cl}_1$-{smod} are related in the following sense: if $U$ is a Type M $A$-module, and $\operatorname{Cl}_1$ is the regular module for $\operatorname{Cl}_1$, then $U\otimes \operatorname{Cl}_1$ is known to be a Type Q $A\otimes \operatorname{Cl}_1$-module, and $\operatorname{Res}^{A\otimes \operatorname{Cl}_1}_{A}U\otimes \operatorname{Cl}_1\simeq U \oplus \Pi U$. On the other hand, every Type Q module $W$ for $A\otimes \operatorname{Cl}_1$ can be obtained this way: let $\phi\in \operatorname{End}_{A\otimes \operatorname{Cl}_1}(W)$ be an odd endomorphism such that $\phi^2=1$ (such a map exists based on the definition of Type Q modules and after rescaling); let $c$ be the Clifford generator in $ \operatorname{Cl}_1$.  Then $c\phi\in \operatorname{End}_{A}(\operatorname{Res}^{A\otimes \operatorname{Cl}_1}_A W)$ is an even endomorphism with $(c\phi)^2=1$, and $A$ preserves each eigenspace of $c\phi$. In particular, the $1$-eigenspace $W_1$ and $(-1)$-eigenspace $W_{-1}$ are both $A$-modules. It is known that $W_{-1}\simeq \Pi W_{1}$, and  $W\simeq W_{1}\otimes \operatorname{Cl}_1$ as $A\otimes \operatorname{Cl}_1$-modules. Since $\mathcal{H}^{\operatorname{od}}_{p,d}\simeq \mathcal{H}^{\operatorname{ev}}_{p,d}\otimes \operatorname{Cl}_1$, we use this result to prove the following case when $n$ is odd.

\begin{theorem}\label{oddclassify}
Assume $n$ is odd, and furthermore, $n^2(n+1)^2+p^2(p+1)^2$ is not a perfect square. Fix $\lambda$ in Row $d$ of the Bratteli diagram $\Gamma$. Let $\mathcal{W}$ be a module for $\mathcal{H}^{\operatorname{od}}_{p,d}$ of Type Q. If $\mathcal{W}$ is a free module over $\operatorname{Cl}_{d+2}$ with basis $\{v_T\}_{T\in \Gamma^{\lambda}}$, and $z_i.v_T=\kappa_T(i)v_T$ for all $T\in \Gamma^{\lambda}$ and $0\leq i\leq d$, then $W\simeq \mathcal{E}^{\lambda}_f$ for some choice of $f$.
\end{theorem}

\begin{proof}
Let $\{v_T\}_{T \in \Gamma^{\lambda}}$ be the given $\operatorname{Cl}_{d+2}$-basis of $\mathcal{W}$ satisfying the assumptions in Theorem~\ref{classify}. By the short discussion before Lemma~\ref{rescalingvectors} and the fact that $c_Mz_i=z_ic_M$ for all $0\leq i\leq d$, each simultaneous eigenspace for $z_0,\dots,z_d$ is two dimensional, spanned by $c^{\epsilon}v_T$ and $c_Mc^{\epsilon}v_T$ for some $c^{\epsilon}=c_0^{\epsilon_0}\cdots c_d^{\epsilon_d}$. Let $\phi\in \operatorname{End}_{\mathcal{H}^{\operatorname{od}}_{p,d}}(\mathcal{W})$ be an odd endomorphism with $\phi^2=1$. Since $\phi$ commutes with the action of $z_0,\dots,z_d$ and preserves their simultaneous eigenspaces, by parity considerations, $\phi v_T$ is a scalar multiple of $c_Mv_T$. Call this scalar $\gamma$ and we have
\begin{align*}
\phi^2(v_T)=\phi(\gamma c_Mv_T)= -c_M \gamma\phi(v_T)= -c_M^2 \gamma^2 v_T = \gamma^2 v_T,
\end{align*}
therefore $\gamma= \pm 1$.  It follows that 
\begin{align*}
(c_M\phi) (v_T)= \gamma c_M^2 v_T =-\gamma v_T
\end{align*}
lies in the $(-\gamma)$-eigenspace of $(c_M\phi)$, and so does any $c^{\epsilon}v_T$. Therefore the free $\operatorname{Cl}_{d+1}$-module generated by $c_0,\dots,c_d$ and the vectors $\{v_T\}_{\Gamma^{\lambda}}$ is a module for $\mathcal{H}^{\operatorname{ev}}_{p,d}$, and it satisfies the assumptions in Theorem~\ref{classify}. In particular, the generators $x_1$ and $s_i$ act as specified in Lemmas~\ref{x1determined} and \ref{sidetermined}, and the grading is as desired.
\end{proof}

\subsection{Irreducible $\mathcal{H}^p_d$-summands of $L(\alpha)\otimes L(\beta)\otimes V^{\otimes d}$}\label{sec45}
The original motivation for constructing modules in Section~\ref{4.2} is that they naturally occur as irreducible summand of  $L(\alpha)\otimes L(\beta)\otimes V^{\otimes d}$. In particular, let $\mathcal{Z}_d=\operatorname{End}_{\q(n)}(L(\alpha)\otimes L(\beta)\otimes V^{\otimes d})$ be the centralizer algebra for the action of $\q(n)$. In \cite[Proposition 3.5]{Wang}, there is a $\q(n)$-analogue of the double centralizer theorem. In particular, if a $\q(n)$-modules $W$ and $U$ are both of Type $Q$, then $W\otimes U$ is no longer irreducible, but is the direct sum of two isomorphic irreducible Type M $\q(n)$-modules. Denote either of them as $2^{-1}W\otimes U$. Recall that $\mathcal{P}_d(\alpha,\beta)$ is the set of partitions in Row $d$ in the Bratteli graph $\Gamma$, which parametrizes all irreducible $\q(n)$-summands in $L(\alpha)\otimes L(\beta)\otimes V^{\otimes d}$.  Then according to \cite[Proposition 3.5]{Wang}, the module $L(\alpha)\otimes L(\beta)\otimes V^{\otimes d}$ decomposes into irreducible $(\q(n),\mathcal{Z}_d)$-bimodules:
\begin{align}
L(\alpha)\otimes L(\beta)\otimes V^{\otimes d}\simeq \bigoplus_{\lambda\in \mathcal{P}_d(\alpha,\beta)} 2^{\delta(\lambda)} L(\lambda) \otimes \mathcal{L}^{\lambda}. \label{doublecentralizer}
\end{align}
Here, $L(\lambda)$ is an irreducible $\q(n)$-module, $\mathcal{L}^{\lambda}$ is an irreducible $\mathcal{Z}_d$-module. Moreover, $L(\lambda)$ is of Type Q if and only $\mathcal{L}^{\lambda}$ is of Type Q. Also, $\delta(\lambda)=0$ if and only if $\ell(\lambda)=n$ is even; $\delta(\lambda)=-1$ if and only if $\ell(\lambda)=n$ is odd.

When $n$ is even, the action $\rho: \mathcal{H}^{\operatorname{ev}}_{p,d}\to \operatorname{End}_{\q(n)}(L(\alpha)\otimes L(\beta)\otimes V^{\otimes d})$ has image $\rho(\mathcal{H}^{\operatorname{ev}}_{p,d})$ lying in the centralizer $\mathcal{Z}_d$, therefore the irreducible $\mathcal{Z}_d$-module $\mathcal{L}^{\lambda}$ can be restricted to the subalgebra  $\rho(\mathcal{H}^{\operatorname{ev}}_{p,d})$, and yield an $\mathcal{H}^{\operatorname{ev}}_{p,d}$-module, $\operatorname{Res}^{\mathcal{Z}_d}_{\rho(\mathcal{H}^{\operatorname{ev}}_{p,d})}\mathcal{L}^{\lambda}$,  using the action of the image. Similarly, when $n$ is odd, we obtain modules $\operatorname{Res}^{\mathcal{Z}_d}_{\rho(\mathcal{H}^{\operatorname{od}}_{p,d})}\mathcal{L}^{\lambda}$. The goal of this section is to prove that these modules  actually coincide with the calibrated modules we constructed in Section~\ref{4.2}.

We  first mention the following lemmas, as a result of Corollary~\ref{Piericontent} and Corollary~\ref{x1squareacts}. In Lemma~\ref{Pre_5_com} we gave a description of the set $\mathcal{P}_0$ of partitions $\lambda$ such that $L(\lambda)$ is an irreducible summand of $L(\alpha)\otimes L(\beta)$. Also recall the odd map $C\in \operatorname{End}(V)$ in (\ref{Cdefinition}), and the even Casimir tensor $\overline{\Omega}=\Omega (1\otimes C)$ in (\ref{evenomega}).

\begin{lemma}\label{Quo_3_scalar}

1) Let $L(\lambda)$ be an irreducible summand of $L(\alpha)\otimes L(\beta)$. Then $L(\lambda)$ has multiplicity two, and there exists one copy of  $L(\lambda)$ in its isotopic component, on which $\overline{\Omega}$ acts via the nonzeroscalar $\sqrt{mp(m-p)}$ or $-\sqrt{mp(m-p)}$, where $m$ is the number of boxes in the first row of $\lambda$.

2) Let $L(\lambda)$ be an irreducible summand of $L(\mu)\otimes V$, where $b$ is the distinct box in $\lambda$ not in $\mu$. Then $L(\lambda)$ has multiplicity two, and there exists one copy of $L(\lambda)$ in its isotypic component, on which $\overline{\Omega}$ acts via the nonzeroscalar $\sqrt{c(b)(c(b)+1)}$ or $-\sqrt{c(b)(c(b)+1)}$.
\end{lemma}
\begin{proof}
We first prove case 1), and case 2) follows by a similar argument. Recall the structural constants $f^{\lambda}_{\lambda,\mu}$ of the Schur's P-functiton and their relation to the multiplicities $m^{\lambda}_{\lambda,\mu}$ of $\mathfrak{q}(n)$-modules in (\ref{structuralconstant}). By Lemma~\ref{Pre_5_com}, when $\alpha$ is a staircase and $\beta$ is a single row, $f^{\lambda}_{\alpha,\beta}=1$ for any $\lambda\in \mathcal{P}_0$, $\ell(\lambda)=\ell(\alpha)=n$, $\ell(\beta)=1$, therefore (\ref{structuralconstant}) implies that $m^{\lambda}_{\alpha,\beta}=2$ for any $\lambda\in \mathcal{P}_0$. By Corollary~\ref{x1squareacts}, $\overline{\Omega}^2=\overline{\Omega}(1\otimes c) \overline{\Omega} (1\otimes c)=\Omega^2$ acts on $L(\lambda)$ via $mp(m-p)$, which is a strictly increasing function on $m$, the number of boxes in the first row of $\lambda$, and $m$ uniquely determines $\lambda$ by Lemma~\ref{Pre_5_com}. Therefore, $\overline{\Omega}^2$ acts on the isotypic component $L(\lambda)^{\oplus 2}$ by this distinct scalar $a=mp(m-p)$. This scalar is nonzero because $m>p$ as shown in Lemma~\ref{kapparelations}.

Choose a basis of $L(\lambda)^{\oplus 2}$ such that the matrix of  $\overline{\Omega}$ is in Jordan normal form. In particular, there exists one vector $v\in L(\lambda)^{\oplus 2}$, such that $\overline{\Omega} .v=\sqrt{a} v$ or  $\overline{\Omega} .v=-\sqrt{a} v$. Without a loss of generality we assume $\overline{\Omega} .v=\sqrt{a} v$. Since $v\in L(\alpha)\otimes L(\beta)$, and $\overline{\Omega}$ commutes with the action of $\mathfrak{q}(n)$ on $L(\alpha)\otimes L(\beta)$, it follows that $\overline{\Omega}(x.v)=x(\overline{\Omega}.v)=\sqrt{a}x.v$ for all $x\in \mathfrak{q}(n)$, therefore if $W$ is the irreducible $\mathfrak{q}(n)$-module generated by $v$, $W\simeq L(\lambda)$ and $\overline{\Omega}$ acts on $W$ by the scalar $\sqrt{a}$.

\end{proof}

Denote by $\# X$ the cardinality of a set $X$.
\begin{lemma}\label{dimensionsagree}
When $n$ is even, the dimension of $\mathcal{L}^{\lambda}$ is equal to $2^{d+1}\# \Gamma^{\lambda}$. When $n$ is odd, the dimension of $\mathcal{L}^{\lambda}$ is equal to $2^{d+2}\# \Gamma^{\lambda}$. 
\end{lemma}
\begin{proof}
By Lemma~\ref{Quo_3_scalar}, each edge in a path $T$ results in two isomorphic copies of an irreducible summand, therefore the multiplicity of $L(\lambda)$ in $L(\alpha)\otimes L(\beta)\otimes V^{\otimes d}$ is always $2^{d+1}$, corresponding to the $d+1$ edges in $T$. Let $\operatorname{mult}W$ denote the multiplicity of a module $W$ in its isotypic component. When $n$ is even , (\ref{doublecentralizer}) implies that $\operatorname{dim}L(\lambda)\cdot\operatorname{dim}\mathcal{L}^{\lambda}=\operatorname{dim}L(\lambda)\cdot \operatorname{mult} L(\lambda)$, $\operatorname{mult} L(\lambda)=\operatorname{dim}\mathcal{L}^{\lambda}$. When $n$ is odd,  since the module $2^{-1}W\otimes U$ has dimension $2^{-1}\operatorname{dim}(W\otimes U)$, (\ref{doublecentralizer}) implies that $2^{-1}\operatorname{dim}L(\lambda)\cdot\operatorname{dim}\mathcal{L}^{\lambda}=\operatorname{dim}L(\lambda)\cdot \operatorname{mult} L(\lambda)$, and $2\operatorname{mult} L(\lambda)=\operatorname{dim}\mathcal{L}^{\lambda}$ as desired.
\end{proof}

\begin{theorem}\label{punchline}
When $n$ is even, the module $\operatorname{Res}^{\mathcal{Z}_d}_{\rho(\mathcal{H}^{\operatorname{ev}}_{p,d})}\mathcal{L}^{\lambda}$ is isomorphic to $\mathcal{D}^{\lambda}_f$ for some choice of $f$, and is therefore irreducible. Similarly, when $n$ is odd, the module $\operatorname{Res}^{\mathcal{Z}_d}_{\rho(\mathcal{H}^{\operatorname{od})}_{p,d}}\mathcal{L}^{\lambda}$ is isomorphic to $\mathcal{E}^{\lambda}_f$ for some choice of $f$, and is therefore irreducible.
\end{theorem}

\begin{proof}
We first prove the case when $n$ is even. We also omit the restriction notation and regard $\mathcal{L}^{\lambda}$ as a module for $\mathcal{H}^{\operatorname{ev}}_{p,d}$. Given any $T\in\Gamma^{\lambda}$, we claim that there exists a vector $v_T \in \mathcal{L}^{\lambda}$ such that $z_i.v_T= \kappa_T(i)v_T$.

This is because by Lemma~\ref{x1squareacts}, based on the path $T$ and the target $\gamma=T^{(0)}\in \mathcal{P}_0$ of its first edge, one can choose a unique isotypic component $L(\gamma)^{\oplus 2}$ on which $z_0^2$ acts as $\kappa_T(0)^2$. By Lemma~\ref{Quo_3_scalar}, one can further choose a copy of $L(\gamma)$ on which $z_0$ acts as either $\kappa_T(0)$ or $-\kappa_T(0)$. We use this copy of $L(\gamma)$ and perform further tensor products: using the target $\mu=T^{(1)}$ of the second edge in $T$, we then choose the unique isotypic component $L(\mu)^{\oplus 2}$ in $L(\gamma)\otimes V$, on which $z_1^2$ acts as $\kappa_T(1)^2$, and furthermore, choose one copy of $L(\mu)$ on which $z_1$ acts as either $\kappa_T(1)$ or $-\kappa_T(1)$. By successfully choosing the summand using Lemma~\ref{Quo_3_scalar}, we have obtained a sequence $\{T^{(0)},\dots,T^{(d)}=\lambda\}$ of irreducible summands corresponding to each vertex in $T$, and a vector $w_T \in L(\lambda)$ on which $z_i$ acts as $\pm\kappa_T(i)$ for all $0\leq i\leq d$.

To force such a vector $w_T$ in $\mathcal{L}^{\lambda}$ instead of being in $L(\lambda)$, notice that $w_T \in L(\lambda)\otimes \mathcal{L}^{\lambda}$ is in the isotypic component of $\mathcal{L}^{\lambda}$ as an $\mathcal{H}^{\operatorname{ev}}_{p,d}$-module. This isotypic component is isomorphic to $(\mathcal{L}^{\lambda})^{\oplus t}$ for some multiplicity $t$, and let $w_T=w_1+\cdots +w_t$ be the unique decomposition of $w_T$ under this isomorphism. Notice $z_i.w_T=\kappa_T(i)w_1+\cdots+\kappa_T(i)w_t$ is the unique decomposition of $z_i.w_1+\cdots+z_i.w_t$, it follows $z_i.w_j=\kappa_T(i)w_j$ for some $w_j\neq 0$, $w_j\in \mathcal{L}^{\lambda}$. We can further choose $w_j$ to be homogeneous: either its even or odd component $u_T$ is nonzero, and $u_T$ also satisfies $\kappa_T(i).u_T=u_T$ for all $0\leq i\leq d$. By acting on $u_T$ with the proper collection of $c_i$, one can negate certain eigenvalues for $z_i$ and obtain a homogeneous vector $v_T$ such that $z_i.v_T=\kappa_T(i)v_T$ for all $0\leq i\leq d$.

Let $U$ be the free $\operatorname{Cl}_{d+1}$-module generated by all such $\{v_T\}_{T\in \Gamma^{\lambda}}$, and $U$ is a submodule of $\mathcal{L}^{\lambda}$. By Lemma~\ref{dimensionsagree}, $\operatorname{dim}U=\operatorname{dim}\mathcal{L}^{\lambda}$, and $U=\mathcal{L}^{\lambda}$. Since $U$ satisfies the assumptions in Thoerem~\ref{classify}, we have $\mathcal{L}^{\lambda}\simeq \mathcal{D}^{\lambda}_f$ for some choice of $f$.

When $n$ is odd, the argument is similar, using the vectors $\{v_T\}_{T\in \Gamma^{\lambda}}$ and the fact that they generate a free $\operatorname{Cl}_{d+2}$-module whose dimension is equal to that of $\mathcal{L}^{\lambda}$ by Lemma~\ref{dimensionsagree}. The only extra ingredient needed is that when $n$ is odd, $L(\lambda)$ and $\mathcal{L}^{\lambda}$ are both of Type Q, therefore the extra assumption in Theorem~\ref{oddclassify} is satisfied.

\end{proof}

\begin{remark}
Since restriction to $\rho(\mathcal{H}^{\operatorname{ev}}_{p,d})$ or $\rho(\mathcal{H}^{\operatorname{od}}_{p,d})$ preserves the irreducibility of the modules $\mathcal{L}^{\lambda}$, we call the subalgebras $\rho(\mathcal{H}^{\operatorname{ev}}_{p,d})$ of $\mathcal{Z}_d$ (in the case when $n$ is even) and $\rho(\mathcal{H}^{\operatorname{od}}_{p,d})$ of $\mathcal{Z}_d$ (in the case when $n$ is odd) to be \emph{ representation theoretically dense} in the centralizer $\mathcal{Z}_d$, as a weak version of the Schur-Weyl duality.
\end{remark}

\section{Appendix}
The following codes were written in \href{http://magma.maths.usyd.edu.au/magma/}{Magma Computational Algebra System}.

\subsection{Codes for Lemma~\ref{kapparelations}, (\ref{nokappas}), (\ref{anothernokappas})} \label{codeskapparelations}
In the following codes, the variables in the polynomial ring represent following quantities in the proof
\begin{center}
\begin{tabular}{c|c|c|c|c|c}
\hline
m & p & mmp & mp1  & mmpp1 & mmpm1 \\
$\sqrt{m}$ &${p}$&$\sqrt{m-p}$&$\sqrt{m+1}$&$\sqrt{m-p+1}$&$\sqrt{m-p-1}$\\
\hline
\end{tabular} 
\end{center}
and the relations among them are recorded in the ideal ``Ideal''.
\begin{verbatim}
C := IntegerRing();  R<m,p,mmp,mp1,mmpp1,mmpm1> := PolynomialRing(C,6);
Ideal := ideal<R|m^2-1-p-mmpm1^2,  m^2+1-mp1^2, m^2+1-p-mmpp1^2, m^2-p-mmp^2>;
k1k1p:=m*mp1*mmp*mmpp1;  k0k0p:=p*k1k1p;
sumk1square:=m^2*(m^2+1)+(m^2-p)*(m^2-p+1);
kappasquare:=m^2*(p+1)*(m^2-p+1);
Q1:=(sumk1square+2*k1k1p)*(2*kappasquare-2*k0k0p-2*k1k1p);
M1:=p*(p+1)*(2*kappasquare-2*k0k0p+2*k1k1p);
F:= Q1-M1;  F in Ideal;
Q2:=(sumk1square+2*k1k1p)*(2*kappasquare+2*k0k0p-2*k1k1p);
M2:=(p*(p+1)+4*(m^2+1)*(m^2-p)*p/(1+p))*(2*kappasquare+2*k0k0p+2*k1k1p);
G:=Q2-M2;  Numerator(G) in Ideal;
\end{verbatim}

\subsection{Codes for Lemma~\ref{firstnontrivial}, (\ref{code1}) }
\begin{verbatim}
Q:=RationalField(); <m,p>:=PolynomialRing(Q,2);
k0:=m*p*(m-p); k1:=m*(m+1); k:=k0+k1; k3:=k0+p*k1;
x:=(k-2*k1-2*k0*k1*(p-1)/k3)^2-p*(p+1)*k-p*(p+1)*k*k0*k1*(p-1)^2/k3^2; 
Numerator(x);
\end{verbatim}

\subsection{Codes for Lemma~\ref{lastrelation}, (\ref{s1Tiszero})}
In the following codes, the variables in the polynomial ring represent following quantities in the proof
\begin{center}
\begin{tabular}{c|c|c|c|c|c|c|c|c}
\hline
m & p & mmp & mp1 & n & np1 & mmpp1 & mmpm1 & mm1\\
$\sqrt{m}$ &$\sqrt{p}$&$\sqrt{m-p}$&$\sqrt{m+1}$&$\sqrt{n}$&$\sqrt{n+1}$&$\sqrt{m-p+1}$&$\sqrt{m-p-1}$&$\sqrt{m-1}$\\
\hline
\end{tabular} 
\end{center}
and the relations among them are recorded in the ideal ``Ideal''.

When $T=L_1$,
\begin{verbatim}
C := IntegerRing();  R<m,p,mmp,mp1,n,np1,mmpp1,mmpm1,mm1> := PolynomialRing(C,9);
Ideal := ideal<R|m^2-1-mm1^2, m^2-1-p^2-mmpm1^2,  m^2+1-mp1^2,  n^2+1-np1^2, 
m^2+1-p^2-mmpp1^2, m^2-p^2-mmp^2>;
N:=n*np1;  k0:=(m*p*(mmp));  k1:=(n*(np1));  k2:=(m*(mp1));
k0p:=k0;  k1p:=k1;  k0pp:=((mp1)*p*(mmpp1));  k2pp:=((mmp)*(mmpp1));
alpha:=-1/(k1-k2);  beta:=1/(k1+k2);  alpha1:=-1/(k2-k1);  beta1:=1/(k1+k2);
alpha2:=-1/(k2pp-k1);  beta2:=1/(k2pp+k1);  
alphap:=-1/(k1p-k2);  betap:=1/(k1p+k2);
A:=k0-k0p;  B:=k1+k1p;  F:=k0-k0pp;  G:=k2+k2pp;
delta:=(-N^2+k1*k1)*k0/(k0*k0+p^2*k1*k1); 
delta1:=(-N^2+k2*k2)*k0/(k0*k0+p^2*k2*k2);
gamma:=(N^2*p^2+k0*k0)*k1/(k0*k0+p^2*k1*k1); 
 gamma1:=(N^2*p^2+k0*k0)*k2/(k0*k0+p^2*k2*k2);
M:=-(k0*k0+p^2*p^2*k2*k2)*(N^2-k2*k2)*(N^2-k2pp*k2pp)/
((k0*k0+p^2*k2*k2)*(k0*k0+p^2*k2*k2)*((k0-k0pp)*(k0-k0pp)+(k2+k2pp)*(k2+k2pp)));
K:=-(k0*k0+p^2*p^2*k1*k1)*(N^2-k1*k1)*(N^2-k1p*k1p)/
((k0*k0+p^2*k1*k1)*(k0*k0+p^2*k1*k1)*((k0-k0p)*(k0-k0p)+(k1+k1p)*(k1+k1p)));
C0:=-beta*gamma*gamma1-alpha*delta*delta1-delta1*delta1*alpha1
-gamma1*gamma1*beta1-gamma*gamma*beta-delta*delta*alpha-delta*delta1*alpha1
-gamma*gamma1*beta1+M*(-F*F*beta2-G*G*alpha2)
-K*(A*A*betap+B*B*alphap)+gamma1+gamma;
C3:=M*(-F^2*beta2-G^2*alpha2)-K*(A^2*betap+B^2*alphap)-gamma*beta*gamma1
-alpha*delta*delta1-gamma1^2*beta1-delta1^2*alpha1-gamma^2*beta-delta^2*alpha
-delta*delta1*alpha1-gamma*gamma1*beta1+gamma1+gamma;
C2:=M*(-G*F*beta2+G*F*alpha2)-K*(B*A*alphap-A*B*betap)+beta*delta*gamma1
+alpha*gamma*delta1-delta1*beta1*gamma1+gamma1*alpha1*delta1-delta*alpha*gamma
+gamma*beta*delta-delta*gamma1*alpha1-gamma*delta1*beta1+delta1-delta;
C1:=M*(G*F*alpha2-G*F*beta2)-K*(-A*B*betap+B*A*alphap)+alpha*delta*gamma1
-gamma*beta*delta1+delta1*alpha1*gamma1-gamma1*beta1*delta1+delta*beta*gamma
-gamma*alpha*delta-gamma*delta1*alpha1+delta*gamma1*beta1+delta1-delta;
Numerator(C0) in Ideal;  Numerator(C1) in Ideal;
Numerator(C2) in Ideal;  Numerator(C3) in Ideal;
\end{verbatim}
To check for other paths, simply replace the first few lines with the following:

$T=L_2$:
\begin{verbatim}
k0:=(m*p*(mmp)); k1:=(m*(mp1)); k2:=(n*(np1)); k0pp:=k0; 
k2pp:=k2; k0p:=((mp1)*p*(mmpp1)); k1p:=((mmp)*(mmpp1));
\end{verbatim}

$T=L_3$: 
\begin{verbatim}
k0:=((mp1)*p*(mmpp1)); k1:=((mmp)*(mmpp1)); k2:=(n*(np1)); k0pp:=k0; 
k2pp:=k2; k0p:=(m*p*(mmp)); k1p:=(m*(mp1));
\end{verbatim}

$T=L_4$:
\begin{verbatim}
k2:=((mmp)*(mmpp1)); k0:=((mp1)*p*(mmpp1)); k1:=(N);
k0pp:=(m*p*(mmp)); k2pp:=(m*(mp1)); k0p:=k0; k1p:=k1;
\end{verbatim}

$T=T_2$:
\begin{verbatim}
k2:=(m*(mp1)); k0:=((m)*p*(mmp)); k1:=((mmpm1)*(mmp)); k0p:=((mm1)*p*(mmpm1)); 
k1p:=(m*(mm1)); k0pp:=((mp1)*p*(mmpp1)); k2pp:=((mmp)*(mmpp1));
\end{verbatim}

$T=T_3$:
\begin{verbatim}
k2:=((mmp)*(mmpm1)); k0p:=((mp1)*p*(mmpp1)); k1p:=((mmpp1)*(mmp)); 
k0:=(m*p*(mmp)); k1:=(m*(mp1)); k0pp:=((mm1)*p*(mmpm1)); k2pp:=((m)*(mm1));
\end{verbatim}

\subsection{Codes for Lemma~\ref{x1determined}, (\ref{a1iszero}), (\ref{a2iszero}), (\ref{a0iszero}). }
\label{a1iszerocodes}
In the following codes, the variables in the polynomial ring represent following quantities in the proof
\begin{center}
\begin{tabular}{c|c|c|c|c}
\hline
m & p & mmp & mp1  & mmpp1 \\
$\sqrt{m}$ &$\sqrt{p}$&$\sqrt{m-p}$&$\sqrt{m+1}$&$\sqrt{m-p+1}$\\
\hline
\end{tabular} 
\end{center}
and the relations among them are recorded in the ideal ``Ideal''.

\begin{verbatim}
C := IntegerRing();  R<m,p,mmp,mp1,mmpp1> := PolynomialRing(C,5);
Ideal := ideal<R|m^2+1-mp1^2, m^2+1-p^2-mmpp1^2, m^2-p^2-mmp^2>;
k0:=m*p*mmp; k1:=m*mp1; k0p:=mp1*p*mmpp1; k1p:=mmp*mmpp1;
ksquare:=m^2*(p^2+1)*mmpp1^2; k1sum:=m^2*(m^2+1)+mmp^2*mmpp1^2;  
denom:=k0*k1-k0p*k1p;
a1:=-p^2*(p^2+1)+ 2*k0*k1*(2*k1^2/ksquare+2*k0*k1*(ksquare-k1sum)/ksquare/denom)
*(k1sum-p^2*(p^2+1))/denom
+(1-(ksquare-k1sum)/denom*(k1sum-p^2*(p^2+1))/denom)
*(p^2*(p^2+1)+4*(m^2+1)*(m^2-p^2)*p^2/(1+p^2))
+2*k1^2-2*k0*k1*(k1sum-p^2*(p^2+1))/denom-(-p^2*(p^2+1)+2*k1^2)
*(2*k1^2/ksquare+2*k0*k1/ksquare*(ksquare-k1sum)/denom);
Numerator(a1) in Ideal;
a2:=1+(2*k1^2/ksquare+2*k0*k1*(ksquare-k1sum)/denom/ksquare)^2
-1/ksquare*(p^2*(p^2+1)+4*(m^2+1)*(m^2-p^2)*p^2/(p^2+1))
-2*(2*k1^2/ksquare+2*k0*k1*(ksquare-k1sum)/denom/ksquare);
b2:=(2*m^2-p^2+1)^2*(m^2-p^2)*(m^2+1)/(p^2+1)^2/(m^2+1-p^2)/m^2;
Numerator(a2-b2) in Ideal;
a0:=k1^4+(2*k0*k1*(k1sum-p^2*(p^2+1))/(2*denom))^2
-ksquare*((k1sum-p^2*(p^2+1))/(2*denom))^2*(p^2*(p^2+1)
+4*(m^2+1)*(m^2-p^2)*p^2/(p^2+1))
-p^2*(p^2+1)*k1^2-(-p^2*(p^2+1)+2*k1^2)*(2*k0*k1*(k1sum-p^2*(p^2+1))/(2*denom));
b0:=-4*p^4*m^2*(m^2+1)*(m^2-p^2)*(m^2-p^2+1)/(p^2-1)^2;
Numerator(a0-b0) in Ideal;
\end{verbatim}

\end{document}